\documentclass [11pt]{amsart}
\usepackage {amsmath, amssymb, amscd, mathrsfs, hyperref, enumerate, url, graphicx, color}
\usepackage[all, cmtip]{xy}
\usepackage[text={6.5in,9in},centering,letterpaper,dvips]{geometry}
\usepackage{epsfig}
\usepackage{lscape}
\usepackage{fancybox}

\newtheorem {theorem}{Theorem}[section]
\newtheorem {lemma}[theorem]{Lemma}
\newtheorem {proposition}[theorem]{Proposition}

\newtheorem {conjecture}[theorem]{Conjecture}
\newtheorem {definition}[theorem]{Definition}

\theoremstyle{remark}
\newtheorem {remark}[theorem]{Remark}
\newtheorem {example}[theorem]{Example}

\def\Q {{\mathbb{Q}}}
\def\C{\mathbb{C}}
\def\R{\mathbb{R}}
\def\Z {\mathbb{Z}}
\def\Tor{\operatorname{Tor}}
\def\lk{\mathit{\ell k}}
\def\cN{\mathcal N}
\def\cH{\mathcal{H}}
\def\cM{\mathcal{M}}
\def\cW{\mathcal{W}}
\def\unknot{U}
\def\BZ{\mathit{BZ}}
\def\tBZ{\widetilde{\mathit{BZ}}}
\def\Zhat{\widehat{Z}}
\def\ZZhat{\widehat{\mathbf{Z}}}
\def\Zpert{Z^{\operatorname{pert}}}
\def\tPsi {\widetilde{\Psi}}
\def\Spinc {\operatorname{Spin}^c}
\def\CS{\operatorname{CS}}
\def\cA{\mathcal{A}}
\def\su{\mathit{SU}(2)}
\def\sl{\mathit{SL}(2, \C)}
\def\SL{\mathit{SL}(2, \Z)}
\def\RHT{\mathbf{3}^r_{\mathbf{1}}}
\def\LHT{\mathbf{3}^\ell_{\mathbf{1}}}
\def\CL{\mathcal{L}}
\def\del{\partial}
\def\V{\mathcal{V}}
\def\Hab{\operatorname{Hab}}
\def\Vert{\operatorname{Vert}}
\def\eps{\varepsilon}
\def\k{\mathbf{k}}

\def\vn{\vec{n}}
\def\va {\vec{a}}
\def\hA{\widehat{A}}
\def\vell{\vec{\ell}}
\def\vz{\vec{z}}
\def\vk {\vec{k}}
\def\vb{\vec{b}}
\def\tJ{\tilde{J}}
\def\Jhat{\hat{J}}
\def\Fd{F^{\dagger}} 
\def\FKU{\tilde{F}_K}
\def\Spr{\mathbb{S}_{p/r}}
\def\Sp{\mathbb{S}_{p}}
\def\vdelta{\vec{\delta}}
\def\ve{\vec{e}}
\def\se {\operatorname{s.e.}}
\def\vw{\vec{w}}
\def\vm{\vec{m}}
\def\vu{\vec{u}}
\def\tA{\widetilde{A}}
\def\rk{\operatorname{rk }}
\def\mK{m(K)}
\def\Char{\mathcal{X}}
\def\vphi{\vec{\varphi}}
\def\sign{\operatorname{sign}}
\def\lambdabb{\lambda_{\operatorname{BB}}}
\def\lambdasf{\lambda_{\operatorname{SF}}}
\def\Span{\operatorname{Span}}
\def\hGamma{\hat{\Gamma}}
\def\unred{\operatorname{unred}}
\newcommand{\bea}{\begin{eqnarray}}
\newcommand{\eea}{\end{eqnarray}}
\newcommand{\be}{\begin{equation}}
\newcommand{\ee}{\end{equation}}
\begin{document}
\title{A two-variable series for knot complements}

\author[Sergei Gukov]{Sergei Gukov}
\address {California Institute of Technology, Pasadena, CA 91125, USA; and Max-Planck-Institut fur Mathematik, Vivatsgasse 7, D-53111 Bonn, Germany}
\email {gukov@theory.caltech.edu}

\author[Ciprian Manolescu]{Ciprian Manolescu}
\address {Department of Mathematics, Stanford University, 450 Jane Stanford Way\\ 
Stanford, CA 94305}
\email {cm5@stanford.edu}

\begin{abstract}
The physical 3d $\cN = 2$ theory $T[Y]$ was previously used to predict the existence of some $3$-manifold invariants  $\Zhat_{a}(q)$ that take the form of power series with integer coefficients, converging in the unit disk. Their radial limits at the roots of unity should recover the Witten-Reshetikhin-Turaev invariants. In this paper we discuss how, for complements of knots in $S^3$, the analogue of the invariants $\Zhat_{a}(q)$ should be a two-variable series $F_K(x,q)$ obtained by parametric resurgence from the asymptotic expansion of the colored Jones polynomial. The terms in this series should satisfy a recurrence given by the quantum A-polynomial. Furthermore, there is a formula that relates $F_K(x,q)$ to the invariants $\Zhat_{a}(q)$ for Dehn surgeries on the knot. We provide explicit calculations of $F_K(x,q)$ in the case of knots given by negative definite plumbings with an unframed vertex, such as torus knots. We also find numerically the first terms in the series for the figure-eight knot, up to any desired order, and use this to understand $\Zhat_a(q)$ for some hyperbolic $3$-manifolds. 
\end {abstract}

\maketitle
\tableofcontents

\newpage \section{Introduction}
Khovanov homology \cite{Khovanov} is by now a well-known invariant of knots and links in $\R^3$, with a number of striking applications, e.g. to concordance and four-ball genus \cite{Rasmussen, Pic}, contact geometry \cite{Ng} and unknot detection \cite{KMUnknot}. Although its original definition is combinatorial in nature, Khovanov homology has properties similar to those of the Floer homologies coming from gauge theory (instanton, Seiberg-Witten). Since Floer theory gives invariants not just for classical knots, but also for closed $3$-manifolds (and knots in those), it is natural to ask if Khovanov homology can be extended to general $3$-manifolds. This is one of the major open problems in quantum topology.

In fact, the Euler characteristic of Khovanov homology is the Jones polynomial, which does have an extension to $3$-manifolds: the Witten-Reshetikhin-Turaev (WRT) invariant \cite{WittenCS, ReshetikhinTuraev}. Thus, one would like to categorify the WRT invariant. However, this invariant is  only defined at roots of unity, and does not have obvious integrality properties to make it the Euler characteristic of a vector space. One strategy pursued in the mathematical literature is to develop categorification at roots of unity; see \cite{Hopfological}, \cite{Qi}, \cite{EliasQi}. 

Different strategies can be pursued from physics. For example, Witten \cite{Witten5B} proposed a gauge-theoretic interpretation of Khovanov homology, in terms of counts of solutions to certain differential equations: the Kapustin-Witten and Haydys-Witten equations. In principle, one can study the solutions to these equations in settings where $\R^3$ is replaced by another three-manifold; see Taubes \cite{Taubes1, Taubes2} for analytical results in this direction. 

In recent work, Gukov-Putrov-Vafa \cite{GPV} and Gukov-Pei-Putrov-Vafa \cite{GPPV} considered the 6d $(0,2)$ theory (describing the dynamics of M5-branes in M-theory) and its reduction on a three-manifold $Y$. The result is a 3d $\cN=2$ theory, denoted $T[Y]$. The BPS sector of its Hilbert space should give rise to homological invariants of $Y$, denoted $\mathcal{H}^{i,j}_{\text{BPS}} (Y; a)$, similar in structure to Khovanov homology. This picture is related by S-duality to Witten's proposal from \cite{Witten5B}; see \cite[Section 2.10]{GPPV}. Furthermore, a similar set-up, in terms of BPS states, was used in \cite{Gukov:2004hz} to describe Khovanov homology and HOMFLY-PT homology for knots in $\R^3$.

The theory $T[Y]$ depends on the choice of a Lie group $G$ but, for simplicity, in this paper we will limit our discussion to $G=\mathit{SU}(2)$, which is the case corresponding to the Jones polynomial.

A rigorous mathematical definition of the invariants predicted in \cite{GPV}, \cite{GPPV} is yet to be found. In fact, such a definition is lacking even for the Euler characteristic of these invariants, which is a power series 
$$\Zhat_a(Y; q) = \sum_{i,j} (-1)^i q^j \rk \mathcal{H}^{i,j}_{\text{BPS}} (Y; a) \; \in \; 2^{-c} q^{\Delta_a} \Z[[q]],$$
for some $c \in \mathbb{Z}_+$ and $\Delta_a \in \Q$. Apart from the three-manifold $Y$ (which in this paper will always be assumed to be a rational homology sphere), the series depends on the choice of a $\Spinc$ structure $a$ on $Y$, up to conjugation; this can also be thought of as (non-canonically) the choice of an Abelian flat connection on $Y$ or, equivalently, of a value $a \in H_1(Y; \Z)/\Z_2$. Up to multiplication by a factor, the invariant $\Zhat_a(Y; q)$ is a power series in $q$ with integer coefficients, which converges for $q$ in the unit disk. When $Y$ is understood from the context, we write $\Zhat_a(q)$ for $\Zhat_a(Y; q)$.

A general conjecture was formulated in \cite{GPV} which relates $\Zhat_a(q)$ to the WRT invariants of $Y$. Specifically, if we consider a certain linear combination of $\Zhat_a(q)$ over different $a$ and then take the limit as $q$ goes to a root of unity, we should obtain the WRT invariant. See Conjecture~\ref{conj:GPV} below for the precise statement. 

Apart from a few trivial cases ($S^3$, lens spaces, $S^1 \times S^2$), the conjecture 
was also verified mathematically in the case of Brieskorn homology spheres with three singular fibers: this is the older work of Lawrence-Zagier \cite{LawrenceZagier}; see also the work of Hikami \cite{Hikami}, \cite{HikamiDecomposition}. However, the physics literature gives several methods for computing $\Zhat_a(q)$ for other $3$-manifolds: 
\begin{itemize}
\item In principle, for any $3$-manifold, one could construct $\Zhat_a(q)$ from the partition function of Chern-Simons theory using resurgence. This is a general method but challenging to put into practice. Some examples (for Seifert fibered spaces) are presented in \cite{GMP}, \cite{Chung} and \cite{CCFGH};

\item An explicit formula for $\Zhat_a(q)$ for all negative definite plumbed three-manifolds is given in Appendix A of \cite{GPPV};

\item Modularity properties can help compute the series for a manifold $Y$ when we know it for $-Y$, the manifold with the opposite orientation. See \cite{CCFGH}.
\end{itemize}

The purpose of this paper is to propose an analogue of the invariants $\Zhat_a(q)$ for three-manifolds with torus boundary, as well as a formula for gluing along tori. In particular, we are interested in knot complements, and in Dehn surgery (gluing a solid torus). One motivation for this work is to understand the theory $T[Y]$ in the case of knot complements. Another motivation is that, in the long term, one could hope to give mathematical definitions of $\Zhat_a(q)$ and its categorification in terms of surgery presentations. Indeed, this was exactly the strategy that worked for the WRT invariants, in that it enabled Reshetikhin and Turaev to give a mathematical definition of Witten's theory. Every closed oriented three-manifold $Y$ can be obtained from $S^3$ by surgery on a link, and Reshetikhin and Turaev expressed the WRT invariants of $Y$ in terms of invariants associated to the link (the colored Jones polynomials). A similar story exists in Heegaard Floer theory, where there are surgery formulas for knots and links \cite{IntSurg, RatSurg, LinkSurg}. In our case, the analogues of $ \Zhat_a(q)$ for links in $S^3$ are also related to colored Jones polynomials, but in a more subtle fashion. The colored Jones polynomial has been categorified \cite{KhovanovColored, BeliakovaWehrli, CooperKrushkal, Webster}, and this should play a role in categorifying $\Zhat_a(q)$ for three-manifolds.

We start with knot complements $Y = \hat Y \setminus \nu K$ that are represented by plumbing graphs with one distinguished vertex. (Examples of knots with such complements include the algebraic knots in $S^3$, i.e., iterated torus knots.) If the plumbing graph satisfies a certain {\em weakly negative definite} condition, we can imitate the formula for $\Zhat_a(q)$ of closed plumbed manifolds from \cite{GPPV} and obtain an invariant
$$ \Zhat_a(Y; z, n, q).$$
This is a series in two variables $z$ and $q$, which depends on the choice of a relative $\Spinc$ structure $a \in \Spinc(Y, \del Y)$, as well as on another variable $n \in \Z$. Furthermore, we have the following gluing result.

\begin{theorem}
\label{thm:glue}
Let $Y^-$ and $Y^+$ be knot complements represented by weakly negative definite plumbing graphs, and $Y = Y^- \cup_{T^2} Y^+$ the result of gluing them along their common torus boundary. Let also $a^-$ and $a^+$ be relative $\Spinc$ structures on $Y^-$ and $Y^+$, which glue together to a $\Spinc$ structure $a$ on $Y$. Then
\[
\Zhat_a(Y; q) = (-1)^{\tau} q^{\xi} \sum_n \oint_{|z|=1} \frac{dz}{2\pi i z} \Zhat_{a^-}(Y^-; z, n, q)  \Zhat_{a^+}(Y^+; z, n, q),
\]
for some $\tau \in \Z$ and $\xi \in \Q$. (See Section~\ref{sec:gluing} for the exact values of $\tau$ and $\xi$.)
\end{theorem}

For a given $Y$ and $a$, we can view the set of invariants $\Zhat_a(Y; z, n, q)$ as an element in a vector space $\V$ associated to the torus $T^2$. Roughly, $\V$ is the space of functions
$$ (\Z+\tfrac{1}{2}) \times \Z \to \k$$
with certain properties, where $\k$ is a field consisting of Novikov-type series in $q$. Theorem~\ref{thm:glue} can then be interpreted as an aspect of a TQFT for plumbed three-manifolds. 

There is an interesting action of $H_1(T^2) = \Z^2$ on $\V$, which allows us to relate the invariants $\Zhat_a(Y; z, n, q)$ for different $a$ and $n$. For example, when the weakly negative definite plumbed manifold $Y$ is the complement of a knot in an integral homology sphere $\hat Y$, all the different $\Zhat_a(Y; z, n, q)$ can be read from a single two-variable series
$$ F_K(x, q) \;  \in \; 2^{-c} q^{\Delta} \Z[x^{1/2}, x^{-1/2}][q^{-1},q]],$$
which corresponds to choosing $x = z^2$, $a=0$ and $n=0$. Here, $c \in \Z_+$ and $\Delta \in \Q$ are some constants.

By computing the invariants associated to the solid torus, and applying Theorem~\ref{thm:glue}, we can prove a Dehn surgery formula. A formula of this type was already conjectured in \cite{GMP}. For $p/r$ surgery, it involves the ``Laplace transform'' 
\be
\label{eq:Laplace}
\CL_{p/r}^{(a)} \; : \ x^u q^v \ \mapsto \begin{cases} q^{-u^2r/p} \cdot q^v & \text{if } \; ru - a \in p\Z, \\ 0 &\text{otherwise.} \end{cases}
\ee

\begin{theorem}
\label{thm:Dehn}
Let $Y$ be the complement of a knot $K$ in an integer homology $3$-sphere $\hat Y$, and let $Y_{p/r}$ the result of Dehn surgery along $K$ with coefficient $p/r \in \Q^*$. Suppose that both $\hat Y$ and $Y_{p/r}$ are represented by negative definite plumbings. Let $F_K(x,q)$ be the series associated to $K$. Then, the invariants of $Y_{p/r}$ are given by
$$\Zhat_a(Y_{p/r})= \eps  q^{d} \cdot \CL_{p/r}^{(a)} \left[ (x^{\frac{1}{2r}} - x^{-\frac{1}{2r}}) F_K (x,q) \right],$$
for some $\eps \in \{\pm 1\}$ and $d \in \Q$. (See Section~\ref{sec:surgeryformula} for the values of $\eps$ and $d$.)
\end{theorem}

We have an explicit formula for $F_K(x, q)$ in the case of torus knots in $S^3$:
 
\begin{theorem}
\label{thm:TorusKnots}
Let $s, t > 1$ with $\gcd(s,t)=1$. For the positive torus knot $K=T(s,t)$, the series $F_K(x,q)$ is given by
\be
F_K(x,q) =   q^{\frac{(s-1)(t-1)}{2}} \cdot \frac{1}{2}\sum_{m\geq 1} \eps_m \cdot \bigl(x^{\frac{m}{2}} - x^{-\frac{m}{2}} \bigr) q^{\frac{m^2 - (st-s-t)^2}{4st}}
\label{eq:GLTst}
\ee
where
\be
\label{eq:epsem}
\eps_m = \begin{cases}
-1 & \operatorname{if } \ m \equiv st+s+t \ \operatorname{ or } \ st-s-t \!\! \pmod{2st}\\
+1 & \operatorname{if } \ m \equiv st+s-t \ \operatorname{ or }\  st-s+t \!\! \pmod{2st}\\
0 & \operatorname{otherwise.}
\end{cases}
\ee
\end{theorem}

If $K= T(s,t)$, its mirror $\mK$ is the negative torus knot $T(s, -t)$. In this case it makes sense to define 
$$ F_{\mK}(x, q) := F_{K}(x, q^{-1})=
 \frac{1}{2} (\Psi(x,q) - \Psi(x^{-1},q)),$$
where 
\be
 \Psi(x,q) =q^{-\frac{(s-1)(t-1)}{2}} \sum_{m\geq 1} \eps_m \cdot x^{\frac{m}{2}}  q^{-\frac{m^2 - (st-s-t)^2}{4st}}
\label{eq:Psixq}
\ee

The series \eqref{eq:Psixq} can be related to the colored Jones polynomials of negative torus knots as follows. In \cite{GaroufalidisLe}, Garoufalidis and Le defined the {\em stability series} of a sequence of power series $Q_N(q) \in \Z[[q]]$ to be a series of the form 
\be
\label{eq:stability1}
\Phi(x,q)=\sum_j \Phi_j(q)x^j,
\ee
such that
\be
\label{eq:stability2}
 \lim_{N \to \infty} q^{-kN } (Q_N(q) - \sum_{j=0}^k \Phi_j(q) q^{jn})=0, \text{ for all } k \geq 0.
 \ee
This encapsulates the asymptotic behavior of $Q_N(q)$, as $N \to \infty$.

It was proved in \cite{GaroufalidisLe} that, for any alternating knot, its colored Jones polynomials (suitably normalized) admit stability series. This is also true for negative knots (those that can be represented by a diagram with only negative crossings), such as the torus knots $T(s,-t)$.

\begin{theorem}
\label{thm:EvenStability}
Let $s, t > 1$ with $\gcd(s,t)=1$. The stability series for the colored Jones polynomials of the negative torus knot $T(s,-t)$ is $(q^{1/2} - q^{-1/2})^{-1} \cdot \Psi(x,q)$, where $\Psi(x,q)$ is the series from \eqref{eq:Psixq}.
\end{theorem}

This direct connection between $F_K(x,q)$ and the stability series is specific to negative torus knots; for example, it even fails for the positive trefoil. For arbitrary knots, the relation between $F_K(x,q)$ and the colored Jones polynomials $J_n(q)$ is more complicated. What we have to do is to start with Rozansky's asymptotic expansion of $J_n(q)$ from \cite{RozanskyMMR}. This is in terms of the variables $x$ and $\hbar$, where $q=e^{\hbar}$ and $x=q^n$:
\be
\label{eq:rozanski}
J_n (e^{\hbar})
 = \frac{1}{\Delta_K (x)} + \frac{P_1 (x)}{\Delta_K (x)^{3}} \hbar + \frac{P_2 (x)}{\Delta_K (x)^{5}} \hbar^2 + \ldots = \sum_{m=0}^{\infty} \sum_{j=0}^m c_{m,j} \, n^j \hbar^m.
\ee
Here, $\Delta_K(x)$ is the Alexander polynomial of $K$, and the coefficients $c_{m,j}$ are Vassiliev invariants of the knot $K$. The series $F_K(x, q)$, or more precisely its normalized version
$$ f_K(x,q) :=\frac{F_K(x, q)}{x^{1/2} - x^{-1/2}}$$
 should be a repackaging of the invariants $c_{m,j}$, in a similar manner to how the $\Zhat_a(q)$ invariants for closed manifolds are obtained from the WRT invariants via resurgence in \cite{GMP}.

\begin{conjecture}
\label{conj:Borel}
For any knot $K \subset S^3$, the Borel resummation of the double series \eqref{eq:rozanski} gives a knot invariant $f_K(x,q)$ with integer coefficients (up to some monomial):
\be
\label{eq:BorelResum}
J_n (e^{\hbar})
\; = \; \sum_{m=0}^{\infty} \sum_{j=0}^m c_{m,j} \, n^j \hbar^m
\; \;\;\stackrel{\text{Borel resum}}{=\joinrel=\joinrel=}\;\; \; f_K (x,q) 
\ee
where $q = e^{\hbar}$ and $x = e^{n \hbar} = q^n$.
\end{conjecture}

Physically, the series $f_K(x,q)$ is a count of BPS states for the $T[Y]$ theory on the knot complement $Y=S^3 \setminus \nu K$. The exponential change of variables $\hbar \leadsto q = e^{\hbar}$ that ``magically'' leads to integrality from a series with non-integer coefficients is, in fact, rather common in the study of BPS states. A well-known example of this is the relation between the (non-integral) Gromov-Witten invariants and the (integral) Donaldson-Thomas invariants; see \cite{MNOP}.

While the resurgence procedure in Conjecture~\ref{conj:Borel} is general, in practice it is hard to work out. We will explain how it is done in a simple example, that of the right-handed trefoil $T(2,3)$, in which case we recover the corresponding series \eqref{eq:GLTst}. 

A better method to compute the series $f_K(x,q)$ is to take advantage of a recurrence relation. The AJ Conjecture says that that the colored Jones polynomials satisfy a difference equation given by the quantization $\hA$ of the $A$-polynomial of the knot; cf. \cite{Garoufalidis}, \cite{Gukov:2003na}. We conjecture that the same recurrence is satisfied by the series $f_K(x,q)$, with initial conditions inspired by Equation~\eqref{eq:rozanski}:

\begin{conjecture}
\label{conj:Aq}
For any knot $K \subset S^3$, the quantum $A$-polynomial of $K$ annihilates the series $f_K(x,q)$:
\be
\label{eq:hA}
 \hA \; f_K(x,q) \; = \; 0.
\ee
 Furthermore, we have 
 \be
\label{eq:limq1}
 \lim_{q \to 1} f_K(x,q) = \se\Bigl(\frac{1}{\Delta_K(x)}\Bigr),\ee where the ``symmetric expansion'' $\se$ denotes the average of the expansions of the given rational function as $x \to 0$ (as a Laurent power series in $x$) and as $x \to \infty$ (as a Laurent power series in $x^{-1}$). 
\end{conjecture}

Equation~\eqref{eq:hA} sets up a recursion for the coefficients $f_m(q)$ of each power of $x$ in the series $f_K(x,q)$. Equation~\eqref{eq:limq1} is a ``boundary value'' which is supposed to determine $f_K(x,q)$ uniquely in combination with \eqref{eq:hA}. In practice, this is done by first calculating the polynomials $P_k(x)$ from \eqref{eq:rozanski}, and then reading off the coefficients of each power of $x$ in \eqref{eq:rozanski}. These coefficients are power series in $\hbar$, and (through resurgence) we can turn them into series in $q$; in fact, for  simple knots the resurgence procedure is trivial, because we happen to obtain polynomials in $q=e^{\hbar}$. In this fashion we get the first few series $f_m(q)$, which act as initial conditions for the recursion given by \eqref{eq:hA}. This gives an effective procedure to compute the first terms of $f_K(x,q)$ (or, if we prefer, of $F_K(x,q)$) to any desired order of precision.  

Experimentally, for the trefoil, the recursion produces the first terms of the series  \eqref{eq:GLTst}, as expected. We can also obtain the first terms of the series for a hyperbolic knot, the figure-eight knot ${\bf 4_1}$:
\be
\label{eq:41}
 F_{\bf 4_1}(x, q) =  \frac{1}{2} (\Xi(x,q) - \Xi(x^{-1},q)),\ee
where 
$$ \Xi(x,q)=x^{1/2} + 2 x^{3/2} + (q^{-1} + 3+ q) x^{5/2} + (2q^{-2} + 2q^{-1}+ 5 +  2 q + 
    2 q^2) x^{7/2}  +\dots $$

To check that we are on the right track, it is helpful to formulate another conjecture, which is inspired by Theorem~\ref{thm:Dehn}. 

\begin{conjecture}
\label{conj:Dehn}
Let $K \subset S^3$ be a knot, and $S^3_{p/r}(K)$ the result of Dehn surgery on $K$ with coefficient $p/r$. Then, there exist $\eps \in \{\pm 1\}$ and $d \in \Q$ such that
 $$\Zhat_a(Y_{p/r})= \eps q^d \cdot \CL_{p/r}^{(a)} \left[ (x^{\frac{1}{2r}} - x^{-\frac{1}{2r}}) F_K (x,q) \right],$$
 provided that the right hand side of this equation is well-defined. 
\end{conjecture}

The proviso of well-definedness in Conjecture~\ref{conj:Dehn} is due to the fact that we can only apply the Laplace transform to $F_K(x,q)$ for some surgery coefficients. The range of applicability depends on the growth properties of the series.

For the figure-eight knot, Conjecture~\ref{conj:Dehn} can be applied, for example, to the $-1$ surgery, which gives the Brieskorn sphere $c$.  We get 
\be
\label{eq:Zm237}
\Zhat_0(-\Sigma(2,3,7)) = -q^{-1/2} ( 1 + q + q^3 + q^4 + q^5 + 2 q^7 + q^8 + 2 q^9 + q^{10} + 2 q^{11} +\dots )
\ee
Observe that \eqref{eq:Zm237} agrees with the answer that was obtained from modularity analysis in \cite[Equation (7.21)]{CCFGH}. This gives some evidence for Conjectures~\ref{conj:Aq} and \ref{conj:Dehn}.

The same conjectures yield predictions for the invariants $\Zhat_a(q)$ of some closed hyperbolic manifolds. For example, for $-1/2$ surgery on the figure-eight knot we have:
\be
\Zhat_0(S^3_{-1/2}({\bf 4_1})) = -q^{-1/2}( 1 - q + 2 q^3 - 2 q^6 + q^9 + 3 q^{10} + q^{11} - q^{14} - 3 q^{15} - q^{16} + 
 \dots)
 \ee
 As far as we know, these are the first computations of $\Zhat_a(q)$ for hyperbolic manifolds in the literature.

\medskip
{\bf Remark.} In addition to presenting the new results, we have written this paper with the goal to better familiarize the mathematical audience with the invariants $\Zhat_a(q)$. Thus, we include a fair amount of background material (Sections~\ref{sec:WRT} and \ref{sec:plumb}), and present the proofs of some ``folklore'' results, such as the invariance of $\Zhat_a(q)$ for plumbed $3$-manifolds (Proposition~\ref{prop:invariance}). 

\medskip
{\bf Organization of the paper.} In Section~\ref{sec:conventions} we list the notational conventions that we will use in this paper.

 In Section~\ref{sec:WRT} we review some known facts about the WRT invariants and the $q$-series $\Zhat_a(q)$ for closed $3$-manifolds. 
 
 In Section~\ref{sec:plumb} we recall the formula for the $\Zhat_a(q)$ invariants of negative definite plumbed $3$-manifolds, and prove that they are independent of the plumbing presentation; we also explain how the labels $a$ can be identified with $\Spinc$ structures. Moreover, we give a more concrete formula for the invariants of Brieskorn spheres with three singular fibers.
 
  In Section~\ref{sec:plumbedKnot} we describe plumbing representations for manifolds with toroidal boundary (knot complements). 
  
  In Section~\ref{sec:invariant} we define the invariants $\Zhat_a(Y; z, n, q)$ for plumbed knot complements, and in particular the series $F_K(x,q)$. Here we prove Theorems~\ref{thm:glue} and \ref{thm:Dehn}. 
  
  In Section~\ref{sec:torus} we study the invariants for torus knots, proving Theorems~\ref{thm:TorusKnots} and \ref{thm:EvenStability}.
  
 In Section~\ref{sec:resurgence} we discuss the resurgence procedure from Conjecture~\ref{conj:Borel}, and 
 apply it to the trefoil.
 
In Section~\ref{sec:recursion} we set up the recursion for the terms in the series $F_K(x,q)$, as in Conjecture~\ref{conj:Aq}. We show how it works in practice for the trefoil and the figure-eight knot.

Finally, in Section~\ref{sec:categorify} we discuss the physical interpretation of $F_K(x,q)$, and make some speculations about how one can approach the categorification of the series $F_K(x,q)$ and of the invariants $\Zhat_a(q)$.

\medskip
{\bf Acknowledgements.} We would like to thank Steve Boyer, Yoon Seok Chae, Nathan Dunfield, Gerald Dunne, Francesca Ferrari, Stavros Garoufalidis, Matthias Goerner, Sarah Harrison, Slava Krushkal, Thang L\^{e}, Jeremy Lovejoy, Satoshi Nawata, Du Pei, Pavel Putrov, Marko Sto\v{s}i\'{c}, Cumrun Vafa, Ben Webster, Don Zagier and Christian Zickert for helpful conversations.

The first author was supported by the U.S. Department of Energy, Office of Science, Office of High Energy Physics, under Award No. DE-SC0011632, and by the National Science Foundation under Grant No. DMS 1664240. The second author was supported by the National Science Foundation under Grant No. DMS-1708320.

\newpage \section{Conventions}
\label{sec:conventions}
With regard to knots, we denote by $U$ be the unknot, by $\RHT$ and $\LHT$  the right-handed resp. left-handed trefoil, and by ${\bf 4_1}$ the figure-eight knot. We also let $T(s,t)$ be the positive $(s,t)$-torus knot, such that, for example, $T(2,3)= \RHT.$ We let $\mK$ denote the mirror of the knot $K$.

We let $S^3_{p/r}(K)$ denote the result of Dehn surgery along a knot $K \subset S^3$, with coefficient $p/r \in \Q$.

With regard to $3$-manifolds, we will follow the orientation conventions in Saveliev's book \cite{SavelievBook}. In particular, we will orient Brieskorn spheres as boundaries of negative definite plumbings, so that, for example,
$$\Sigma(2,3,5)= S^3_{-1} (\LHT)= - S^3_1(\RHT),$$
$$ \Sigma(2,3,7)= S^3_{-1} (\RHT)= - S^3_1(\LHT)= S^3_1({\bf 4_1}) = - S^3_{-1}({\bf 4_1}).$$
These are the same conventions as in Heegaard Floer theory \cite{AbsGraded, Plumbed}, and opposite to the ``positive Seifert orientation" conventions in other sources.

For lens spaces, we let
$$ L(p,r) = S^3_{-p/r}(U),$$
which is the usual convention but different from the one is \cite{KMOS} or in \cite{OSlens}.

Note that $L(p,r)$ depends only on $r$ mod $p$, and there are symmetries $L(p,r)=L(p, r^{-1})=-L(p,p-r)$.

For example, we have
$$ L(5,1)= S^3_5(\RHT)= - S^3_{-5}(\LHT)= -S^3_5(U) = S^3_{-5}(U) = -L(5,4),$$
$$ L(7,1) = S^3_7(\RHT)=-  S^3_{-7}(\LHT)= -S^3_7(U) = S^3_{-7}(U) = -L(7,6), $$
$$ S^3_6(\RHT)=-  S^3_{-6}(\LHT)= L(3,2)\# L(2,3) = -L(3,1)\# L(2,1).$$

With regard to quantum invariants, if we use the Kauffman bracket with variable $A$, we let $q=A^{4}$ in the definition of the Jones polynomial. Thus, for example,
$$ J_{\RHT}(q)= q^{-1}+q^{-3}-q^{-4}.$$
This is the convention used in most of the literature on WRT invariants of 3-manifolds, for example in \cite{LawrenceZagier}, \cite{BeliakovaChenLe} or \cite{GPV}. However, it is the opposite of the convention in the categorification literature and in most knot theory books, for example in \cite{Lickorish}, and also in \cite{GaroufalidisLe}, where $q$ is replaced by $q^{-1}$, i.e. $q=A^{-4}$.

The colored Jones polynomial of a knot $K$ is denoted $J_{K,n}$ or just $J_n$ (when $K$ is implicit), so that $J_2 = J$ is the usual Jones polynomial. (Some sources call that $J_1$.) The colored Jones polynomial is  {\em normalized} so that for the unknot we have $J_{U,n}=1$.  Thus, $$ \tJ_{n}(q) :=[n] \cdot J_{n}(q)$$ is the {\em unnormalized} or {\em unreduced} version of the colored Jones polynomial, where $$[n]=\frac{q^{n/2}-q^{-n/2}}{q^{1/2}-q^{-1/2}}$$ is the ``quantum integer.''
Our conventions here follow \cite{ArmondDasbach}, but are opposite to those used by Khovanov in \cite{KhovanovColored}, where $J_n$ was the unnormalized Jones polynomial and $\tJ_n$ the normalized (or reduced) version.

The Alexander polynomial of a knot $K$ is denoted $\Delta_K(x)$, and it is normalized so that it symmetric with respect to $x \leftrightarrow x^{-1}$ and
$$\Delta_K (1) = 1, \ \ \ \Delta_{\unknot} (x)=1.$$

We will also make use of the $q$-Pochhammer symbol 
$$ (x; q)_{n} := (1-x)(1-xq)(1-xq^2)\dots(1-xq^{n-1}).$$
We usually just write $(x)_n$ for $(x; q)_n$. We allow for $n = \infty$, in which case the $q$-Pochhammer symbol is an infinite product (which can be expanded into a power series).

As noted in Conjecture~\ref{conj:Aq}, given a rational function $Q(x)$, we define the symmetric expansion $\se(Q(x))$ to be the average of the expansions of $Q(x)$ as $x \to 0$ and as $x \to \infty$. For example,
$$ \se\Bigl (\frac{1}{x+x^{-1}} \Bigr) = \frac{1}{2} \bigl( (x - x^3 + x^5 - \dots ) + (x^{-1} - x^{-3} + x^{-5} - \dots)\bigr).$$

\newpage \section{WRT invariants and $q$-series for closed $3$-manifolds}
\label{sec:WRT}

\subsection{WRT invariants}
Let $Y$ be a closed, connected, oriented three-manifold. We denote by $\cA$ the space of $\su$ connections over $Y$ modulo gauge equivalence. Let $ \CS : \cA \to \R/\Z$ be the Chern-Simons functional. The Chern-Simons path integral is given by
$$Z_{\CS}(Y; k) = \int_{\cA} e^{2\pi i (k-2) \CS(A)} DA.$$
See \cite{WittenCS}. We denote $\hbar = 2\pi i / k$ and set 
$$ q = e^{\hbar} = e^{2\pi i/k}.$$
For example, 
\[
Z_\text{CS}(S^2\times S^1; k)=1,
\qquad Z_\text{CS}(S^3; k)=\sqrt\frac{2}{k}\,\sin\frac{\pi}{k}
=\frac{q^{1/2}-q^{-1/2}}{i\sqrt{2k}}
\]

The Witten-Reshetikhin-Turaev (WRT) invariant is a normalization of $Z_{\CS}$ used in the math literature:
$$\tau_k(Y) = \frac{i \sqrt{2k}}{q^{1/2} - q^{-1/2}} Z_{\CS}(Y; k),$$
so that $ \tau_k(S^3) = 1$. A mathematical definition of $\tau_k$ was given in \cite{ReshetikhinTuraev}. The definition of $\tau$ can be extended to $q$ being any root of unity, giving a map
$$\tau(Y): \{\text{roots of unity} \} \to \C.$$  Strictly speaking, in the definition, one also needs to choose a fourth root of $q$, denoted $A$. This is not necessary when $Y$ is an integral homology sphere. Furthermore, in that case, Habiro \cite{Habiro} showed that one can express $\tau(Y)$ as the evaluation (at any desired root of unity) of an element
$$ \Hab(Y) = \sum_{n \geq 0} a_n(q) \cdot (q)_n \in \widehat{\Z[q]},$$
where $$ (q)_n = (1-q)(1-q^2)\dots(1-q^n)$$ and
$\widehat{\Z[q]}=\varprojlim \Z[q]/((q)_n)$ is called the Habiro ring. (The polynomials $a_n(q)$ are not unique.)

One consequence is that the values of $\tau(Y)$ at any root of unity $q$ are algebraic integers in $\Z[q]$. If we know them at the standard root of unity $\xi = e^{2\pi i/k}$, then we know them at any other primitive $k$th root of unity, by acting with the Galois group $\operatorname{Gal}(\Q(\xi)/\Q)$. See \cite{BeliakovaChenLe}, \cite{BeliakovaBuhlerLe}, \cite{BeliakovaLe} for extensions of these results to other three-manifolds. 

When we reverse the orientation of the manifold, we have
\be
\label{eq:reversal}
 \tau(-Y)(q) = \tau(Y)(q^{-1}).
 \ee

\subsection{The $q$-series}
In \cite{GPV}, \cite{GPPV}, a new set of three-manifold invariants was predicted from physics. They have integrality properties, and are in fact ordinary power series in $q$ (as opposed to elements in the Habiro ring). For rational homology spheres, their relation to the WRT invariants should be as follows.

\begin{conjecture}
\label{conj:GPV}
Let $Y$ be a closed $3$-manifold $Y$ with $b_1(Y)=0$. Let $\Spinc(Y)$ be the set of $\Spinc$ structures on $Y$, with the action of $\Z_2$ by conjugation.  Set
$$ T:=\Spinc(Y)/\Z_2.$$
Then, for every $a \in T$, there exist invariants 
$$\Delta_a \in \Q, \ \ c \in \mathbb{Z}_+, \ \ \Zhat_{a}(q) \in 2^{-c}q^{\Delta_a} \Z[[q]],$$ 
with $\Zhat_a(q) $ converging in the unit disk $\{|q| < 1\}$, such that, for infinitely many $k$, the radial limits $\lim_{q \to e^{2\pi i/k}} \Zhat_a(q)$ exist and can be used to recover the Chern-Simons path integral in the following way:
\be
\label{eq:convergence}
 Z_{\CS}(Y; k)=(i \sqrt{2k})^{-1} \sum_{a,b \in T} e^{2\pi i k \cdot \lk(a,a)} |\cW_b|^{-1} S_{ab}\Zhat_b(q)|_{q \to e^{2\pi i/k}}.
 \ee
Here, the coefficients $S_{ab}$ are given by
\be
\label{eq:Sab}
 S_{ab} = \frac{e^{2\pi i \lk(a,b)} + e^{-2\pi i \lk(a,b)}}{|\cW_a|  \cdot \sqrt{|H_1(Y; \Z)|}},
 \ee
where the group $\cW_x=\operatorname{Stab}_{\Z_2}(x)$ is $\Z_2$ if $x= \bar x$ and is $1$ otherwise.
\end{conjecture}

Conjecture~\ref{conj:GPV} is basically Conjecture 2.1 in \cite{GPPV}, but updated to take into account various developments that have increased our understanding since then:\\
\begin{enumerate}[(a)]
\item Conjecture 2.1 in \cite{GPPV} is stated to hold for {\em any} value of $k$. This should be true, for example, for negative definite plumbings. However, recent insight from the theory of mock modular forms suggests that, for other $3$-manifolds (e.g., positive definite plumbings), the relation to the WRT invariants only holds as stated for the values of $k$ in some congruence classes. At the remaining values, there are certain corrections to the formula; see \cite{CCFGH}.\\

\item We restricted here to rational homology spheres. For manifolds $Y$ with $b_1(Y) > 0$, the analogue of $T$ proposed in \cite{GPV} and \cite{GPPV} was the set of connected components of the space of Abelian flat $\su$ connections on $Y$ (modulo conjugation), which can be identified with $(\Tor H_1(Y; \Z))/\Z_2.$ However, it is now believed that, for some $3$-manifolds, there should also be series $\Zhat_a(q)$ associated to certain non-Abelian flat connections of special type; cf. \cite{CGPS}.\\

\item Even for rational homology spheres, our set of indices $T$ differs from the one proposed in \cite{GPPV}, where it was $H_1(Y; \Z)/\Z_2$ (the space of Abelian flat connections on $Y$). The two sets can be identified, using an affine isomorphism between $\Spinc(Y)$ and $H_1(Y; \Z) \cong H^2(Y; \Z)$ that takes conjugation of $\Spinc$ structures to the symmetry $a \to -a$ on $H_1(Y; \Z)$. However, this identification is not canonical. For an explanation of why it is more natural to consider $\Spinc$ structures (in the case of plumbed manifolds), see Section~\ref{sec:spinc}. Note that $\Spinc$ structures also appear naturally in Seiberg-Witten (or Heegaard Floer) theory, and this theory is related to $T[Y]$; cf. \cite[Section 3]{GPV}.\\

\item
We also made a change of convention compared to \cite[Conjecture 2.1]{GPPV}. There, the factor $|\cW_b|^{-1}$ from \eqref{eq:convergence} did not appear, but was rather incorporated into $\Zhat_b(q)$ itself. We chose this different convention because it makes the gluing and Dehn surgery formulas more elegant, and because it is consistent with the formula for plumbings in \cite[Appendix A]{GPPV}. 
\end{enumerate}
\medskip

Let us make a few more comments about Conjecture~\ref{conj:GPV}.

\begin{remark}
The linking numbers $\lk(a, b)$ in \eqref{eq:convergence}, \eqref{eq:Sab} appeared in \cite[Conjecture 2.1]{GPPV} as the usual linking numbers on $H_1(Y; \Z)$, cf. \cite[Equation (2.1)]{GPPV}. Here, we define them on $\Spinc$ structures by using a $\Z_2$-equivariant identification between $\Spinc$ structures and $H_1(Y; \Z)$ as in point (c) above. The linking numbers are independent of this identification. 
\end{remark}

\begin{remark}
After an identification as in (c), the numerator in the formula  \eqref{eq:Sab} for $S_{ab}$ admits a geometric interpretation: It is the trace of the holonomy of the flat connection labeled by $a$ along a $1$-cycle representing the homology class $b$. 
\end{remark}

\begin{remark} \label{rem:ZHS}
The simplest version of Conjecture~\ref{conj:GPV} is for $Y$ an integral homology $3$-sphere. Then, $T=\{0\}$, $S_{00}=1$, and the conjecture predicts the existence of a single series $\Zhat_0(q)$ that converges (up to a power of $q$) to $$2(q^{1/2} - q^{-1/2})\tau_k(Y) = 2i \sqrt{2k} \cdot Z_{\CS}(Y; k)$$ as $q \to e^{2\pi i/k}$.
\end{remark}

\begin{remark}
One can conjecture that the limit on the right hand side of \eqref{eq:convergence} also exists at primitive $k^{\operatorname{th}}$ roots of unity different from $e^{2\pi i/k}$. The analogue of \eqref{eq:convergence} should hold for many such values, with the left hand side replaced by the (suitably normalized) WRT invariant at that root of unity. \end{remark}

\begin{remark}
Equation~\eqref{eq:convergence} does not characterize the series $\Zhat_a(q)$ uniquely. Indeed, consider Euler's pentagonal series
\be
\label{eq:Euler}
 (q)_{\infty} = \prod_{n=1}^{\infty} (1-q^n) = \sum_{m \in \Z} (-1)^m q^{m(3m-1)/2}.
 \ee
This series converges in the unit disk $|q| < 1$, and the result approaches $0$ near each root of unity.
Thus, to every $\Zhat_a(q)$ we could add $(q)_{\infty}$ (times any polynomial in $q$, if we prefer) and get a new series that has the same limits at the roots of unity. However, physics predicts that there is a particular series $\Zhat_a(q)$, among all those that satisfy \eqref{eq:convergence}, which is the one that should be categorified; cf. Remark~\ref{rem:hblock} below.
\end{remark}

\begin{remark} \label{rem:hblock}
The name ``homological block'' is used in \cite{GPV}, \cite{GPPV} to refer to the series $\Zhat_a(q)$. This is due to the fact that we expect $\Zhat_a(q)$ to have a homological refinement, a bi-graded vector space $\mathcal{H}^{*,*}_{\text{BPS}}(Y)$ whose Euler characteristic is $\Zhat_a(q)$. We can think of this as an extension of Khovanov homology to three-manifolds. We denote the Poincar\'e polynomial of $\mathcal{H}^{*,*}_{\text{BPS}}(Y)$ by
$\ZZhat_a(q, t)$
so that $$\ZZhat_a(q, -1) = \Zhat_a(q). $$
\end{remark}

As a simple example, when $Y=S^3$, the relevant series is $$\Zhat_0(q) = q^{-1/2}(-2+2q),$$
with $\Delta_0 = 1/2$ and $c=1$. Its categorification is more complicated; according to \cite[Equation (6.80)]{GPV}, its Poincar\'e series is
\begin{align}
 \ZZhat_0 (S^3; q, t)&= -2q^{-1/2}\frac{(-tq)_{\infty}}{(t^2q^2)_{\infty}} \label{ZS3qtq1} \\
&= -2q^{-1/2}\bigl(1+tq+ (t+t^2)q^2 + (t+2t^2+t^3)q^3 + (t+2t^2+2t^3+t^4)q^4 + \dots \bigr) \notag
\end{align}

\begin{remark}
\label{rem:unred}
There are also ``unreduced'' versions of the invariants $\Zhat_a$ and $\ZZhat_a$. For example, for $Y=S^3$, according to \cite[Equation (6.49)]{GPV} and \cite[Equation (3.6)]{GPPV}, we have
$$ \Zhat^{\unred}_0(S^3; q) = \frac{\Zhat_0(S^3; q)}{(q)_{\infty}}= \frac{-2q^{-1/2}}{(q^2)_{\infty}}= -2q^{1/2}(1+q^2+q^3+2q^4+\dots)$$
and
$$
\ZZhat^{\unred}_0(S^3; q,t)= \frac{\ZZhat_0(S^3; q,t)}{(-tq)_{\infty}}= \frac{-2q^{-1/2}}{(t^2q^2)_{\infty}}= -2q^{-1/2}(1+t^2q^2 + t^2q^3 + (t^2+t^4)q^4 + \dots)
$$
\end{remark}

\subsection{Surgery and Laplace transforms}
The definition of the WRT invariant $\tau_k(Y)$ in  \cite{ReshetikhinTuraev} is  based on representing a three-manifold by integral surgery on a link in $S^3$. In particular, for a knot $K \subset S^3$ and $p \in \Z$ non-zero, the formula is
\begin{equation}
\label{eq:tausurgery}
 \tau_k(S^3_p(K))= \frac{\sum_{n=1}^{k-1} [n]^2 q^{\frac{p(n^2-1)}{4}}J_{K, n}(q)}{\sum_{n=1}^{k-1} [n]^2 q^{\sign(p)\frac{(n^2-1)}{4}}}
\end{equation}
where $J_{K,n}(q)$ is the colored Jones polynomial and $\text{sign}(p)\in \{\pm 1\}$.

In \cite{BeliakovaBlanchetLe} and \cite{BeliakovaLe}, Beliakova, Blanchet and L\^e used this formula to express the Habiro series of certain $3$-manifolds in terms of ``Laplace transforms.''  As mentioned in the Introduction, the Laplace transform $\CL_{p/r}^{(a)}$ takes a series in two variables $x, q$ into a series in a single variable $q$, by the formula~\eqref{eq:Laplace}.

In our situation, Habiro \cite{HabiroKnots}, \cite{Habiro} showed that there are Laurent polynomials $C_m(q) \in \Z[q, q^{-1}]$ such that
\begin{equation}
\label{eq:cyclotomic}
 J_{K, n}(q)= \sum_{m \geq 0} C_m(q) (q^{n+1})_m (q^{1-n})_m.
 \end{equation}
This is called the cyclotomic expansion of the colored Jones polynomial. If we set $x = q^n$ and write 
\begin{equation}
\label{eq:Fhabiro}
 C_K(x,q)= \sum_{m\geq 0} C_m(q) (qx)_m (qx^{-1})_m
\end{equation}
and plug \eqref{eq:cyclotomic} into \eqref{eq:tausurgery}, we find that $ \tau_k(S^3_p(K))$ is a linear combination of expressions of the form $\CL_p^{(a)} ((x+x^{-1}-2)\cdot C_K(x,q)).$ This can be generalized to rational surgeries $S^3_{p/r}(K)$, using the formula for the WRT invariants of rational surgeries in \cite{BeliakovaLe}. What we get is that the Laplace transforms
$$ \CL_{p/r}^{(a)} \left[ (x^{\frac{1}{2}} - x^{-\frac{1}{2}}) (x^{\frac{1}{2r}} - x^{-\frac{1}{2r}}) C_K(x,q)\right]$$
can be combined to give the Habiro expansion $\Hab(S^3_{p/q}(K))$.

Inspired by this, in \cite[Section 5.3]{GMP}, Gukov, Mari{\~{n}}o and Putrov conjectured that the $q$-series of $Y$ can be obtained by a Laplace transform from a two-variable series associated to the knot: 
\be
\Zhat_a \big( S^3_{p/r} (K); q \big)
\; = \; \CL_{p/r}^{(a)} \left[ (x^{\frac{1}{2}} - x^{-\frac{1}{2}}) (x^{\frac{1}{2r}} - x^{-\frac{1}{2r}}) f_K (x,q) \right].
\label{prLaplace2}
\ee
In contrast to what was claimed in \cite{GMP}, our new understanding is that $f_K$ is not directly related to the cyclotomic expansion (just as the $q$-series $\Zhat_a$ is not directly related to the Habiro series). In this paper we will explore different ways of constructing the series $f_K(x,q)$.

Note that $C_K(x,q)$ is not a true power series in $q$ and $x$, because in the summation in \eqref{eq:Fhabiro} there may be infinitely contributions from the same monomial $x^u \cdot q^v$. Rather, it is a cyclotomic expansion, and by applying the Laplace transform to each of the summands and then summing up, we obtain an element of the Habiro ring. On the other hand, our new object $f_K(x,q)$ will be a Laurent power series, and by applying the Laplace transform we will obtain $\Zhat_a(q)$, which is a true power series (up to a monomial). Roughly, the cyclotomic expansion and the Habiro series are tailored to $q$ being a root of unity, whereas $f_K(x,q)$ and $\Zhat_a(q)$ deal with $|q| <1$.

\subsection{An example: the Poincar\'e sphere} The Poincar\'e sphere is obtained by $(-1)$ surgery on the left-handed trefoil:
\be
P = \Sigma (2,3,5) = S^3_{-1} (\LHT)
\label{trefoil235}
\ee
The cyclotomic expansion of the colored Jones polynomial for $\LHT$ is
\be
\label{eq:LHTcyclo}
J_{\LHT,n}(q)= \sum_{m \geq 0} q^m (q^{n+1})_m (q^{1-n})_m.
\ee
Applying the surgery formula \eqref{eq:tausurgery}, we find that the WRT invariant of $P$ is
\be
\label{eq:taup}
\tau_k(Y)= \frac{1}{1-q} \sum_{m=1}^{\infty} q^{m-1} (q^{m})_{m}.
\ee
The above expression can be evaluated at any root of unity, and in fact it is the Habiro series for $P$. 

The $q$-series $\Zhat_0$ for $P$ was computed in \cite[Section 3.4]{GMP} using resurgence, and can also be deduced from the general case of negative definite plumbings (since $P$ is the boundary of the $-E_8$ plumbing). We have
\be
\label{eq:zhp}
\Zhat_0(P; q)= q^{-3/2}(2-A(q)),
\ee
where 
\be
A(q)  =  \sum_{n=0}^{\infty} q^n (q^n)_n= \sum_{n=1}^{\infty} \chi_+(n) q^{(n^2-1)/120} =1+ q + q^3 + q^7 - q^8 - q^{14} -q^{20}-q^{29}+q^{31}+\dots
\label{chichirel}
\ee
where
\be
\chi_+(n) = \begin{cases} 1 & \text{if } n \equiv 1, 11, 19, 29  \pmod{60},\\
-1 & \text{if } n \equiv 31, 41, 49, 59 \pmod{60},\\
0 & \text{otherwise.}\end{cases}
\ee
Thus, we have
$$\Zhat_0(P; q)=q^{-3/2}(1-q-q^3-q^7+q^8+q^{14}+q^{20}+q^{29}-q^{31}+\cdots)$$

This expression (which is easily seen to converge for $|q| < 1$) had already appeared in the math literature, in the older work of Lawrence and Zagier, where they proved the following.
\begin{theorem}[Lawrence-Zagier \cite{LawrenceZagier}] \label{thm:LZ} For every root of unity $\xi$, the radial limit of $\Zhat_0(P; q)$ as $q \to \xi$ gives the (renormalized) WRT invariant of $P$:
$$ \lim_{q \to \xi} \Zhat_0(P; q) = 2(q^{1/2}-q^{-1/2})\tau(P)(\xi).$$
\end{theorem}

This shows that Conjecture~\ref{conj:GPV} is satisfied for the Poincar\'e sphere; cf. Remark~\ref{rem:ZHS}.

Interestingly, note that $A(q)$, when written as the sum 
$A(q) = \sum_{n=0}^{\infty} q^n (q^n)_n$, can also be viewed as an element of the Habiro ring, and thus evaluated at roots of unity. As we take radial limits towards a root of unity $\xi$, in view of \eqref{eq:taup}, \eqref{eq:zhp} and Theorem~\ref{thm:LZ}, we get
$$ \lim_{q \to \xi} A(q) = 2A(\xi).$$
This relation is specific to the Poincar\'e sphere. In general, for an arbitrary $3$-manifold, we cannot interpret its  Habiro series as an actual power series in $q$.

The expression $A(q)$ is the false theta function associated to the following Ramanujan mock modular form of order 5:
$$\chi_0(q)= \sum_{n=0}^{\infty} \frac{q^n}{(q^{n+1})_n}= 1+ q+q^2+ 2q^3+ q^4 + 3q^5+2q^6+3q^7+3q^8+5q^9+\dots$$
The $q$-series $\Zhat_0(q)$ for $-P$ (the Poincar\'e sphere with the opposite orientation) is in fact
$$ \Zhat_0(-P; q) = q^{3/2}(2-\chi_0(q)).$$

\begin{remark}
For any three-manifold $Y$, recall from \eqref{eq:reversal} that the WRT invariant of $-Y$ is obtained from the one for $Y$ by taking $q \to q^{-1}$. The series $\Zhat_a(Y)$ and $\Zhat_a(-Y)$ are not as easily related. Rather, one needs to find an analytic continuation of $ \Zhat_a(Y)$ outside the unit disk, and then take $q \to q^{-1}$. This is how one can obtain $A(q)$ from $\chi_0(q)$, and vice versa. For other examples of such relations (using modularity properties of the respective series), we refer to \cite{CCFGH}.
\end{remark}

\newpage \section{Plumbed manifolds}
\label{sec:plumb}

\subsection{Plumbings} \label{sec:PlumbedDef}
Let $\Gamma$ be a weighted graph, that is, a graph together with the data of integer weights associated to vertices. Throughout this paper we will always assume (for simplicity) that $\Gamma$ is a tree. 

If $\Vert$ is the set of vertices of $\Gamma$, we let $m_v \in \Z$ be the weight of a vertex $v \in \Vert$, and $\text{deg}(v)$ be the degree of $v$, that is, the number of edges meeting at that vertex. Let $s$ be the cardinality of $\Vert$. Consider the $s \times s$ matrix $M$ given by
\begin{equation}
 M_{v_1,v_2}=\left\{
\begin{array}{ll}
 1,& v_1,v_2\text{ are connected}, \\
 m_v, & v_1=v_2=v, \\
0, & \text{otherwise}.
\end{array}
\right.\qquad v_i \in \Vert.
\label{linking}
\end{equation}

From $\Gamma$ we can construct a framed link $L(\Gamma)$ made of one unknot component for each vertex $v \in \Vert$, with framing $m_v$, and with the components corresponding to $v_1$ and $v_2$ chained together whenever we have an edge from $v_1$ and $v_2$. (See Figure~\ref{fig:Sigma237} for an example.) We let $W(\Gamma)$ be the four-dimensional manifold obtained by attaching two-handles to $B^4$ along $L(\Gamma)$ or, equivalently, by plumbing together disk bundles over $S^2$ with Euler numbers $m_v$. Let $Y=Y(\Gamma)$ be the boundary of $W(\Gamma)$. This is a closed, oriented three-manifold whose first homology is
$$ H_1(Y) \cong \Z^s / M \Z^s.$$
The manifolds $Y$ obtained this way are always graph manifolds, that is, made of Seifert fibered pieces glued along tori (in the JSJ decomposition). We will mostly be interested in the case where $M$ is nondegenerate, so that $Y$ is a rational homology sphere. When $M$ is negative definite, we will say that $Y$ is a {\em negative definite plumbed three-manifold}.

\begin {figure}
\begin {center}
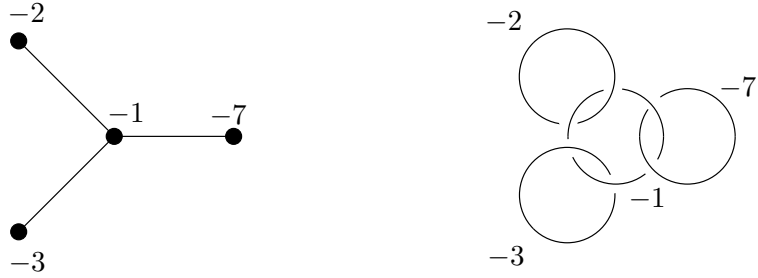
\caption {On the left we show a plumbing graph representing the Brieskorn sphere $\Sigma(2,3,7)$. The corresponding surgery diagram is on the right.}
\label{fig:Sigma237}
\end {center}
\end {figure}

There is a set of {\em Neumann moves} (sometimes also called ``3d Kirby moves'') on weighted trees that change the graph but not the manifold $Y(\Gamma)$: see Figure~\ref{fig:kirby}. In \cite[Theorem 3.2]{NeumannCalculus}, Neumann showed that two plumbed trees represent the same three-manifold if and only if they are related by a sequence of these moves.

\begin {figure}
\begin {center}
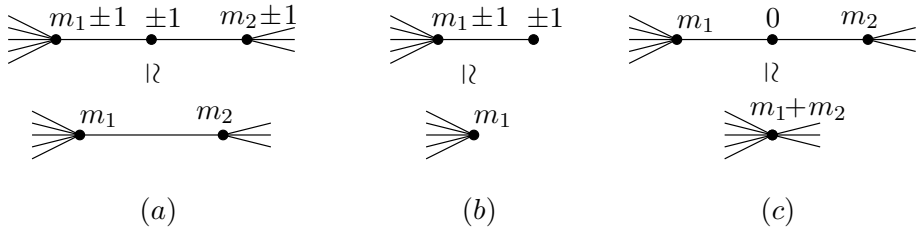
\caption {Moves on plumbing trees that preserve the $3$-manifold.}
\label{fig:kirby}
\end {center}
\end {figure}

One large class of plumbed three-manifolds is obtained as follows. Consider the Seifert bundle over $S^2$ with orbifold Euler number $e \in \Q$, and Seifert invariants
$$ (b_1, a_1), \dots, (b_n, a_n),$$
with $\gcd(a_i, b_i) = 1$. 
This manifold is usually denoted
$$M \left( b ; \frac{a_1}{b_1} , \ldots , \frac{a_n}{b_n} \right),$$
where
$$ b=e- \sum_{i=1}^n \frac{a_i}{b_i} \in \Z.$$

With respect to changing orientations, we have
$$ -M \left( b ; \frac{a_1}{b_1} , \ldots , \frac{a_n}{b_n} \right)= M\left( -b ; -\frac{a_1}{b_1} , \ldots , -\frac{a_n}{b_n} \right).$$

Write $b_i/a_i$ as a continued fraction
\be
\label{eq:pf}
\frac{b_i}{a_i} = k_1^i - \cfrac{1}{ k_2^i - \cfrac{1} {\ddots \, - \cfrac{1}{ k_{s_i}^i}}}
\ee
Let $\Gamma$ be the star-shaped graph with $n$ arms, such that the decoration of the central vertex is $b$ 
and along the $i^{\operatorname{th}}$ arm we see the decorations $-k_1^i, -k_2^i, \dots, -k_{s_i}^i$, in that order starting from the central vertex. Then $Y(\Gamma)$ is the given Seifert fibration. We will mostly focus on negative definite plumbings, so we will usually take $b < 0$ and $0 < a_i < b_i$.

In some cases (for example, if $n \leq 2$, the manifold $M(b; a_1/b_1, \dots, a_n/b_n)$ is a lens space, $S^1 \times S^2$, or a connected sum. We call such cases {\em special}, and the other Seifert manifolds {\em generic}. Lens spaces can be represented by both negative definite and positive definite plumbings. For generic Seifert manifolds, the orbifold Euler number $e$ determines whether the Seifert fibration admits such plumbings.

\begin{theorem}[Neumann-Reymond  \cite{NeumannRaymond}]
\label{thm:euler}
Let $M$ be a generic Seifert bundle over $S^2$, with orbifold number $e$. Then, $M$ can be represented by a positive definite plumbing iff $e > 0$, and by a negative-definite one iff $e < 0$.
\end{theorem}

Finally, we note that if
\begin{equation}
\label{eq:e}
 b_1 b_2 \dots b_n e = -1,
 \end{equation}
then $Y$ is an integral homology sphere, denoted $\Sigma(b_1, \dots, b_n)$. The values of $a_i$ are uniquely determined by the values of $b_i$, the fact that $b \in \Z$, and the condition \eqref{eq:e}. In fact, any Seifert fibered integral homology sphere is of the form $\pm \Sigma(b_1, \dots, b_n)$ for some $b_i$.

\subsection{Identification of $\Spinc$ structures}
\label{sec:idSpinc}
A $\Spinc$ structure on an oriented $n$-dimensional manifold $Y$ is a lift of the structure group of its tangent bundle from $\operatorname{SO}(n)$  to $$\Spinc(n) = \operatorname{Spin}(n) \times_{\Z/2} S^1.$$ When they exist (which they always do for $n \leq 4$), $\Spinc$ structures form an affine space over $H^2(Y; \Z)$. We are interested in describing the space $\Spinc(Y)$ of such structures in the case where $Y$ is $Y(\Gamma)$ for a plumbing tree $\Gamma$. Such a description has already appeared in the literature on  Heegaard Floer homology \cite{Plumbed}, but we will give a slightly different description, tailored to our purposes.

Let us first consider $\Spinc$ structures on the four-manifold $W=W(\Gamma)$ with boundary $Y$. Note that $H_1(W)=0$ and $H_2(W) \cong \Z^s$, with a basis given by the $2$-spheres associated to the vertices of $\Gamma$. $\Spinc$ structures on $W$ can be canonically identified (via the first Chern class $c_1$) with characteristic vectors $K \in H^2(W)$, that is, those such that
$$ K(v) \equiv \langle v, v \rangle, \ \ \text{ for all } v \in H_2(W).$$
If we identify $H^2(W) \cong \Z^s$ using Poincar\'e duality, we see that characteristic vectors are those $K \in \Z^s$ such that
$$ K \equiv \vm\!\! \pmod{2\Z^s},$$
where $\vm$ is the vector made of the weights $m_v$ for $v \in \Vert$. Therefore, we have a natural identification
\be
\label{eq:spincm}
 \Spinc(W) \cong 2\Z^s + \vm.
 \ee
This identification has the nice property that conjugation of $\Spinc$ structures corresponds to the involution $a \leftrightarrow -a$ on the right-hand side.

For the three-manifold $Y$, Poincar\'e duality and the long exact sequence
$$ H^2(W, \del W) \to H^2(W) \to H^2(Y) \to 0$$
gives the identification $H_1(Y) \cong H^2(Y) \cong \Z^s/M\Z^s$. Further, we see that the map 
$$\Spinc(W) \to \Spinc(Y)$$
is surjective and, using \eqref{eq:spincm}, we obtain a natural identification
\be
\label{eq:s1}
\Spinc(Y) \cong (2\Z^s + \vm)/(2M\Z^s),
\ee
again taking the conjugation symmetry to $a \leftrightarrow -a$.

We claim that there is also a natural identification
\be
\label{eq:s2}
\Spinc(Y) \cong (2\Z^s + \vdelta)/(2M\Z^s),
\ee
where $\delta \in \Z^s$ is the vector given by the degrees (valences) of the vertices of $\Gamma$, that is
\be
\label{eq:vdelta}
\vdelta = (\delta_v)_{v \in \Vert}, \ \ \ \delta_v = \text{deg}(v).
\ee

Observe that 
$$\vm + \vdelta = M \vu,$$
where $\vu = (1, 1, \dots, 1).$ To go from \eqref{eq:s1} to \eqref{eq:s2}, we will use the map
\begin{align}
\label{eq:mdelta}
\phi: (2\Z^s + \vm)/(2M\Z^s)  \ & \xrightarrow{\cong} (2\Z^s + \vdelta)/(2M\Z^s),\\ [\vell] &\to [\vell - M \vu ] \notag
\end{align}
taking $[\vm]$ to $[-\vdelta].$ Note that \eqref{eq:mdelta} commutes with the conjugation symmetry $a \leftrightarrow -a$.

To be justified in calling the resulting identification \eqref{eq:s2} {\em natural}, we should check that it does not depend on how we represent the manifold $Y$ by a plumbing tree. Specifically, for each of the Neumann moves from Figure~\ref{fig:kirby}, we should describe some (reasonably simple) isomorphisms between the spaces $(2\Z^s + \vdelta)/(2M\Z^s)$ before and after the move. Similar isomorphisms already exist between the spaces $(2\Z^s + \vm)/(2M\Z^s)$ before and after the move, due to their identifications  \eqref{eq:s1} with $\Spinc(Y)$. These before / after isomorphisms should commute with \eqref{eq:mdelta}. Further, all our isomorphisms should commute with the conjugation symmetry.

Let us explain how this is done for the Kirby move from Figure~\ref{fig:kirby} (a), with the signs being $-1$. We use $M$ to denote the adjacency matrix for the bottom graph in that figure, and $M'$ to denote the one for the top graph. Similarly, we use $\vm$ and $\vdelta$ for the bottom graph, and $\vm'$ and $\vdelta'$ for the top. For the bottom graph, we write a vector $\vell \in \Z^s$ as a concatenation $$\vell =(\vell_1, \vell_2),$$ where $\vell_1$ corresponds to the part of the graph on the left of the edge where we do the blow-up (including the vertex labeled $m_1$), and $\vell_2$ corresponds to the part on the right (including the vertex labeled $m_2$). From $\vell$ we can construct a vector for the top graph of the form
$$ T\vell := (\vell_1, 0, \vell_2) \in \Z^{s+1},$$
with the $0$ entry corresponding to the newly introduced vertex. Let also
$$ \ve_0 = (0, \dots, 0, 0, 1, 0, 0, \dots, 0)$$
be the vector with the $1$ entry in the position of the new vertex. Note that
$$ M' \ve_0 = (0, \dots, 0, 1, -1, 1, 0, \dots, 0)$$
with the nonzero entries being for the three vertices shown in the figure.

Note that
$$ \vm' = T\vm + (0, \dots, 0, -1, -1, -1, 0, \dots, 0) = T\vm - M' \ve_0 - 2\ve_0$$
and
$$ \vdelta' = T\vdelta + (0, \dots, 0, 0, 2, 0, 0, \dots, 0) = T\vdelta + 2 \ve_0.$$

With this in mind, the before / after isomorphisms are given by
\begin{align}
\label{eq:mm}
\psi^m: (2\Z^s + \vm)/(2M\Z^s)  \ & \xrightarrow{\cong} (2\Z^{s+1} + \vm')/(2M'\Z^{s+1}),\\ [\vell] &\to [T\vell + M'\ve_0] \notag
\end{align}
and
\begin{align}
\label{eq:deltadelta}
\psi^{\delta} : (2\Z^s + \vdelta)/(2M\Z^s)  \ & \xrightarrow{\cong} (2\Z^{s+1} + \vdelta')/(2M'\Z^{s+1}),\\ [\vell] &\to [T\vell]. \notag
\end{align}
We claim that these commute with the identification $\phi$ from \eqref{eq:mdelta}, i.e.,
\be
\label{eq:psph}
 \psi^{\delta} \circ \phi = \phi \circ \psi^m.
 \ee
Since all our maps are affine, it suffices to check this when evaluated on a single element, say $[\vm].$ We have
$$(\psi^{\delta} \circ \phi)([\vm]) = \psi^{\delta}([-\vdelta]) = [-T\vdelta]=[2\ve_0 - \vdelta']$$
and
$$(\phi \circ \psi^{m})([\vm]) = \phi([T\vm + M'\ve_0]) = \phi([\vm' + M'\ve_0 + 2\ve_0 + M'\ve_0]) = \phi([\vm' +  2\ve_0]) = [2\ve_0 - \vdelta'],$$
which proves \eqref{eq:psph}.

The other Neumann moves can be treated similarly.

The isomorphisms between the different spaces $(2\Z^s + \vdelta)/(2M\Z^s)$ will appear again later, in the proof of Proposition~\ref{prop:invariance}.

\subsection{The $q$-series} \label{sec:PlumbedSeries} Let us review the formula for the $q$-series $\Zhat_a (q)$ for the closed $3$-manifolds given by negative definite plumbings along trees. This formula was proposed in \cite[Appendix A]{GPPV}, by applying Gauss reciprocity and a regularization procedure to the WRT invariants of those manifolds. It was shown in \cite{CCFGH} that the same formula also works for some graphs that are not negative definite; see Definition~\ref{def:wnd} below.

We keep the notation from Subsections~\ref{sec:PlumbedDef} and ~\ref{sec:idSpinc}. The series $\Zhat_a(q)$ will be canonically indexed by $\Spinc$ structures on $Y$; or, if we prefer, by $\Spinc$ structures modulo the conjugation $a \leftrightarrow -a$, since we will have  $$\Zhat_a(q) = \Zhat_{-a}(q).$$ See  Section~\ref{sec:spinc} for a discussion of the identification between $\Spinc$ structures and Abelian flat connections.

As in Equation \eqref{eq:s2}, we have identifications 
\be
\label{eq:spincid}
  \Spinc(Y) \cong (2\Z^s + \vdelta)/(2M\Z^s) \cong 2 \operatorname{Coker } M  + \vdelta.
  \ee
Pick some $\va \in 2\Z^s + \vdelta$ that represents a class
$$a \in (2\Z^s + \vdelta)/(2M\Z^s) \cong \Spinc(Y).$$  Then, the formula in \cite{GPPV} reads as follows:
\be
\Zhat_a(q) = (-1)^{\pi} q^{\frac{3\sigma-\sum_v m_v}{4}} \cdot \text{v.p.} \oint\limits_{|z_v|=1} \prod_{v\in \Vert} \frac{dz_v}{2\pi i z_v}  \left( z_v - \frac{1}{z_v} \right)^{2-\deg(v)} \cdot \Theta_a^{-M}(\vz)
\label{eq:plumbing1}
\ee
where
\be
\Theta_a^{-M}(\vz)= \sum_{\vell \in 2M\Z^s + \va} q^{-\frac{(\vell, M^{-1}\vell)}{4}} \prod_{ v\in \operatorname{Vert}} z_v^{\ell_v}.
\label{eq:Thetaa}
\ee
Here, v.p. denotes taking the principal value of the integral. This is given by the average of the integrals over the circles $|z_v|=1+\epsilon$ and $|z_v|=1-\epsilon$, for $\epsilon > 0$ small. (For simplicity, we will drop v.p. from notation from now on.) Also, $\pi=\pi(M)$ denotes the number of positive eigenvalues of $M$, and $\sigma=\sigma(M)$ is the signature of the matrix $M$, that is, the number of positive minus the number of negative eigenvalues. We have $\sigma = 2\pi-s$. Further, when $M$ is negative definite, we simply have $\pi=0$ and $\sigma=-s$. 

\begin{remark}
The sign $(-1)^{\pi}$ in \eqref{eq:plumbing1} was missing in \cite{GPPV}, since that paper only dealt with the negative definite case. The sign is necessary for the formula to give an invariant; cf. Proposition~\ref{prop:invariance} below.
\end{remark}

Let us give two other formulas for $\Zhat_a(q)$, easily obtained from \eqref{eq:plumbing1}.

First, note that $$\oint\limits_{|z_v|=1} \frac{dz_v}{2\pi i z_v}$$ applied to a Laurent series in $z_v$ or $z_v^{-1}$ simply has the effect of taking the constant coefficient of that series. With this in mind, we can turn \eqref{eq:plumbing1} into the formula
\begin{equation}
\Zhat_a(q) \; := \; 
2^{-s} (-1)^{\pi} q^{\frac{3\sigma-\sum_v m_{v}}{4}}
\sum_{\vell \in 2M\Z^s+\va}F_{\vell} \,q^{-\frac{(\vell,M^{-1}\vell)}{4}}
\label{eq:plumbing2}
\end{equation}
where $F_{\vell}$ are the expansion coefficients of
\begin{multline}
\label{eq:Fell}
F (z_1,\ldots,z_s)=\sum_{\vell} F_{\vell} \prod_v z_v^{\ell_v}=\\
= \prod_{v\,\in\,\Vert}\left\{
{\scriptsize \begin{array}{c} \text{Expansion} \\ \text{at } x\rightarrow 0 \end{array} }
\frac{1}{(z_v-1/z_v)^{\text{deg}\,v-2}}
+
{\scriptsize \begin{array}{c} \text{Expansion} \\ \text{at } x\rightarrow \infty \end{array} }
\frac{1}{(z_v-1/z_v)^{\text{deg}\,v-2}}\right\}.
\end{multline}

Second, let us transform the formula \eqref{eq:plumbing1} into one made from local contributions to each edge and vertex. We write
\be
\vell = 2M \vn + \va, \ \ \vn = (n_v)_{v \in \operatorname{Vert}}, \ \va=  (a_v)_{v \in \operatorname{Vert}} \in \Z^s,
\ee
so that
\be
\frac{(\vell, M^{-1}\vell)}{4} = (\vn, M\vn) + (\va, \vn) + \frac{(\va, M^{-1}\va)}{4}
\ee
From here we get a new formula
\be
\label{eq:newformula}
\Zhat_a(q) = (-1)^{\pi} q^{ \frac{3\sigma-(\va, M^{-1}\va)}{4}} \sum_{n_v} \, \oint\limits_{|z_v|=1} \, \frac{dz_v}{2 \pi i z_v} \, \prod_{\text{Vert}} \big( \ldots \big)
\, \prod_{\text{Edges}} \big( \ldots \big)
\ee
where the factor associated to a vertex $v$ with framing coefficient $m_{v}$ is
\be
 q^{-m_vn_v^2-\frac{m_v}{4}-a_vn_v} z_v^{2m_v n_v+a_v} \left( z_v - \frac{1}{z_v} \right)^2
\label{vertexrulenew}
\ee
and the factor for an edge $(u, v)$ is
\be
 q^{-2n_u n_v} \frac{z_u^{2n_v} z_v^{2n_u}}{\big( z_u - \tfrac{1}{z_u} \big) \big( z_v - \tfrac{1}{z_v} \big)}.
\label{edgerulenew}
\ee

These factors have a physical meaning. Each vertex $v$ in the plumbing graph contributes to
the 3d $\cN=2$ theory $T[Y]$ a vector multiplet with $G = SU(2)$ and supersymmetric Chern-Simons coupling at level $a_v$. Similarly, each edge $(u, v)$ of the plumbing graph contributes to $T[Y]$ matter charged under gauge groups $SU(2)_u$ and $SU(2)_v$.

In \cite[Appendix A]{GPPV}, the formula \eqref{eq:plumbing1} was introduced under the assumption that $M$ is negative definite. This condition guarantees that there is a lower bound on the exponents of $q$ that appear in  \eqref{eq:plumbing1}, and that there are only finitely many terms involving the same exponent of $q$. Hence, the right hand side of \eqref{eq:plumbing1} is well-defined. 

The negative definite condition can be relaxed as follows. (Compare \cite[Section 6.1]{CCFGH}.)

\begin{definition}
\label{def:wnd}
We say that a plumbing graph $\Gamma$ is {\em weakly negative definite}\footnote{This is not to be confused with {\em negative semi-definite}.} if the corresponding matrix $M$ is invertible, and $M^{-1}$ is negative definite on the subspace of $\Z^s$ spanned by the vertices of degree $\geq 3$. 
\end{definition}

If $\Gamma$ is weakly negative definite, then $\Zhat_a(Y)$ is still well-defined. Indeed, if a vertex $v$ has degree $\leq 2$, then the corresponding expansions in \eqref{eq:Fell} are finite, which means that only finitely many values of $\ell_v$ produce nonzero contributions. Hence, what we need is a lower bound on the exponent $-\frac{(\vell,M^{-1}\vell)}{4}$ with $\ell_v$ taking values in a finite set for vertices $v$ of degree $\leq 2$. The weakly negative definite condition ensures this.

Examples of weakly negative definite graphs that are not negative definite can be obtained using the Neumann move (c) from Figure~\ref{fig:kirby}. Observe that the diagram on the top has a vertex labeled $0$, and thus cannot be negative definite. Nevertheless, if the bottom graph is negative definite, the top graph can be seen to be weakly negative definite. Then, the formula \eqref{eq:plumbing1} still makes sense for the top graph, and gives the same result as for the bottom graph. (See Proposition~\ref{prop:invariance} below for the invariance result.)

\begin{remark}
\label{rem:0nd}
In some cases, one can even define $\Zhat_a(q)$ when $M$ is not invertible, provided that $\va$ is in the image of $M$. Then, we could write $\vell = M \vk$ for some $\vk$, and replace $(\vell, M^{-1}\vell)$ with $(\vk, M\vk)$ in \eqref{eq:Thetaa}. This allows one to consider manifolds with $b_1 > 0$. For a simple example of this, take the graph with a single vertex, labeled by $0$, and $\va=(0)$. The graph represents $S^1 \times S^2$, and the formula \eqref{eq:Thetaa} gives $$\Zhat_0(S^1 \times S^2)=-2,$$ 
in our normalization.
\end{remark}

\begin{remark}
A rigorous proof of the convergence of $\Zhat_a(q)$ to the WRT invariants, as $q$ approaches a root of unity, has not yet appeared in the literature. In the special case where $Y=\Sigma(b_1, b_2, b_3)$ is a Seifert fibered integer homology sphere with three singular fibers, convergence to the WRT invariants follows from the work of Lawrence and Zagier \cite[Theorem 3]{LawrenceZagier}, combined with Proposition~\ref{prop:brieskorn} below.
\end{remark}

\subsection{Invariance}
To make sure that the formula \eqref{eq:plumbing1} gives an invariant of the plumbed manifold $Y$ and the $\Spinc$ structure $a$, we need to check that it does not depend on the presentation of $Y$ as a plumbing.  This fact is well-known to experts, but we include a proof here for completeness.

\begin{proposition}
\label{prop:invariance}
The series $\Zhat_a(q)$ defined in \eqref{eq:plumbing1} is unchanged by the Neumann moves from Figure~\ref{fig:kirby}.
\end{proposition}

\begin{proof}
Consider the move (a), with the signs on top being $-1$. We keep the notation from Section~\ref{sec:PlumbedSeries} for the quantities associated to the bottom graph, and we use a prime to denote those for the top graph. For example, $M$ is the matrix for the bottom graph, and $M'$ the one for the top graph. The quadratic form associated to $M'$ has an extra negative term compared to the one for $M$, namely
$$-x_1^2 - 2x_0x_1 - x_0^2 - 2x_0x_2 - x_2^2 - 2x_1x_2 = - (x_1 + x_0 + x_2)^2,$$
where $x_1, x_0, x_2$ are the variables for the three vertices shown in the figure (in this order, from left to right). Therefore, the signature $\sigma'$ is $\sigma-1$, and the number $\pi$ of positive eigenvalues does not change. The quantity $3\sigma - \sum_v m_v$ does not change either, so we have the same factors in front of the integral in   \eqref{eq:plumbing1}.

 For the bottom graph, as in Section~\ref{sec:idSpinc}, let us write a vector $\vell \in \Z^s$ as a concatenation $$\vell =(\vell_1, \vell_2),$$ where $\vell_1$ corresponds to the part of the graph on the left of the edge where we do the blow-up (including the vertex labeled $m_1$), and $\vell_2$ corresponds to the part on the right (including the vertex labeled $m_2$). From $\vell$ we can construct a vector for the top graph
$$\vell' := (\vell_1, 0, \vell_2) \in \Z^{s+1}.$$
(This was denoted $T\vell$ in Section~\ref{sec:idSpinc}.)

If a $\Spinc$ structure $a$ is represented by the vector $\va$ for the bottom graph, we will represent it by $\va'$ for the top graph. If $z_0$ denotes the variable for the newly introduced vertex of degree $2$ in the top graph, observe that the integral in   \eqref{eq:plumbing1} only picks up the constant coefficient from the powers of $z_0$ in the theta function \eqref{eq:Thetaa}. Thus, we can just sum over vectors in $2M' \Z^{s+1} + \va$ that have a $0$ in the respective spot; that is, those of the form $\vell'$ for some $\vell \in 2M \Z^s + \va$. By simple linear algebra, if
$$ M^{-1}\vell = \vw = (\vw_1, \vw_2)$$
then
$$ (M')^{-1} \vell' = (\vw_1, w_0, \vw_2),$$
where $w_0$ is the sum of the entries of $\vw$ at the two vertices abutting the edge where we do the blow-up.
This implies that
\be
\label{eq:elmatrix}
 (\vell, M^{-1}\vell) = (\vell', (M')^{-1}\vell').
\ee
From here we get that the integrals give the same result for the two graphs, and hence the series $\Zhat_a(q)$ are the same.

The case of the move (a) with the sign $+1$ is similar, but with some differences. The quantity $3\sigma - \sum_v m_v$ is unchanged by the move, but $\pi' = \pi+1$ so the sign switches. Given $\vell = (\vell_1, \vell_2)\in \Z^s$, this time we define
\be
\label{eq:vellpplus}
 \vell' = (\vell_1, 0, -\vell_2)
\ee
and then \eqref{eq:elmatrix} still holds. We choose $\va'$ from $\va$ by \eqref{eq:vellpplus}. To compare the theta functions for the two graphs, given the change of sign in $\vell_2$ we need to do the substitutions $z_v \to z_v^{-1}$ for all the vertices $v$ on the right hand side. Let $\Vert_2$ be the set of these vertices. In the formula \eqref{eq:plumbing1}, all the expressions $z_v - \frac{1}{z_v}$ pick up a sign for $v \in \Vert_2$. This produces an overall sign of $(-1)^g$, where
$$ g = \sum_{v \in \Vert_2} (2-\deg(v))$$
is odd. The resulting $-1$ sign cancels with the one that gives the discrepancy in $(-1)^{\pi}$.

Next, we consider the move (b), with the sign of the blow-up being $-1$. Then, we have $\pi'=\pi$, $\sigma' = \sigma -1$ and the quantity $3\sigma - \sum_v m_v$ is $1$ lower for the top graph as for the bottom one. Thus, doing the blow-up gives an extra factor of $q^{-1/4}$ in front of the integral in \eqref{eq:plumbing1}.

For the bottom graph let us write vectors as 
$$\vell=(\vell_0, \ell_1) \in \Z^s,$$ with the entry $\ell_1$ being for the vertex labeled $m_1$. For the top graph we have corresponding vectors of the form
$$ \vell'_{\pm} = (\vell_0, \ell_1 \pm 1, \mp 1).$$
Given a vector $\va$ for the bottom graph, the same $\Spinc$ structure $a$ can be represented by either $\va'_+$ or $\va'_-$ in the top graph. Let $z_0$ be the variable for the newly introduced terminal vertex in the top graph, and $z_1$ for the vertex labeled by $m_1$ in the bottom graph and $m_1 -1$ in the top graph. For the top graph, the integrand in \eqref{eq:plumbing1} has new factors
$$ \left(z_1 - \frac{1}{z_1}\right)^{-1} \cdot  \left(z_0 - \frac{1}{z_0}\right).$$
Thus, the integral only picks up expressions from the theta function of the form $\prod_v z_v^{\ell_v}$ where $\ell_0=\mp 1$, and these come with a sign $\pm 1$. We have
$$ (\vell, M^{-1}\vell) = (\vell'_{\pm}, (M')^{-1}\vell'_{\pm})  + 1.$$   
This means that the relevant terms (those with $\ell_0=\mp 1$) in the theta function for $-M'$ sum up to
$$ q^{1/4}   \left(z_1 - \frac{1}{z_1}\right) \cdot \Theta^{-M}_a(\vz).$$
Putting everything together, we get the same answer for $\Zhat_a(q)$ as computed from the two graphs.

Move (b) with the sign of the blow-up being $+1$ is similar, but now we use vectors of the form
$$ \vell'_{\pm} = (\vell_0, \ell_1 \pm 1, \pm 1).$$

For move (c), the top graph has an extra negative and an extra positive eigenvalue, so $\pi' = \pi+1$ and $3\sigma - \sum_v m_v$ is unchanged. Let $z_1, z_0, z_2$ be the variables for the vertices $v_1, v_0, v_2$ shown in the top graph (in this order), and $z_b$ the one for the vertex $v_b$ shown in the bottom graph. In the integral for the top graph, since the middle vertex has degree $2$, we only care about contributions from vectors $\vell'$ with $\ell'_0=0$. We write these vectors as
$$ \vell'=(\vell_1, \ell_1, 0, \ell_2, \vell_2)$$
where $\vell_1$ has the entries for vertices on the left side of the graph (not including $v_1$), and $\vell_2$ has the entries for vertices on the left side of the graph (not including $v_2$). From $\vell'$ we can create a vector for the bottom graph
$$ \vell = (\vell_1, \ell_1 - \ell_2, - \vell_2).$$
A linear algebra exercise shows that
\be
\label{eq:vem0}
 (\vell, M^{-1}\vell) = (\vell', (M')^{-1}\vell').
 \ee
Given a vector 
$$\va = (\va_1, a_b, -\va_2)$$ representing the $\Spinc$ structure $a$ for the bottom graph, we will use $\va'$ for the top graph, where $\va'$ is any vector of the form
$$ \va'=(\va_1, a_1, 0, a_2, \va_2)$$
with $a_1-a_2 = a_b$ and $a_1$, $a_2$ of the correct parity (determined by the degrees of those vertices). 

Using \eqref{eq:vem0}, we see that in the theta function for $-M'$, we have identical powers of $q$ from all vectors of the form $\vell'$ with $\vell_1, \vell_2$ and $\ell_1 - \ell_2=\ell_b$ fixed. Thus, we are integrating an expression of the form
$$ \Bigl( \sum_{n \in \Z} (z_1 z_2)^{2n} \Bigr) \cdot G(\vz').$$
When we integrate this over the circle with respect to $z_1$ and $z_2$, we  pick up the constant terms in $z_1$ and $z_2$. In particular, we only get contributions from the terms in $G(\vz')$ where the exponents of $z_1$ and $z_2$ are the same even number, and any such term contributes once. Thus, we could get the same result by setting $z_1z_2=1$ in the expression above and just integrate over a single variable $z_b :=z_1 = z_2^{-1}$. Let us also change variables from $z_v$ to $z_v^{-1}$ for all vertices $v$ on the right hand side of the graph. Then, observing that
$$ (2-\deg(v_1)) + (2-\deg(v_2)) = 2-\deg(v_b),$$
we get that the integral in the formula  \eqref{eq:plumbing1} for the top graph recovers the one for the bottom graph, up to a sign of $-1$. This sign is canceled by the one coming from the discrepancy between $(-1)^{\pi}$ and $(-1)^{\pi'}$.
\end{proof}

\begin{remark}
The Neumann moves preserve the property of a graph being weakly negative definite (so that $\Zhat_a(q)$ is well-defined). For a proof of this fact, see \cite[Appendix A]{CCFGH}. Thus, if $Y$ admits a weakly negative definite plumbing diagram, then all of its plumbing diagrams are weakly negative definite.  For example, in view of Theorem~\ref{thm:euler}, a generic Seifert bundle over $S^2$ with positive orbifold Euler number ($e > 0$) admits a positive definite plumbing (which is clearly not weakly negative definite), and therefore cannot be represented by any weakly negative definite plumbing. It follows that generic Seifert bundles over $S^2$ have weakly negative definite representations if and only if $e < 0$. 
\end{remark}

\subsection{$\Spinc$ structures versus Abelian flat connections}
\label{sec:spinc}
In \cite{GPV}, \cite{GPPV}, the invariants $\Zhat_a(q)$ were indexed by Abelian flat connections (modulo conjugation). In the case of rational homology spheres, these connections correspond to elements of $H_1(Y; \Z)/\Z_2$. However, in Section~\ref{sec:PlumbedSeries}, the labels were $\Spinc$ structures on $Y$ (modulo conjugation). Let us discuss this discrepancy.

By Poincar\'e duality, we have $H_1(Y; \Z) \cong H^2(Y; \Z)$. Further, there is a well-known affine (non-canonical) identification
$$\Spinc(Y) \cong H^2(Y; \Z).$$
This identification can be made canonical after choosing a $\Spinc$ structure $a_0$ that should correspond to $0 \in H^2(Y; \Z)$. Concretely, in the case of a negative definite plumbed manifold $Y=Y(\Gamma)$ as in Section~\ref{sec:PlumbedDef}, by \eqref{eq:spincid}, we have $\Spinc(Y) \cong (2\Z^s + \vdelta)/(2M\Z^s)$. On the other hand, $H_1(Y; \Z) \cong H^2(Y; \Z)$ is canonically $\Z^s/M\Z^s$. Choosing an $\va_0 \in 2\Z^s + \delta$ will give the identification 
$$\Spinc(Y) \cong H_1(Y; \Z), \ \ \ [\va] \to [(\va-a_0)/2].$$

In order for this identification to commute with the conjugation on the two sides, we need that $\va_0 \in M \Z^s$. One can prove (by induction on the number of vertices in the plumbing graph) that such an $\va_0$ always exists, i.e. $$(2\Z^s + \vdelta) \cap (M \cdot \Z^s) \neq \emptyset.$$ 
If this intersection contained elements in a single equivalence class modulo $2M\Z^s$ and conjugation, then the identification between $\Spinc$ structures and Abelian flat connections would be uniquely determined (canonical). However, this is not always the case, as the following example shows.

Let $\Gamma$ be the following graph:
$$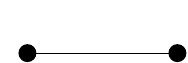$$

 The resulting plumbed manifold $Y=Y(\Gamma)$ is the lens space $L(8,3)$. There is a self-diffeomorphism $h: Y \to Y$ given by interchanging the two vertices. We have
$$ H_1(Y; \Z) \cong \Z^2/\text{Span}\langle (-3,1), (1, -3) \rangle.$$
The quotient $H_1(Y; \Z)/\Z_2$ consists of $5$ elements $\alpha_j$, with representatives $(j,0) \in \Z^2$, for $j=0, \dots, 4$. These in fact correspond to the Abelian flat connections given by sending the generator of $\pi_1(Y)$ to the diagonal matrix $\mathit{diag}(\omega^j, \omega^{-j})$, where $\omega = e^{\pi i /4}.$ The diffeomorphism $h$ preserves $\alpha_0$, $\alpha_2$ and $\alpha_4$, but interchanges $\alpha_1$ with $\alpha_3$.

On the other hand, we have
$$ \Spinc(Y) \cong (2\Z + 1)^2/\text{Span}\langle (-6,2), (2, -6) \rangle.$$
The quotient $\Spinc(Y)/\Z_2$ also consists of $5$ elements, represented by $(1,1)$, $(3,1)$, $(-1,3)$, $(3,-1)$ and $(5, -1)$. Interestingly, the diffeomorphism $h$ induces the identity map on $\Spinc(Y)/\Z_2$.

There are two possible choices of $\va_0$ (modulo $2M\Z^2$ and conjugation), namely $(-1,3)$ and $(3,-1)$. This shows that there is no canonical identification between $\Spinc$ structures and Abelian flat connections.
Indeed, the two possible identifications are interchanged under the diffeomorphism $h$. 

With regard to the invariants $\Zhat_a$, we can compute 
$$ \Zhat_{(1,1)}=q^{1/4}, \ \ \Zhat_{(5,-1)}=q^{-1/8},$$
and the other three series are zero. Under one of the two possible identifications, the element $(5,-1) \in \Spinc(Y)/\Z_2$ corresponds to the flat connection $\alpha_1$, and under the other identification, to the flat connection $\alpha_3$. Thus, if we wanted to label $\Zhat_a$ by Abelian flat connections, we would need to first make a choice between $\alpha_1$ and $\alpha_3$ (the two connections interchanged by $h$).

In conclusion, this example shows that it is more natural to use labels by $\Spinc$ structures instead of Abelian flat connections. It would be interesting to understand how $\Spinc$ structures arise in the resurgence picture (cf. Section~\ref{sec:resurgence} below), where we more naturally encounter Abelian flat connections.

\subsection{Brieskorn spheres}
We now present a simplification of the formula \eqref{eq:plumbing1} in the case of the integral homology Seifert fibrations $\Sigma(b_1, b_2, b_3)$ with three singular fibers. The same simplified formula appeared in the work of Lawrence and Zagier \cite[Section 6]{LawrenceZagier}. Yet another derivation of this formula was found by Chung \cite{Chung}, using resurgence analysis.

Let us first introduce the following false theta functions (Eichler integrals of weight $3/2$ vector-valued modular forms):
\be
\tPsi^{(a)}_p (q)  :=  \sum_{n=0}^\infty \psi^{(a)}_{2p}(n) q^{\frac{n^2}{4p}} \ \ \in q^\frac{a^2}{4p}\,\Z[[q]]
\label{Psis}
\ee
where
\be
\psi^{(a)}_{2p}(n)  =  \left\{
\begin{array}{cl}
 \pm 1, & n\equiv \pm a~\mod~ 2p\,, \\
0, & \text{otherwise}.
\end{array}\right.
\ee
We will use $\tPsi^{n_a (a) + n_b (b) + \cdots}_p (q)$ as a shorthand notation for a linear combination
\be
\tPsi^{n_a (a) + n_b (b) + \cdots}_p (q) \; := \;
n_a \tPsi^{(a)}_p (q)
+ n_b \tPsi^{(b)}_p (q)
+ \ldots
\label{linearPsitilde}
\ee

\begin{proposition}
\label{prop:brieskorn}
Consider the Brieskorn sphere $Y=\Sigma(b_1, b_2, b_3)$, where the positive integers $b_1 < b_2 < b_3$ are pairwise relatively prime. Then
\be
\label{eq:Zbrieskorn}
  \Zhat_0(Y) = q^{\Delta} \cdot (C- \tPsi^{(\alpha_1)-(\alpha_2)-(\alpha_3)+(\alpha_4)}_{b_1b_2b_3}(q))
\ee
where
\begin{align*}
 \alpha_1 &=b_1b_2b_3 - b_1b_2-b_1b_3-b_2b_3, \\
\alpha_2 &=b_1b_2b_3 + b_1b_2-b_1b_3-b_2b_3, \\
\alpha_3 &=b_1b_2b_3 - b_1b_2 + b_1b_3-b_2b_3, \\
\alpha_4 &=b_1b_2b_3 + b_1b_2 + b_1b_3-b_2b_3,
\end{align*}
$\Delta$ is some rational number, and
\[
C =\begin{cases}
2q^{1/120} & \text{if } (b_1, b_2, b_3)= (2,3,5),\\
0 & \text{otherwise.}
\end{cases}
\]
\end{proposition}

\begin{proof}
The Brieskorn sphere $Y$ is the Seifert manifold $M(b;  a_1/b_1, a_2/b_2, a_3/b_3),$ where $b < 0$ and $ a_1, a_2, a_3 > 0$ are chosen such that
$$ b + \sum_{i=1}^3 \frac{a_i}{b_i} =  -\frac{1}{b_1b_2b_3}.$$

As a negative definite plumbed manifold, $Y$ comes from a tree $\Gamma$ with central vertex labelled $b$ and three legs with labels $-k_{i1}, \dots, -k_{is_i}$, $(i=1, \dots, 3)$, giving the continued fraction decompositions of $b_i/a_i$ as in \eqref{eq:pf}. The total number of vertices is
$$ s = s_1 + s_2 + s_3 + 1.$$

Recall the formula  \eqref{eq:plumbing1} for $\Zhat_a$ of a negative definite plumbed manifold. In our case, since $H_1(Y; \Z)=0$, there is a unique value $a=0$ to consider. We have
$$ \Zhat_0(Y) =  q^{-\frac{3s+\sum_v m_v}{4}}   \oint\limits_{|z_v|=1} \prod_{v\in \Vert} \frac{dz_v}{2\pi i z_v}  \left( z_v - \frac{1}{z_v} \right)^{2-\deg(v)} \cdot \sum_{\vell \in 2M \Z^s + \vdelta} q^{-\frac{(\vell, M^{-1}\vell)}{4}} \prod_{v \in \Vert} z_v^{\ell_v}.$$

Concretely, each integral $\oint_{|z_v|=1} \tfrac{dz_v}{2\pi i z_v}$ gives the constant coefficient in the  expansion in $z_v$. For the three terminal vertices (of degree $1$), this means that the values of $\ell_v$ should be $\eps_i=\pm 1$, for $i=1,2,3$; and doing the integral results in a sign of $-\eps_i$ in front of the expression. For the degree two vertices (i.e. all but the central vertex and the three terminal vertices), we get $\ell_v =0$. Letting $m$ be the value $\ell_v$ on the central vertex, note that (since $\det(M)=\pm 1$) the condition $\vell \in 2M \Z^L + \vdelta$ means that $m$ should be odd.

Furthermore, the central vertex has degree $3$. Writing
$$\left(z_v - \frac{1}{z_v}\right )^{-1} = -\frac{1}{2} \sum_{k \in 2\Z + 1} \operatorname{sign}(k) z_v^k$$
we find that the principal value of the integral over $z_v$ for the central vertex gives $1/2$ times a sum over all expressions with $\ell_v= m$ odd, with coefficients $\operatorname{sign}(m)$. Thus, we have
\be
\label{eq:zoya}
 \Zhat_0(Y) =  -\frac{1}{2}\cdot {q^{-\frac{3s+\sum_v m_v}{4}}} \cdot \sum_{m  \text{ odd}} \sum_{\eps_i \in \{\pm 1\}} \eps_1\eps_2\eps_3\operatorname{sign}(m) \cdot q^{-\frac{(\vell, M^{-1}\vell)}{4}},
 \ee
where $\vell$ is the vector with coordinates $m$ for the central vertex, $\eps_i$ ($i=1,2,3$) for the final vertices, and $0$ for all intermediate vertices.

Let $\lambda$ denote the $(1,1)$ entry in the matrix $M^{-1}$, corresponding to the central vertex, and $\mu_i$ the entry in the first row, in the column for the terminal vertex on the $i$th leg,  $i=1,2,3$. Let also $\tau_{ij}$ be the diagonal entry in the row for the terminal vertex on the $i$th leg and the column for the terminal vertex on the $j$th leg. Then, the exponent of $q$ in the last factor in \eqref{eq:zoya} is
\be
-\frac{(\vell, M^{-1}\vell)}{4}= - \frac{\lambda m^2 + 2m(\sum_i \mu_i \eps_i) +\sum_{i,j} \tau_{ij}}{4}.
\ee

 To compute the values $\lambda$, $\mu_i$ and $\tau_{ij}$, we consider the corresponding $(s-1)\times(s-1)$ minors in the matrix 
$$ M =  \left( \begin{array}{c|ccccc|ccccc|ccccc}
b & 1 & 0 & \dots     & 0	& 1 & 0 & \dots  & 0 & 0 & 1 & 0 & \dots & 0& 0 \\
\hline
1 & -k_{11} & 1 & & & & & & & & & & & & &  \\
0 & 1& -k_{12} & 1& & & & & & & & & & & & \\
\vdots & & 1&\ddots  & & & & & & & & & & & &   \\
0 & & & &  \ddots & 1& && & & & & & & &   \\
0 & & & &  1 & -k_{1s_1} & & & & & & & & & &  \\
\hline
1 & & &  & &	 	& -k_{21} & 1& & & & & & & &  \\
0 & & & & 	&	& 1& -k_{22} & 1& & & & & & &  \\
\vdots  & & & & &	& & 1 & \ddots& & & & &  &  &\\
0 & &  & & 	&	& & & & \ddots &1 & & & & &  \\
0 & &  & & 	&	& & & & 1& -k_{2s_2}& & & & &  \\
\hline
1 & &  & & & & & & & & 		& -k_{31} & 1  & & &  \\
0 & &  & & & & & & & & 		& 1& -k_{32}& 1 & &  \\
\vdots  & & & & & & & & & & 	& & 1& \ddots  &  &\\
0 & &  & & & & & & & &		& & & & \ddots & 1\\
0 & &  & & & & & & & &		& & & & 1 & -k_{3s_3}
\end{array} \right).$$
Note that $\det(M) = (-1)^s$ and therefore $\det(-M)=1$. Furthermore, the three large diagonal blocks have  determinants $(-1)^{s_i}b_i$, $i=1,2,3$. With this in mind, we find that
$$ \lambda = -b_1b_2b_3, \ \mu_1 =-b_2b_3, \ \mu_2 = -b_1 b_3, \ \mu_3 = -b_1 b_2,$$
$$ \tau_{ij} = - b_k \ \text{ for } \ \{i,j,k\}=\{1,2,3\},$$
whereas 
$\tau_{ii}$ is (up to a sign) the determinant of the linking matrix for the graph where we delete the terminal vertex on the $i$th leg. In fact, 
$$h_i := -\tau_{ii} > 0$$ 
equals the the cardinality of $H_1$ of the corresponding plumbed manifold.

Thus, the exponent of $q$ is
$$ \frac{b_1b_2b_3}{4} \Bigl(m^2 + 2m\sum_i  \frac{\eps_i}{b_i} + 2\sum_{i < j} \frac{\eps_i\eps_j}{b_i b_j} \Bigr) + \frac{\sum_i h_i}{4} = \frac{b_1b_2b_3}{4}\Bigl(m+\sum_i  \frac{\eps_i}{b_i} \Bigr)^2 - \frac{b_1b_2b_3}{4} \sum \frac{1}{b_i^2} +\frac{\sum_i h_i}{4}.$$
From here we get
$$\Zhat_0(Y) = -\frac{q^{\Delta}}{2}  \cdot \sum_{m  \text{ odd}} \sum_{\eps_i \in \{\pm 1\}} \eps_1\eps_2\eps_3 \operatorname{sign}(m) \cdot q^{\frac{b_1b_2b_3}{4} \left(m+\sum_i  \frac{\eps_i}{b_i} \right)^2 },$$
where
\be
\label{eq:Delta}
\Delta =  \frac{1}{4} \Bigl( \sum_i h_i - 3s- \sum_v m_v -\frac{b_2b_3}{b_1} - \frac{b_1b_3}{b_2} - \frac{b_1b_2}{b_3} \Bigr).
\ee

By making use of the symmetry that reverses the signs of all $\eps_i$ and $m$ at once, we can turn the sum over $m \in 2\Z +1$ into one over $m=2n+1 > 0$. Thus
\be
\label{eq:zoy}
\Zhat_0(Y)= -q^{\Delta} \cdot \sum_{\eps_i \in \{\pm 1\}}  \sum_{n \geq 0} \eps_1\eps_2\eps_3\cdot q^{{b_1b_2b_3}\left(n^2 +n+ \frac{1}{4}+ (n+ \frac{1}{2})\sum_i  \frac{\eps_i}{b_i} + \frac{1}{4} (\sum_i  \frac{\eps_i}{b_i})^2 \right)}.
\ee

Let $p=b_1b_2b_3$. If we fix $\eps_1, \eps_2$ and let $\eps_3=-1$, note that 
$$b_1b_2b_3\Bigl(1+\sum_i  \frac{\eps_i}{b_i}\Bigr) = \alpha_k$$ for some $k\in\{1,2,3,4\}$. Therefore, the corresponding summation over $n$ in \eqref{eq:zoy} becomes
\be
\label{eq:eps3m}
 - \sum_{n \geq 0} \eps_1\eps_2\cdot q^{p n^2 +\alpha_k n+ \frac{\alpha_k^2}{4p}}.
 \ee
On the other hand, for $\eps_3=1$ (after replacing $n$ with $n-1$) we get the contribution
\be
\label{eq:eps3p}
\sum_{n \geq 1}\eps_1\eps_2 \cdot q^{p n^2 -\alpha_k n+ \frac{\alpha_k^2}{4p}}.
 \ee
When $(b_1, b_2, b_3)\neq (2,3,5)$, since the $b_i$ are relatively prime, it is easy to see that
\be
\frac{1}{b_1} + \frac{1}{b_2} + \frac{1}{b_3} < 1
\ee
and therefore 
$$0 <\alpha_k < 2p, \ \ \text{for } k=1,\dots,4.$$ 
In this case, by replacing $a$ with $\alpha_k$ and $n$ with $2pn\pm \alpha_k$ in \eqref{Psis}, we can write
\be
\label{eq:2psi}
\tPsi^{(\alpha_k)}_{p}= \sum q^{pn^2 + k\alpha_i + \frac{\alpha_k^2}{4p}} - \sum q^{pn^2 - n\alpha_k + \frac{\alpha_k^2}{4p}},
\ee
where the first sum is over all $n$ with $n + \frac{\alpha_k}{2p} \geq 0$, and the second is over $n$ with $n - \frac{\alpha_k}{2p} \geq 0$. When $0 <\alpha_k < 2p$, this happens exactly when $n \geq 0$ for the first sum, and when $n \geq 1$ for the second sum.

When $(b_1, b_2, b_3)\neq (2,3,5)$, since the $b_i$ are relatively prime, it is easy to see that
\be
\frac{1}{b_1} + \frac{1}{b_2} + \frac{1}{b_3} < 1
\ee
and therefore $0 <\alpha_k < 2p$ for all $k=1,\dots,4$. Therefore, the sum of the two expressions \eqref{eq:eps3m} and \eqref{eq:eps3p} is exactly $-\eps_1\eps_2 \tPsi^{(\alpha_k)}_{p}$. The eight kinds of terms in the sum in \eqref{eq:zoy} combine in pairs to give four different $\tPsi^{(\alpha_k)}_{p}$ (up to some signs), and we arrive at the formula \eqref{eq:Zbrieskorn}, with $C=0$. 

For the Poincar\'e sphere $P=\Sigma(2,3,5)$, we have $\alpha_1=-1,$ $\alpha_2=19$, $\alpha_3=11$, $\alpha_4=31$. In this case, in \eqref{eq:2psi}, the first summation is over $n \geq 1$ and the second is over $n \geq 0$. This gives the extra term $C=2q^{1/120}$ in the formula \eqref{eq:Zbrieskorn}. Further, the expression $\tPsi^{(\alpha_1)-(\alpha_2)-(\alpha_3)+(\alpha_4)}_{b_1b_2b_3}(q)$ in  \eqref{eq:Zbrieskorn} agrees with $q^{1/120}\cdot A(q)$ from \eqref{chichirel}. Also, we can compute $h_1=8$, $h_2=4$,  $h_3=2$, $s=8$, $m_v =-2$ for all $v$, and hence $\Delta = -181/120$. This gives
$$\Zhat_0(P)= q^{-3/2}(2-A(q)),$$
in agreement with \eqref{eq:zhp}.
\end{proof}

In principle, the same method can be used to simplify the formula for more general plumbings. If there is more than one vertex of degree three, we end up with a sum over several different indices $m_i$ instead of $m$. If there is a vertex of index more than three, we have to factor out $(z_v - 1/z_v)^{\deg(v)-2}$ in the integral, and the formula gets more complicated. Also, if the manifold is not a homology sphere, we would have to split the sum according to elements in $\Spinc(Y)/ \Z_2$.

\newpage \section{Knot complements from plumbings}
\label{sec:plumbedKnot}

In this section we study three-manifolds with torus boundary that arise from (negative definite) plumbings. Manifolds of this type have previously appeared in \cite{OSS}, the context of Heegaard Floer homology and lattice homology. 

\subsection{Plumbing representations}
\label{sec:PlumbedKnot}
Let us keep the notation from Subsection~\ref{sec:PlumbedDef}, but now assume that in the weighted tree $\Gamma$ we distinguish one particular vertex $v_0$. We will mostly be interested in the case where $v_0$ has degree one, but this condition is not necessary for most of the discussion. Let $\hGamma$ be obtained from $\Gamma$ by deleting $v_0$ and the edges incident to it. (Note that $\hGamma$ is disconnected if $v_0$ has degree at least $2$. In that case, we still have a plumbed manifold $Y(\hGamma)$, which is the connected sum of the plumbings associated to each component of $\hGamma$.) 

The component associated to $v_0$ in the link $L(\Gamma)$ represents a knot $K \subset Y(\hGamma)$. We let $Y(\Gamma, v_0)$ denote the complement of a tubular neighborhood of $K$ in $Y(\hGamma)$. This is a three-manifold with torus boundary. Moreover, the boundary is parameterized, in the following sense.

\begin{definition}
A compact, oriented three-manifold $Y$ is said to have {\em parametrized torus boundary} if $\del Y$ is homeomorphic to $T^2$ and, furthermore, we have specified an orientation-preserving homeomorphism $f: T^2 \to \del Y$, where $T^2 = \R^2/\Z^2$ is the standard torus.
\end{definition}

A parametrization of a torus boundary $\del Y$ produces two simple closed curves on $\del Y$, the images of $f(\{0\} \times S^1)$ and $f(S^1 \times \{0\})$ under the homeomorphism $f$. We call these the {\em meridian}  and the {\em longitude}. Conversely, choosing two simple closed curves on $\del Y$ whose classes span $H_1(\del Y)$ determines the parametrization, up to isotopy. (The curves should be oriented such that the orientation on $\del Y$ induced by $f$ agrees with its orientation from being the boundary of $Y$. In practice, this means we should specify the orientation on one curve, and then the other is automatically determined.)

If we have a manifold $Y$ with parametrized torus boundary, we can form a closed manifold
$$ \hat Y = Y \cup_{\del Y} (S^1 \times D^2),$$
by gluing the boundaries such that the meridian of $\del Y$ gets matched to the meridian $\mathit{pt} \times \del D^2$ on the solid torus. Thus, we can view $Y$ as the complement of (a neighborhood of) the knot $K \subset \hat Y$:
$$ Y = \hat Y \setminus \nu K,$$ 
 where $K=S^1 \times \{0\}$ is the core of the solid torus and $\nu K = S^1 \times D^2$ is a tubular neighborhood of $K$. Further, the longitude on $\del Y$ specifies a framing of the knot $K$.

In the case $Y=Y(\Gamma, v_0)$, we take $\hat Y = Y(\hGamma)$. The plumbing representation specifies  the meridian $\mu$ of the knot, as well as a longitude $\lambda$ given by the framing of the knot $K$. The framing is determined by the weight $m_{v_0}$ of $v_0$ in the graph $\Gamma$. We will call it the {\em graph framing}. We orient the longitude $\lambda$ counterclockwise. This gives a parametrization of $\del Y(\Gamma, v_0)$.

It is important to distinguish $\lambda$ from two other natural choices of longitude: One is the {\em blackboard (zero) framing}   of $K$, which we denote by $\lambdabb$, so that in $H_1(\del Y)$ we have the relation
$$ \lambda = \lambdabb + m_{v_0} \mu.$$
The other is the (rational) {\em Seifert framing}, which is the combination $$\lambdasf = \lambdabb + m_{\operatorname{SF}}\cdot  \mu$$ that gets sent to zero under the map $H_1(\del Y; \Q) \to H_1(Y; \Q)$. This exists, and is unique, provided that $Y(\hGamma)$ is a rational homology sphere; for example, if $\hGamma$ is negative definite. The exact value of $m_{\operatorname{SF}}$ depends on the graph.

\begin{example}
\label{ex:unknot01}
Both diagrams in Figure~\ref{fig:unknot01} represent the unknot in $S^3$. The distinguished vertex $v_0$ is marked by an empty circle. On the left, the blackboard and the Seifert framings coincide ($m_{\operatorname{SF}}=0$), and the graph framing differs by $p \mu$ from them. On the right, we have $  \lambda = \lambdabb + (p-1) \mu = \lambdasf + p \mu$ and $m_{\operatorname{SF}}=-1$.
\end{example}

\begin {figure}
\begin {center}
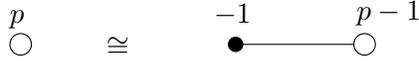
\caption {Plumbing diagrams for the unknot in $S^3$.}
\label{fig:unknot01}
\end {center}
\end {figure}

\begin{example}
\label{ex:trefoil6}
Consider the graph $\Gamma$ on the left of Figure~\ref{fig:trefoil6}. The manifold $Y(\hGamma)$ is just $S^3$. By doing three successive blow-downs on the corresponding Kirby diagram, we get a trefoil. The graph framing $p$ on the unknot from $v_0$ becomes the $p+6$ framing on the trefoil. Thus, $Y(\Gamma, v_0)$ is the complement of the trefoil in $S^3$, with framing $p+6$ compared to the Seifert framing; that is,  $\lambda = \lambdabb + p \mu = \lambdasf+ (p+6) \mu$ and $m_{\operatorname{SF}}=6$.
\end{example}

\begin {figure}
\begin {center}
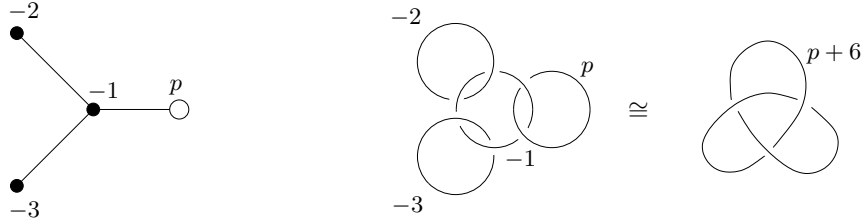
\caption {A plumbing diagram for the trefoil in $S^3$.}
\label{fig:trefoil6}
\end {center}
\end {figure}

The Neumann moves from Figure~\ref{fig:kirby} apply equally well to graphs with a distinguished vertex, and they can involve this vertex (as, for example, in the blow-up move from Figure~\ref{fig:unknot01}); the only restriction is that we do not allow blowing down the distinguished vertex. Any two diagrams of the same plumbed manifold with parametrized boundary are related by these Neumann moves. Note that such moves leave the graph and Seifert framings unchanged, but may change the blackboard framing.

We can change the parametrization of the boundary in a plumbing graph as follows. First, we can change the longitude (the graph framing) by simply changing the weight of $v_0$. This does not change the meridian, so the underlying manifold $\hat Y= Y(\hGamma)$ is the same. Second, we can also change the meridian, by adding to the graph an extra leg, starting at the distinguished vertex, and making the end of the leg the new distinguished vertex, as in Figure~\ref{fig:ChangeFraming}. (The weights on the new leg can be arbitrary.) This usually changes the manifold $\hat Y$.

\begin {figure}
\begin {center}
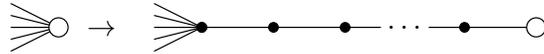
\caption {Changing the parametrization of the boundary for a plumbed knot complement.}
\label{fig:ChangeFraming}
\end {center}
\end {figure}

\begin{example}
The graph in Figure~\ref{fig:trefoil237} is obtained from the one in Figure~\ref{fig:trefoil6} by a move of the kind in Figure~\ref{fig:ChangeFraming}. It still represents the trefoil complement, but the meridian has changed such that $\hat Y = \Sigma(2,3,7).$
\end{example}

\begin {figure}
\begin {center}
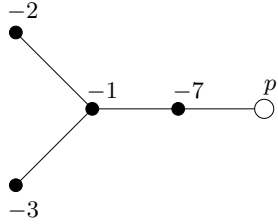
\caption {A plumbing diagram for the trefoil complement, viewed as the complement of a knot in the Brieskorn sphere $\Sigma(2,3,7)$.}
\label{fig:trefoil237}
\end {center}
\end {figure}


Note that the knot complements that are obtained from plumbing diagrams are all graph manifolds. If it happens that $Y(\hGamma) = S^3$, this means that all the knots in $S^3$ that can be obtained this way are algebraic, i.e., iterated torus knots. For example, we cannot obtain hyperbolic knots such as the figure-eight in this way. Plumbing diagrams for torus knots will be shown in Section~\ref{sec:PlumbedTorus}.


We can glue together two plumbed knot complements $Y_1$ and $Y_2$ to produce a (closed) plumbed manifold. The simplest way to do so is to glue the boundaries so that the graph longitude $\lambda_1$ of $Y_1$ is glued to the graph longitude $\lambda_2$ of $Y_2$, and the meridian $\mu_1$ of $Y_1$ is glued to the meridian $-\mu_2$ of $Y_2$. (The minus sign on $\mu_2$ is needed so that the orientations are consistent.) We call this the {\em standard gluing}. A plumbing diagram for the resulting manifold $Y_1 \cup Y_2$ is obtained from those for $Y_1$ and $Y_2$ by identifying their distinguished vertices and adding up the weights there, as shown in Figure~\ref{fig:glue}. 


\begin {figure}
\begin {center}
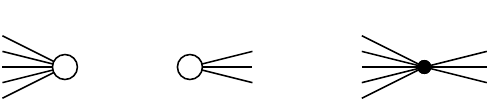
\caption {Standard gluing of knot complements, shown by plumbing diagrams. }
\label{fig:glue}
\end {center} 
\end {figure}




\subsection{$\Spinc$ structures on $3$-manifolds with boundary}
\label{sec:spincY}
We now discuss $\Spinc$ and relative $\Spinc$ structures on manifolds with boundary, and how they behave under gluing. Of course, we are mostly interested in the case of gluing plumbed knot complements.

Suppose, in general, that we have a compact oriented three-manifold $Y$ with boundary a surface $\Sigma$. Note that $\Spinc$ structures on $\Sigma$ are uniquely characterized by their first Chern class $c_1 \in H^2(\Sigma; \Z)$. The $\Spinc$ structure with $c_1=0$ is called trivial. This gives also a $\Spinc$ structure (still called trivial) on $\Sigma \times [0,1]$.  We define a {\em relative $\Spinc$ structure} on $Y$ to be a choice of extending the trivial $\Spinc$ structure on a collar neighborhood of $\del Y =\Sigma$ to a $\Spinc$ structure on $Y$. The space of relative $\Spinc$ structures, $\Spinc(Y, \del Y)$, is affinely isomorphic to $H^2(Y, \del Y) \cong H_1(Y)$. On the other hand, the space of (ordinary)  $\Spinc$ structures on $Y$, denoted $\Spinc(Y)$, is affinely isomorphic to $H^2(Y) \cong H_1(Y, \del Y)$. There is a natural map
\be
\label{eq:spincrel}
\Spinc(Y, \del Y) \to \Spinc(Y)
\ee
and this map is surjective, since (up to affine isomorphism) it comes from the long exact sequence of a pair
$$ \dots \to H^2(Y, \del Y) \to H^2(Y) \to H^3(\del Y) =0.$$

There is an action of $H_1(\Sigma)$ on $\Spinc(Y, \del Y)$ given as follows. Consider the map
\be
\label{eq:c}
c: H_1(\Sigma) \cong H^1(\Sigma) \to H^2(Y, \del Y)
\ee
obtained by composing Poincar\'e duality with the connecting homomorphism from the exact sequence in cohomology for $(Y, \del Y)$. We let $\gamma\in H_1(\Sigma)$ act on $a \in \Spinc(Y, \del Y)$ by 
\be
\label{eq:actc}
a \mapsto a + c(\gamma).
\ee

Next, suppose we have two three-manifolds $Y^+$ and $Y^-$ with boundaries $\del Y^+ = -\del Y^-= \Sigma$. We can glue them to get a closed three-manifold
$$Y = Y^- \cup_\Sigma Y^+.$$
A $\Spinc$ structure $a \in \Spinc(Y)$ has restrictions 
 $$a|_{Y_\pm} \in \Spinc(Y_\pm).$$ 
 We also have a map
\be
\label{eq:sp12}
 \Spinc(Y^-, \del Y^-) \oplus \Spinc(Y^+, \del Y^+) \to \Spinc(Y)
 \ee
given by gluing relative $\Spinc$ structures.  Given a pair $(a^-, a^+)$ in the preimage of $a$, note that $a^{\pm}$ are always lifts of $a|_{Y^{\pm}}$ under the map \eqref{eq:spincrel}. Further, \eqref{eq:sp12} is affinely isomorphic to the map in the Mayer-Vietoris sequence
$$ \dots \to H_1(\Sigma) \to H_1(Y^-) \oplus H_1(Y^+) \to H_1(Y) \to 0. $$
This shows that the map \eqref{eq:sp12} is surjective, and that any two pairs $(a^-, a^+)$ with the same image $a$ are related via an element $\gamma \in H_1(\Sigma)$, by 
\be
\label{eq:aa}
(a^-, a^+) \to (a^- + c(\gamma), a^+ + c(\gamma)).
\ee

\subsection{$\Spinc$ structures on plumbed knot complements}
\label{sec:spincplumbed}
Let us specialize the discussion from Section~\ref{sec:spincY} to plumbed knot complements. Let $\Gamma$, $v_0$, $\hGamma$, $Y=Y(\Gamma, v_0)$ and $\hat Y = Y(\hGamma)$ be as in Section~\ref{sec:PlumbedKnot}. We let $\Vert$ be the set of vertices in $\Gamma$ (including $v_0$). We will denote the elements of $\Vert$ by 
$$v_1, \dots, v_s,$$ with $v_s=v_0$. When considering vectors labeled by the vertices of $\Gamma$, we will put the entry for $v_i$ in position $i$. Thus, $\Z^s$ will be the set of such vectors with integer entries, and $\Z^{s-1} = \Z^{s-1} \times 0 \subset \Z^s$ the subset consisting of those vectors such that the entry labeled by $v_0$ is zero. 

As in Sections~\ref{sec:PlumbedDef} and \ref{sec:PlumbedSeries}, we let $\vdelta = (\delta_v)_{v \in \Vert}$ be the vector made of the degrees of the vertices, and $M$ be the linking matrix of $\Gamma$. If  $\hat M$ denotes the framing matrix for  for the closed manifold $\hat Y = Y(\hGamma)$, we have
\be
\label{eq:MhatM}
 M =  \left( \begin{array}{ccc|c}
& & &   * \\
& {\hat M} &     &\vdots \\
& &    &* \\
\hline
* & \dots  &  * & m_{v_0}\\
\end{array} \right).
\ee


Let us identify $\Z^s$ with $H_2(W(\Gamma))$, where $W(\Gamma)$ is the four-dimensional cobordism given by surgery along the link coming from the graph $\Gamma$. The basis element $\ve_i$ ($i=1, \dots, r$) of $\Z^s$ corresponds to the two-handle attached along the unknot component from the vertex $v_i$. From here, we get natural identifications
\be
\label{eq:relH}
 H^2(Y, \del Y) \cong H_1(Y)  \cong \Z^s/M \Z^{s-1}.
 \ee
The first isomorphism is Poincar\'e duality and the second is given by taking $\ve_i \in \Z^s$ into the element of $H_1(Y)$ represented by the belt sphere of the two-handle associated to $v_i$. In particular, the meridian $\mu$ and the (graph) longitude $\lambda$ of the knot $K \subset Y(\hGamma)$ correspond to the vectors $\ve_s$ and $M\ve_{s}$, respectively. Thus, we can identify 
$$ H_1(\Sigma)\cong \Span\langle \ve_s, M\ve_{s} \rangle \subset \Z^s$$
and the map $c: H_1(\Sigma) \to H^2(Y, \del Y)$ from \eqref{eq:c} with the composition
$$ \Span\langle \ve_s, M\ve_{s} \rangle \hookrightarrow \Z^s \twoheadrightarrow \Z^s/M\Z^{s-1}.$$
If we are interested in $H^2(Y) \cong H_1(Y, \del Y)$, this would be identified with 
$$\Z^s/(  \Span\langle  \ve_s, M\ve_{s} \rangle + M \Z^{s-1}) = \Z^s /(  \Span\langle  \ve_s \rangle + M \Z^{s}). $$

With regard to $\Spinc$ structures, we can identify those on the cobordism $W(\Gamma)$ with elements of $2\Z^s + \vdelta$, by considering their first Chern class. From here, in the spirit of \eqref{eq:spincid}, we get identifications
\be
\label{eq:spincYrel}
\Spinc(Y, \del Y) \cong  (2\Z^s + \vdelta)/(2M \Z^{s-1})
\ee
and
$$  \Spinc(Y) \cong  (2\Z^s + \vdelta)/( \Span\langle 2\ve_{s}, 2M\ve_s \rangle + 2M \Z^{s-1}).$$
The action of $H_1(\Sigma)$ on $\Spinc(Y, \del Y)$ is by adding multiples of $2\ve_{s}$ and $2M\ve_s$.

\subsection{$\Spinc$ structures and gluing}
\label{sec:spincglued}
Suppose we have two plumbed knot complements 
$$Y^-= Y(\Gamma^-, v_0^-), \ \ \ Y^+=Y(\Gamma^+, v_0^+)$$
with framing matrices $M^-$ and $M^+$. The standard gluing described at the end of Section~\ref{sec:PlumbedKnot} yields a closed three-manifold
$$ Y = Y^- \cup_\Sigma Y^+ = Y(\Gamma),$$
where $\Gamma$ is obtained from the graphs $\Gamma^-$ and $\Gamma^+$ as in Figure~\ref{fig:glue}. 

We order the vertices in $\Gamma^-$ as before, with $v_0^-$ being $v_s^-$. On the other hand, for convenience, we order the vertices in $\Gamma^+$ by starting with $v_0^+$. Then, the framing matrix $M$ for $\Gamma$ takes the form
$$ M =  \left( \begin{array}{ccc|c|ccc}
& & & * & & & \\
& {\hat M}^- &  & \vdots  & & 0 \\
& & & * & & & \\
\hline
*& \dots  & * & m_{v_0^-} + m_{v_0^+} &* & \dots & *\\[.5em]
\hline
& & & * & & & \\
& 0 & & \vdots & & {\hat M}^- & \\
& & & * & & & \\
\end{array} \right)$$
where $\hat M^-$ and $\hat M^+$ are the framing matrices for $\hGamma^- = \Gamma^-\setminus \{v_0^-\}$ and $\hGamma^+=\Gamma^+ \setminus \{v_0^+\}$. Thus, $M$ is ``almost block diagonal,'' being obtained by joining the diagonal blocks $M^-$ and $M^+$ at the central entry.

Under our identifications, the map 
$$ \Spinc(Y^-, \del Y^-) \oplus \Spinc(Y^+, \del Y^+) \to \Spinc(Y)$$
from \eqref{eq:sp12} is given by 
\be
\label{eq:aaa}
 \bigl([(a^-_1, \dots, a^-_{s-1}, a^-_s)], [(a^+_1, a^+_2, \dots, a^+_t)] \bigr) \mapsto [ (a^-_1, \dots, a^-_{s-1},  a^-_s+ a^+_1, a^+_2, \dots, a^+_t)].
 \ee
As a consistency check, let us see directly that the map in \eqref{eq:aaa} is well-defined, i.e., its effect does not depend on the representatives $\va^- = (a^-_1, \dots, a^-_{s-1}, a^-_s)$ and $\va^+= (a^+_1, a^+_2, \dots, a^+_t)$ that we chose for the relative $\Spinc$ structures. Let us write
$$ \va^- * \va^+ = (a^-_1, \dots, a^-_{s-1},  a^-_s+ a^+_1, a^+_2, \dots, a^+_t).$$
In view of \eqref{eq:aa}, we just need to check the effect of acting simultaneously on both $\va^-$ and $\va^+$ by an element of $H_1(\Sigma)$. It suffices to consider acting by the meridian $\mu^- = -\mu^+$ and the longitude $\lambda^-=\lambda^+$. (Recall that when identifying the boundaries of $Y^-$ and $Y^+$, we have a change of orientations in the meridians.) Acting by the meridian $\mu^-$ adds $2$ to $a_s^-$, and subtracts $2$ from $a_1^+$, so it leaves $\va^- * \va^+$ unchanged. Acting by the longitude means adding the vector $2M^- \ve_s^-$ to $\va^-$ and the vector $2M^+ \ve_1^+$ to $\va^+$, which results in adding $2M\ve_s$ to $\va^- * \va^+$. Adding an element in the image of $2M$ does not change the resulting $\Spinc$ structure on $Y$. 

\subsection{Dehn surgeries}
\label{sec:DehnS}
Suppose we have a manifold $Y$ with parametrized torus boundary, where the parametrization is specified by a longitude $\lambda$ and a meridian $\mu$. The result of Dehn surgery on $Y$ with coefficient $p/r \in \Q \cup \{\infty\}$ is  
$$Y_{p/r} := Y \cup (S^1 \times D^2),$$
where the gluing is by a diffeomorphism of the boundaries such that $p \mu + r \lambda$ gets identified with  $\mathit{pt} \times \del D^2$. The typical situation is that $\hat Y = Y_{\infty}$ is an integer homology three-sphere, $K \subset \hat Y$ is the knot whose complement gives $Y$, as in Section~\ref{sec:PlumbedKnot}, and $\lambda$ is the Seifert longitude. In that case, we write $\hat Y_{p/r}(K)$ for $Y_{p/r}$.

However, here we are concerned with the setting in which $Y = Y(\Gamma, v_0)$ is plumbed, and $\lambda$ is the graph longitude. (This may or may not agree with the Seifert longitude.) Then, Dehn surgery with coefficient $p/r$ can be thought of as the result of standard gluing between $\Gamma$ and the linear plumbing graph showed in Figure~\ref{fig:UnknotRational}, where the labels $k_1, \dots, k_s$ give the continued fraction expression
\be
 \frac{p}{r} \; = \; k_1 - \cfrac{1}{k_2 - \cfrac{1}{\ddots - \cfrac{1}{k_s}}}.
\label{contfract}
\ee

\begin {figure}
\begin {center}
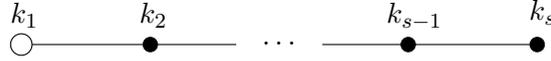
\caption {A plumbing diagram for the solid torus $\Spr$.}
\label{fig:UnknotRational}
\end {center}
\end {figure}

We denote by $\Spr$ the solid torus with parametrized boundary, as represented by the graph in Figure~\ref{fig:UnknotRational}. When writing vectors labeled by the vertices of the graph, let us keep the ordering $1, \dots, s$, that is, the distinguished vertex comes first. The relative $\Spinc$ structures on $\Spr$ form an affine space over $H_1(\Spr) =\Z$. 


On the other side, if the manifold $\hat Y$ is an integer homology sphere, then the relative $\Spinc$ structures on $Y$ form an affine space over $\Z$, too. The generator of $H_1(Y)$ can be taken to be the meridian and, in terms of \eqref{eq:spincYrel}, relative $\Spinc$ structures differ from each other by multiples of the vector $[(0, \dots, 2)]$. Under our identifications, the gluing map
$$ \Spinc(Y, \del Y) \oplus \Spinc(\Spr, \del \Spr) \to \Spinc(Y_{p/r})$$
corresponds to
\be
\label{eq:Zp}
 \Z \oplus \Z \to \Z/p, \ \ (a^-, a^+) \to a^- + a^+ \pmod{p}
 \ee
up to an affine transformation.

\newpage \section{An invariant of plumbed knot complements}
\label{sec:invariant}

In this section we will define an analogue of $\Zhat_a(q)$ for the manifolds with toroidal boundary introduced in Section~\ref{sec:plumbedKnot}.

\subsection{The definition}
\label{sec:ZhatKnot}
Let $Y = Y(\Gamma, v_0)$ be a plumbed knot complement. (We keep the notation from Section~\ref{sec:PlumbedKnot}.) We assume that the pair $(\Gamma, v_0)$ is weakly negative definite, in the following sense. (Compare with Definition~\ref{def:wnd}.) 

\begin{definition}
\label{def:WND}
Let $\Gamma$ be a plumbing tree and $v_0 \in \Vert$ a distinguished vertex. The pair $(\Gamma, v_0)$ is called {\em weakly negative definite} if if the corresponding matrix $M$ is invertible, and $M^{-1}$ is negative definite on the subspace of $\Z^s$ spanned by the non-distinguished vertices of degree $\geq 3$. 
\end{definition}

Pick a relative $\Spinc$ structure
$$ a \in \Spinc(Y, \del Y) = (2\Z^s + \vdelta)/(2M \Z^{s-1})$$
with a representative 
$$ \va = (a_v)_{v \in \Vert} \in 2\Z^s + \vdelta.$$

We define the invariant $\Zhat_a$ for $Y$ by analogy with the formula  \eqref{eq:newformula} in the closed case, in terms of contributions from the vertices and edges:
\be
\label{eq:newformula2}
\Zhat_a(Y; z, n, q) = (-1)^{\pi} q^{\frac{3\sigma-(\va, M^{-1}\va)}{4}} \sum_{n_v} \, \oint\limits_{|z_v|=1} \, \frac{dz_v}{2 \pi i z_v} \, \prod_{\text{Vert}} \big( \ldots \big)
\, \prod_{\text{Edges}} \big( \ldots \big)
\ee
Here, the sums and the integrals are only over the variables $n_v$ and $z_v$ with $v\neq v_0$. For $v_0$, we set
$$ n = n_{v_0}, \ \ \ z = z_{v_0},$$
and let $n$ and $z$ be part of the input in $\Zhat_a$.

In \eqref{eq:newformula2}, the factor for a vertex $v \neq v_0$ with framing coefficient $m_{v}$ is
\be
 q^{-m_vn_v^2-\frac{m_v}{4}-a_vn_v} z_v^{2m_v n_v+a_v} \left( z_v - \frac{1}{z_v} \right)^2
\label{vertexrulenew2}
\ee
and the factor for an edge $(u, v)$ is
\be
 q^{-2n_u n_v} \frac{z_u^{2n_v} z_v^{2n_u}}{\big( z_u - \tfrac{1}{z_u} \big) \big( z_v - \tfrac{1}{z_v} \big)}.
\label{edgerulenew2}
\ee
The contribution from the vertex $v_0=v_s$ is set to be
\be
q^{-m_{v_0} n^2-\frac{m_{v_0}}{4}-a_{v_0}n} z^{2m_{v_0} n+a_{v_0}} \left( z - \frac{1}{z} \right).
\label{v0rule}
\ee
This is almost as in \eqref{vertexrulenew2}, except we don't square the factor $z-1/z$. We choose this contribution to be as such with an eye toward the standard gluing, where $v_0$ gets joined to another unframed vertex.

If we prefer, we can also write 
$$\vell = 2M \vn + \va \in \Z^s$$
and express $\Zhat_a(Y; z, n, q)$ in a manner similar to the formulas \eqref{eq:plumbing1} or \eqref{eq:plumbing2} for closed manifolds. For example, the analogue of \eqref{eq:plumbing1} reads
\be
\Zhat_a(Y; z, n, q) = (-1)^{\pi} \left( z - \frac{1}{z} \right)^{1-\deg(v_0)} q^{\frac{3\sigma-\sum_v m_v}{4}}  \oint\limits_{|z_v|=1} \prod_{\substack{v\in \Vert \\ v \neq v_0} } \frac{dz_v}{2\pi i z_v} \left( z_v - \frac{1}{z_v} \right)^{2-\deg(v)} \! \! \! \! \Theta_a^{-M}(\vz)
\label{eq:plumbing1open}
\ee
where
\be
\label{eq:ThetaYopen}
\Theta_a^{-M}(\vz)= \sum_{\vell} q^{-\frac{(\vell, M^{-1}\vell)}{4}} \prod_{ v\in \operatorname{Vert}} z_v^{\ell_v}.
\ee
In \eqref{eq:plumbing1open} we only integrate and sum over $z_v$ with $v\neq v_0$, and the sum in \eqref{eq:ThetaYopen} is over $\vell =2M \vn + \va$ where the last entry of $\vn$ is fixed to be $n$. The formula \eqref{eq:plumbing1open} makes it clear that $\Zhat_a(z, n, q)$ does not depend on which  representative $\va$ we choose for the relative $\Spinc$ structure $a$. Also, the weakly negative condition ensures that $\Zhat_a$ is well-defined.

With regard to conjugation of $\Spinc$ structures, observe that we have the symmetry:
\be
\label{eq:conjsymm}
 \Zhat_{-a}(Y; z, q, n) = -\Zhat_a(Y; z^{-1}, q, -n).
 \ee
 
 \begin{proposition}
The series $\Zhat_a(q)$ defined in \eqref{eq:newformula2} is unchanged by the Neumann moves from Figure~\ref{fig:kirby}.
\end{proposition}

 \begin{proof}
 This is entirely similar to the proof of Proposition~\ref{prop:invariance}.
 \end{proof}
 
 \begin{remark}
 \label{rem:Mdeg}
 We can sometimes make sense of $\Zhat_a$ even when $M$ is not invertible. (Compare Remark~\ref{rem:0nd} in the closed case.) An important example of this will appear in Lemma~\ref{lem:KnotCo}.
 \end{remark}

\subsection{The solid torus}
Let us compute the invariants $\Zhat_a$ for the solid torus $\Spr$ represented by the diagram in Figure~\ref{fig:UnknotRational}. We assume that
$$p \neq 0,  \ \ r > 0, \ \ \gcd(p, r)=1.$$

Let us start with the simple case where $r=1$, $s=1$ and $k_1=p$. Then, $\Gamma$ is the graph with a single (distinguished) vertex, labeled by $p \neq 0$. There are infinitely many relative $\Spinc$ structures, labeled by $a \in \Z$, and represented by the vectors $\va = (2a)$.  Starting from \eqref{eq:plumbing1open}, we  compute
\be
\label{eq:Sp}
 \Zhat_{a}(\Sp; z, n, q)=-\operatorname{sign}(p) \cdot q^{\frac{3\operatorname{sign}(p)-p}{4}} q^{-\frac{1}{p}(pn+a)^2} z^{2pn+2a} \bigl( z - z^{-1} \bigr).
 \ee

Let us now assume that we have at least two vertices ($s \geq 2$). Recall from Section~\ref{sec:DehnS} that the $\Spinc$ structures on $\Spr$ are an affine space over $\Z$. Pick a $\Spinc$ structure $a$ represented by some vector $\va = (a_1, \dots, a_s)$. From the formula  \eqref{eq:plumbing1open} we get
$$ \Zhat_{a}(\Spr; z, n, q)= (-1)^{\pi} q^{\frac{3\sigma-\sum k_i}{4}} \Bigl( - \sum_{\vell \in \Lambda^+_{a,n}} q^{-\frac{(\vell, M^{-1}\vell)}{4}} + \sum_{\vell\in \Lambda^-_{a,n} } q^{-\frac{(\vell, M^{-1}\vell)}{4}}   \Bigr)$$
where
$$ \Lambda^{\pm}_{a,n} = \{ \vell = (2j+1, 0, \dots, 0, \pm 1) \mid j \in \Z, \ \vell = 2M \vn + \va \text{ for some } \vn = (n, n_2, \dots, n_s) \in \Z^s \}.$$
We have
$$ M = \begin{pmatrix} k_1 & 1 &   & & & \\
1 & k_2 & 1 & &   &  \\
 & 1 & k_3 &  & & \\
& & & \ddots & & \\
 & & & & k_{s-1} & 1\\
   & & & & 1 & k_s
 \end{pmatrix}.$$
The determinant of $M$ is either $p$ or $-p$. In fact, 
$$ \det(M) = (-1)^{\nu} |p| = (-1)^{s+\pi} \sign(p) p$$
where $\nu$ is the number of negative eigenvalues in $M$.

To get a hold of the exponents $-(\vell, M^{-1}\vell)/4$ for $\vell \in \Lambda^{\pm}_{a,n}$, we only need to know the entries of $M^{-1}$ in the positions $(1,1)$, $(1,s)$ and $(s, s)$. They are
$$ (M^{-1})_{11} = r/p, \ \ (M^{-1})_{1s} = \eps/{p}, \ \ (M^{-1})_{ss} = {D}/{p},$$
where
$$\eps = \sign(p) \cdot (-1)^{\pi+1} \in \{\pm 1\}$$ 
and $D$ is the determinant of the top-left $(s-1)\times (s-1)$ minor in $M$. Therefore, for $\vell = (2j+1, 0, \dots, 0, \pm 1)$, we find that
$$-\frac{(\vell, M^{-1}\vell)}{4} = -\frac{r}{p}\Bigl(j+\frac{1}{2}\pm \frac{\eps}{2r}\Bigr)^2 + \frac{1}{4pr} - \frac{\tau}{4p}.$$
Moreover, one can check that the values of $j$ that contribute to $\vell \in  \Lambda^{\pm}_{a,n}$ are those such that 
$$ rj \pm \frac{\eps}{2}= pn +\beta,$$
where $\beta \in  \Z + \frac{1}{2}$ is some value depending on $k_1, \dots, k_s$ and $a$. Also, as we keep $k_1, \dots, k_s$ fixed and vary the relative $\Spinc$ structure $a$, the values of $\beta$ vary affinely with respect to $a$.

From here we get 
\be
\label{eq:zzz}
 \Zhat_{a}(\Spr; z, n, q) = (-1)^{\pi} q^{\frac{3\sigma-\sum k_i}{4} + \frac{1}{4pr} - \frac{D}{4p}} \cdot \begin{cases}
 -  q^{-\frac{r}{p}(j+\frac{1}{2}+\frac{\eps}{2r})^2}z^{2j+1} & \text{if } \frac{pn+\beta}{r}-\frac{\eps}{2r}   = j \in \Z,\\
  q^{-\frac{r}{p}(j+\frac{1}{2}-\frac{\eps}{2r})^2} z^{2j+1} & \text{if } \frac{pn+\beta}{r}+\frac{\eps}{2r}  = j \in \Z,\\
0 & \text{otherwise}.
\end{cases}
\ee
In order for this formula to satisfy the symmetry \eqref{eq:conjsymm}, we see that we must have
$$ \beta = a -\frac{r}{2}.$$
This fixes a canonical identification 
\be
\label{eq:ahalf}
\Spinc(\Spr, \del \Spr) \cong \begin{cases} \Z & \text{for } r \text{ odd,}\\
\Z + \frac{1}{2} & \text{for } r \text{ even,}\end{cases}\ee
so that conjugation of relative $\Spinc$ structures corresponds to the map $a \to -a$.
Note that there exists a self-conjugate relative $\Spinc$ structure on $\Spr$ if and only if $r$ is odd.

Now, the formula \eqref{eq:zzz} can be written more simply as
\be
\label{eq:Zcases}
 \Zhat_{a}(\Spr; z, n, q) = \begin{cases} 
 \pm \sign(p) \cdot q^{\alpha(p,r)} \cdot q^{-\frac{(pn+a)^2}{pr}} 
  z^{2j+1} & \text{if } \frac{pn+a}{r}\mp \frac{1}{2r} - \frac{1}{2}  = j \in \Z,\\
0 & \text{otherwise}
 \end{cases}
\ee
where
\be
\label{eq:alpha}
 \alpha(p,r) =\frac{3\sigma-\sum k_i}{4} + \frac{1}{4pr} - \frac{D}{4p} \in \Q.
 \ee

Equation \eqref{eq:Zcases} comes with the caveat that the two contributions (from the two choices of sign) add up in the special case where we have $$ \frac{pn+a}{r}\mp \frac{1}{2r} - \frac{1}{2} \in \Z$$ for both choices of sign. This happens if and only if $r=1$. In that case, we can recover the formula \eqref{eq:Sp} from \eqref{eq:Zcases}, with $\alpha(p,1) = (3\sign(p)-p)/4$.

The expression $\alpha(p,r)$ from \eqref{eq:alpha} is independent of the presentation of $p/r$ as a continued fraction. Indeed, one can check it is invariant under the Neumann moves. Better yet, this expression can be related to the Dedekind sums 
$$ s(p,r) = \sum_{i=1}^{r-1} \Bigl(\Bigl(\frac{i}{r}\Bigr)\Bigr) \Bigl(\Bigl(\frac{ip}{r}\Bigr)\Bigr), \ \ \ \text{where} \ ((x)) := \begin{cases}x-\lfloor x \rfloor - 1/2 & \text{if } x \in \R - \Z,\\
0 & \text{if } x \in \Z.\end{cases}$$
Using Barkan's evaluation \cite{Barkan} for Dedekind sums in terms of continued fractions, we find that
\be
\alpha(p,r) = 3\sign(p) \Bigl( s(p,r) + \frac{1}{4} \Bigr) - \frac{p}{4r}.
\ee

\subsection{Gluing formula}
\label{sec:gluing}
In this section we prove Theorem~\ref{thm:glue}. Consider the standard gluing of two plumbed knot complements $Y^-$ and $Y^+$. We use the notation from  Section~\ref{sec:spincglued}, and assume that $(\Gamma^-, v_0^-)$, $(\Gamma^+, v_0^+)$ and $\Gamma$ are weakly negative definite. Pick
$$ a^- \in \Spinc(Y^-, \del Y^-), \ \ a^+ \in  \Spinc(Y^+, \del Y^+)$$
with representatives $\va^-$ and $\va^+$ such that the entries corresponding to the unframed vertices are zero. Let $a$ be the image of $(a^-, a^+)$ in $\Spinc(Y)$ under the map \eqref{eq:sp12}, and let $\va$ be the representative of $a$ given by joining $\va^-$ and $\va^+$, as in \eqref{eq:aaa}. 

Now, from the formulas \eqref{eq:newformula} and \eqref{eq:newformula2} for the $\Zhat$ invariants, we  deduce the desired gluing formula:
\be
\label{eq:glue}
\boxed{\Zhat_a(Y; q) = (-1)^{\tau} q^{\xi} \sum_n \oint_{|z|=1} \frac{dz}{2\pi i z} \Zhat_{a^-}(Y^-; z, n, q)  \Zhat_{a^+}(Y^+; z, n, q)}
\ee
where
$$\tau = \pi(M) - \pi(M^-) - \pi(M^+)$$
and
$$\xi =  \frac{3}{4} \bigl (\sigma(M) -\sigma(M^-) - \sigma(M^+) \bigr)-  \frac{1}{4}\bigl((\va, M^{-1} \va)- (\va^-, (M^{-})^{-1}\va^-) - (\va^+, (M^{+})^{-1}\va^+)  \bigr)  \in \Q.$$
One can check that $\xi=\xi(a^-, a^+)$ depends only on the $\Spinc$ structures $a^-$ and $a^+$, and not on their representatives $\va^-$ and $\va^+$.

\subsection{TQFT properties}
\label{sec:TQFT}
Let $Y = Y(\Gamma, v_0)$ with $(\Gamma, v_0)$ weakly negative definite.  From now on we will assume that the distinguished vertex $v_0$ has degree $1$. Then, for any $a \in \Spinc(Y, \del Y) \cong  (2\Z^s + \vdelta)/(2M \Z^{s-1})$ represented by a vector $\va \in 2\Z^s + \vdelta$, we have that $a_{v_0}$ must be odd. Looking at the factors containing the variable $z$ in \eqref{edgerulenew2} and \eqref{v0rule}, we see that $z$ always appears with an odd exponent in $\Zhat_a(Y; z, n, q)$. Thus, it is convenient to introduce a new variable
 $$x = z^2.$$
We can write
\be
\label{eq:bmn}
\Zhat_a(Y; z, n, q)= \sum_{m \in \Z+\frac{1}{2}}  b(m,n) x^m, \ \ b(m,n) \in \k.
\ee
Here, $\k$ is the Novikov-type field consisting of power series of the form
$$ \sum_{s \in \Q} b_\omega q^\omega, \ \ b_\omega \in \Q$$
such that the set $\Omega = \{\omega \mid b_\omega \neq 0\} \subset \Q$ is bounded below, and its projection to $\Q/\Z$ is finite.

\begin{remark}
Physically, $m$ and $n$ are interpreted as electric and magnetic fluxes. The basis with discrete variables $(m,n)$ is sometimes called the {\em flux basis}, whereas the choice of variables $(x,n)$ is called the {\em fugacity basis}. 

Geometrically, the variable $m$ corresponds to the meridian of the boundary torus, and the variable $n$ to the longitude. This can be seen in terms of the $H_1(T^2)$ action that will be defined in  Section~\ref{sec:action}.
\end{remark}

\begin{definition}
A function 
$$b: (\Z+\tfrac{1}{2}) \times \Z \to \k, \ \ b(m, n) = \sum b(m, n, \omega) q^\omega$$ is said to be {\em well-behaved} if the set $\Omega = \{ \omega \in \Q \mid b(m, n, \omega) \neq 0 \text{ for some } m, n \}$ is bounded below, and its projection to $\Q/\Z$ is finite. 
The space of well-behaved functions is denoted by $\V$.
\end{definition}

Observe that $\V$ is a vector space over $\k$. Furthermore, on $\V$ we have a (partially defined) bilinear pairing $\langle \cdot , \cdot \rangle$ given by
\be
\label{eq:paired}
 \langle b^-, b^+\rangle = \sum_{m \in \Z+\frac{1}{2}} \sum_{n \in \Z}  b^-(m, n) b^+(m, n)  \in \k.
\ee

Observe that, for a closed negative definite plumbed manifold $Y$, the invariants $\Zhat_a$ take values in $\k$. Moreover, for a negative definite plumbed manifold $Y = Y(\Gamma, v_0)$ with torus boundary, the function $b$ giving the coefficients of $\Zhat_a(Y;z, n, q)$ in \eqref{eq:bmn} is well-behaved. Hence, we can view the invariants associated to $Y$ and $a \in \Spinc(Y, \del Y)$ as elements
$$ \Zhat_a(Y) \in \V.$$

The gluing formula \eqref{eq:glue} reads
\be
\label{eq:glue2}
 \Zhat_a(Y; q) = (-1)^{\tau} q^{\xi} \left\langle  \Zhat_{a^-}(Y^-),  R \Zhat_{a^+}(Y^+)\right \rangle,
 \ee
 where the map $R$ corresponds to reversing the orientation of the meridian:
 $$ R: \V \to \V, \ \ (Rb)(m, n) = b(-m,n).$$ 

We see here some of the structure of a $2+1$ topological quantum field theory (TQFT), decorated by $\Spinc$ structures. In this theory, to the torus $T^2$ we associated the vector space $\V$, and to $3$-manifolds with boundary $T^2$ (equipped with relative $\Spinc$ structures) we associate elements of $\V$. To closed $3$-manifolds (equipped with $\Spinc$ structures) we associate elements of the underlying vector field $\k$. Our invariants satisfy a gluing formula, given by the bilinear pairing as in \eqref{eq:glue2}.

The mapping class group of the torus is the modular group $\SL$. This acts on $\V$ by pre-composition: if $b \in \V$, $X \in \SL$ and 
$$ X\begin{pmatrix} m \\ n \end{pmatrix} = \begin{pmatrix} m' \\ n' \end{pmatrix},$$
then $X \cdot b \in \V$ is such that
$$ (X \cdot b)(m, n) = b(m', n').$$ 

While our discussion so far has been limited to trees with a single unframed vertex (i.e., knot complements), one can define similar invariants for negative definite plumbing trees with several unframed vertices, all of degree one. These correspond to link complements. If we have such a graph, with $k$ unframed vertices, we will have invariants of the form
$$ \Zhat_a(Y; z_1, \dots, z_k, n_1, \dots, n_k, q),$$
labeled by relative $\Spinc$ structures, and with one pair of variables $(z_i, n_i)$ for each boundary component. Thus, the invariants take values in $\V^{\otimes k}$. If we glue two such manifolds along a single boundary component, we have a gluing formula similar to \eqref{eq:glue2}.

Of course, we do not yet have a TQFT.  We have only defined our invariants for three-manifolds coming from negative definite plumbing graphs, and for the surface of genus $1$. In Section~\ref{sec:recursion} we will also construct invariants for other knot complements in $S^3$, and for some surgeries on knots. We hope that in the future our theory will be extended to a true ($\Spinc$-decorated) TQFT in $2+1$ dimensions.


\subsection{The $\Z^2$ action}
\label{sec:action}
Looking at the gluing formula \eqref{eq:glue2}, it is interesting that the right-hand side gives the same answer, for any choice of pair $(a^-, a^+)$ that maps to a fixed $a \in \Spinc(Y)$. We can explain this fact from the point of view of a $\Spinc$-decorated TQFT, by introducing the following action of $H_1(T^2) \cong \Z^2$ on the vector space $\V$.

We will denote the action by $A$. We let the generator $(1, 0) \in \Z^2$ correspond to the meridian $\mu \in H_1(T^2)$, and the generator $(0,1)$ to the longitude $\lambda \in H_1(T^2)$. We let these generators act on $\V$ by
\be
\label{eq:Tmu}
 (A_\mu b)(m, n) =  b(m-1, n)
 \ee
 and
\be
\label{eq:Tlambda}
 (A_\lambda b)(m, n) = b(m, n+1).
 \ee

If $\gamma = (u, v) \in \Z^2 \cong H_1(T^2)$, we simply set
$$ A_{\gamma} = A_{\mu}^u A_{\lambda}^v.$$

It is easy to see that this action is orthogonal with respect to the pairing \eqref{eq:paired}:
\be
\label{eq:Tsymm}
  \langle  b^-, b^+\rangle = \langle  A_{\gamma} b^-, A_{\gamma} b^+\rangle.
  \ee

For a three-manifold $Y$ with boundary $\del Y = T^2$, recall that we also have the action \eqref{eq:actc} of $H_1(T^2)$ on $\Spinc(Y)$: $a \mapsto a + c(\gamma)$. Let us describe how the invariants $\Zhat_a(Y)$ behave with respect to this action. For simplicity, we will not keep track of overall factors of $q$. For $b, b' \in \V$, we will write 
$$ b \simeq b' $$
if, for all $m \in \Z + \tfrac{1}{2}, n \in \Z$, there exists $\beta(m,n) \in \Q$ such that
$$ b'(m, n)=  q^{\beta(m,n)} b(m, n).$$
We will also use the symbol $\simeq$ to denote elements of $\k$ that differ by an overall factor of $q^\beta$ for some $\beta \in \Q$.

\begin{proposition}
\label{prop:TQFTaction}
Let $Y=Y(\Gamma, v_0)$ be a negative definite plumbed manifold with torus boundary. Then, for any $a \in \Spinc(Y, \del Y)$ and $\gamma \in H_1(T^2)$, we have
 \be
\label{eq:translateSpinc}
 \Zhat_{a+c(\gamma)}(Y) \simeq A_\gamma \Zhat_a(Y).
 \ee
\end{proposition}

\begin{proof}
It suffices to check this when $\gamma$ is one of the two generators $\mu$ and $\lambda$.

Suppose $\gamma$ is the meridian $\mu$. Pick a representative $\va = (a_1, \dots, a_s) \in 2\Z^s + \vdelta$ of $a$. Recall from Section~\ref{sec:spincplumbed} that the action of $\mu$ is given by adding  $2\ve_{s}=(0, \dots, 0, 2)$ to the vector $\va$, that is, changing $a_{v_0}=a_s$ to $a_s + 2$. Let us use the formula \eqref{eq:newformula2} or $\Zhat_a(Y; z, n, q)$ in terms of contributions from edges and vertices. Observe that $a_{v_0}$ appears in this formula only through the overall factor $q^{- \frac{(\va, M^{-1}\va)}{4}}$ and through the contribution from $v_0$ given by \eqref{v0rule}. Adding $2$ to $a_{v_0}$ multiplies $\Zhat_a(Y; z, n, q)$ by a factor of $q^{-2n} z^2 = q^{-2n}x$. Up to a power of $q$, this corresponds to acting by $A_{\mu}$, according to \eqref{eq:Tmu}.

Next, suppose $\gamma$ is the graph longitude $\lambda$. Pick a representative $\va = (a_1, \dots, a_s) \in 2\Z^s + \vdelta$ of $a$. As noted in Section~\ref{sec:spincplumbed}, the action of $\lambda$ is given by adding  $2M\ve_{s}$ to the vector $\va$. This time we will use the formula \eqref{eq:plumbing1open} for $\Zhat_a(Y; z, n, q)$. The Theta function in \eqref{eq:ThetaYopen} is defined by summing over
$$ \vell =2M \vn + \va = 2M(\vn - \ve_s) + (\va + 2M\ve_{s}).$$
Therefore, adding $2M\ve_{s}$ to $\va$ yields the same result, provided we replace $n$ by $n-1$. This corresponds to the action of $A_{\lambda}$ in \eqref{eq:Tlambda}. Note that in this case we have 
 \be
\label{eq:translatelambda}
 \Zhat_{a+c(\lambda)}(Y) = A_\lambda \Zhat_a(Y)
 \ee
on the nose, rather than up to a power of $q$.
\end{proof}

Let us now go back to the gluing formula \eqref{eq:glue2}. Recall from \eqref{eq:aa} that any two choices of $(a^-, a^+)$ that produce the same $a \in \Spinc(Y)$ are related by
\be
\label{eq:aa2}
(a^-, a^+) \to (a^- + c(\gamma), a^+ + c(\gamma)).
\ee
for some $\gamma \in H_1(T^2)$. Using Proposition~\ref{prop:TQFTaction} and the orthogonality property  \eqref{eq:Tsymm}, we have the following consistency check:
$$\left\langle \Zhat_{a^-+ c(\gamma)}(Y^-), R \Zhat_{a^+ + c(\gamma)}(Y^+) \right \rangle \simeq\left \langle A_\gamma \Zhat_{a^-}(Y^-), RA_{R(\gamma)}  \Zhat_{a^+}(Y^+) \right \rangle =$$
$$= 
\left \langle A_\gamma \Zhat_{a^-}(Y^-),  A_{\gamma}R  \Zhat_{a^+}(Y^+) \right \rangle =
\left\langle  \Zhat_{a^-}(Y^-),  R \Zhat_{a^+}(Y^+)\right \rangle.$$
Here, by $R(\gamma)$ we meant the image of $\gamma$ under the map that reverses the orientation of the meridian
$$ R : H_1(T^2) \to H_1(T^2), \ \ R(u, v) = (-u, v).$$
When applying Proposition~\ref{prop:TQFTaction} to $Y^+$, we got the action $A_{R(\gamma)}$ instead of $A_{\gamma}$, because in the standard gluing we identified the meridian of $\del Y^-$ (which we took to be $\mu=(1, 0) \in H_1(T^2)$) with the reverse of the meridian of $\del Y^+$.

\subsection{Eliminating some variables}
\label{sec:eliminate}
The invariants $\Zhat_a$ for plumbed knot complements involve several variables: the relative $\Spinc$ structure $a$, as well as $z=x^{1/2}$, $n$ and $q$ (or, if we prefer, $a$, $m$, $n$ and $q$). However, because of the symmetry \eqref{eq:translateSpinc}, we can reduce these variables by two, using the action of $H_1(\del Y) \cong \Z^2$. A well-known fact in three-dimensional topology (a consequence of Poincar\'e duality and the long exact sequence of a pair) says that the kernel $S$ of the map $H_1(\del Y) \to H_1(Y)$ has rank one (is a copy of $\Z$ in $\Z^2$). If $\gamma \in S$, then
$$ a + c(\gamma) = a \in \Spinc(Y, \del Y).$$
Thus, if we apply Proposition~\ref{prop:TQFTaction} to elements of $H_1(\Sigma)$ that are in $S$, we obtain a symmetry of the invariants $\Zhat_a(Y; z, n, q)$ for fixed $a$. On the other hand, if we apply it to elements of $H_1(\Sigma)$ that are not in $S$, we get a relation between the invariants $ \Zhat_a(Y; z, n, q)$ for different relative $\Spinc$ structures $a$.

There is also the dependence of $\Zhat_a$ on the way we parametrized the boundary. We usually fix what the meridian is, since this determines the closed-up manifold $\hat Y$ and the knot $K \subset \hat Y$. We may want to vary the graph longitude $\lambda$, which we can do by changing the value of $m_{v_0}$. Looking at the formula \eqref{eq:newformula2}, we see that adding $1$ to $m_{v_0}$ changes the contribution \eqref{v0rule} from $v_0$ by 
$ q^{-n^2 - \frac{1}{4}} x$
and also changes the overall factor $q^{- \frac{(\va, M^{-1}\va)}{4}}$. In any case, knowing the invariants for one choice of $m_{v_0}$ allows us to know them for all the other choices.

Let us make this more concrete in the case when $\hat Y$ is an integral homology sphere, so that $Y = \hat Y \setminus \nu K$ has the homology of a solid torus. Then, the space $\Spinc(Y, \del Y)$ is an affine copy of $H^2(Y, \del Y) \cong H_1(Y) \cong \Z$. The kernel $S \subset H_1(\Sigma)$ is the span of the Seifert longitude $\lambdasf$. Acting by the meridian $\mu$ (which is not in $S$) allows us to determine the invariants for all $a \in \Spinc(Y, \del Y)$ structures from the invariant for any $a$. Also, acting by the graph longitude $\lambda$ (which is some combination of $\lambdasf$ and $\mu$) tells us that to know the invariants $\Zhat_a(Y; z, n, q)$ it suffices to know them for $n=0$, when they give a power series in $z=x^{1/2}$ and $q$. Finally, knowing the invariants for one choice of graph longitude allows us to know them for all possible choices.

A natural choice of graph longitude is the Seifert longitude itself. However, this is exactly the case where the matrix $M$ is not invertible, because the result $Y_0$ of zero surgery on $K$ has $b_1(Y_0) =1$, and 
$$ H_1(Y_0; \Z) \cong \Z^s/M \Z^s.$$
Therefore, in that case $(\Gamma, v_0)$ is not weakly negative definite in the sense of Definition~\ref{def:WND}, so we do not expect all the invariants $\Zhat_a$ to be well-defined. Nevertheless, they are defined for $a=0$, in the following setting. (Compare Remark~\ref{rem:Mdeg}.)

\begin{lemma}
\label{lem:KnotCo}
Let $Y = Y(\Gamma, v_0)$ for some plumbed tree $\Gamma$ with distinguished vertex $v_0$. Suppose that $\hat Y$ is an integer homology sphere, $\hat \Gamma$ is negative definite, and that the graph longitude is the Seifert longitude. Let $a=0 \in \Spinc(Y, \del Y)$ be self-conjugate. Then, any representative $\va$ of $a$ is of the form $\va=M\vb$ for some $\vb \in \Z^{s-1}$, and the invariants $\Zhat_0(Y; z, n, q)$ are well-defined by the formula \eqref{eq:newformula2}, where in the power of $q$ we use the exponent
$ -(\vb, M\vb)/4$
instead of $- (\va, M^{-1}\va)/4$. 
\end{lemma}

\begin{proof}
Let $\va \in 2\Z^s + \vdelta$ be a representative for $a$. Since $a$ is self-conjugate, we have that the classes of $\va$ and $-\va$ modulo $2M\Z^{s-1}$ are the same, so $2\va \in 2M \Z^{s-1}$. It follows that we can write $\va = M \vb$ for some $\vb \in M \Z^{s-1}$.

Recall that, in the case where the pair $(\Gamma, v_0)$ is (weakly) negative definite, by setting $\vell =  2M \vn + \va$, we saw that the formula \eqref{eq:newformula2} for $\Zhat_a$ is equivalent to the formula \eqref{eq:plumbing1open} involving the theta function from \eqref{eq:ThetaYopen}. The same is true in our setting, with the theta function being
\be
\label{eq:Theta3}
\Theta_a^{-M}(\vz)= \sum_{\vell} q^{-\frac{(2\vn+\vb, M(2\vn+\vb))}{4}} \prod_{ v\in \operatorname{Vert}} z_v^{\ell_v}
\ee
where $$\vell = 2M \vn + \va = M (2\vn + \vb).$$

The exponent of $q$ in \eqref{eq:Theta3} is
$$ - (\vn, M\vn) - (\vn, M \vb) - \frac{(\vb, M\vb)}{4}.$$

To see that the formula \eqref{eq:plumbing1open} makes sense, it suffices to check that:
\begin{enumerate}[(i)]
\item As we vary $\vn \in \Z^s$ such that $n = (\vn, \ve_s)$ is fixed, the expression $(\vn, M\vn) + (\vn, M \vb)$ is bounded above;
\item As we vary $\vn \in \Z^s$ such that $n= (\vn, \ve_s)$ and $(\vn, M\vn) + (\vn, M \vb)$ are fixed, there are only finitely many possible values for $\vell = 2M\vn + \va.$
\end{enumerate}

By writing $\vn = \vw + n \ve_s$ with $\vw \in \Z^{s-1}$, we see that
$$ (\vn, M\vn) + (\vn, M \vb)= (\vw, M \vw) + (\vw, 2Mn \ve_s + M \vb) + (n\ve_s, nM\ve_s + M\vb)$$
is the sum of a quadratic, a linear, and a constant term in $\vw$. Since the quadratic term is negative definite (because $\hat \Gamma$ is negative definite), the desired claims (i) and (ii) follow.

In fact, it is worth noting that we can prove a stronger claim than (i):

\smallskip
(iii)  As we vary $\vn \in \Z^s$ arbitrarily, the expression $(\vn, M\vn) + (\vn, M \vb)$ is bounded above.
\smallskip

Indeed, since $M$ is symmetric, degenerate, and its restriction to $\Z^{s-1} \subset \Z^s$ is negative definite, it follows that $M$ admits an eigenbasis $\vphi_1, \dots,\vphi_s$ with eigenvalues 
$$\lambda_1=0, \lambda_2, \dots, \lambda_s\  \text{ with }\ \lambda_i < 0 \text{ for } i \geq 2.$$
Writing $\vn = \sum c_i \vphi_i$, we get
$$ (\vn, M\vn) + (\vn, M \vb) = \sum_{i \geq 2}\bigl( \lambda_i c_i^2 + c_i(\vphi_i, M \vb)\bigr) + c_1(\vphi_1, M\vb).$$
The terms in the sum are bounded above (because $\lambda_i< 0$) and the last term is zero, because $M$ is symmetric and $M\vphi_1=0$. This proves (iii).
\end{proof}

\subsection{Simpler knot invariants} 
Let $Y = Y(\Gamma, v_0)$ for some plumbed tree $\Gamma$ with distinguished vertex $v_0$. This gives a closed-up manifold $\hat Y = Y(\Gamma -v_0)$ and a knot $K \subset \hat Y$. Suppose that $H_1(\hat Y; \Z)=0$, the graph $\hat \Gamma$ is negative definite, and that $\lambda = \lambdasf$, as in Lemma~\ref{lem:KnotCo}. In view of that lemma and the discussion preceding it, in order to know the invariants $\Zhat_a(Y; z, n, q)$ (for any relative $\Spinc$ structure, and even for other choices of graph longitude), it suffices to know them for 
$$a=0, \ \ n=0.$$
In fact, observe that in this case, the kernel of the map $H_1(\del Y) \to H_1(Y)$ is spanned by $\lambda$, so acting by $\lambda$ relates the invariants $\Zhat_0(Y; z, n, q)$ to each other. By \eqref{eq:translatelambda}, we have
$$ \Zhat_0(Y; z, n, q) = A_{\lambda} \Zhat_0(Y; z, n, q) = \Zhat_0(Y; z, n+1, q),$$
so $\Zhat_0(Y; z, n, q)$ is independent of $n \in \Z$.
We denote
\be
\label{eq:FK}
 \boxed{F_K(x, q) := \Zhat_0(Y; x^{1/2}, n, q)}
 \ee
This is the simplest knot invariant associated to $K$ from the 3d $\cN = 2$ theory $T[Y]$. It involves just two variables $x$ and $q$. We have 
$$ F_K \in 2^{-c} q^{\Delta} \Z[x^{1/2}, x^{-1/2}][q^{-1},q]],$$
for some $c \in \Z_+$ and $\Delta \in \Q$. Here, $\Z[x^{1/2}, x^{-1/2}][q^{-1}, q]]$ denotes the ring of Laurent power series in $q$ with coefficients in the polynomial ring $\Z[x^{1/2}, x^{-1/2}]$. The fact that there is an overall lower bound on the exponents of $q$ follows from the claim (iii) that was established in the proof of Lemma~\ref{lem:KnotCo}. The fact that, if we fix the power of $q$, then the coefficient is a Laurent polynomial in $x^{1/2}$ (rather than a power series) is a consequence of the claim (ii) from the same proof.

Since $a$ is self-conjugate, the symmetry \eqref{eq:conjsymm} gives
\be
\label{eq:Fantisymm}
F_K(x, q) = - F_K(x^{-1}, q).
\ee
Therefore, we can write
\be
\label{eq:FKf}
 F_K(x, q) = \frac{1}{2} \sum_{m \geq 1} f_m (q) \cdot (x^{\frac{m}{2}} - x^{-\frac{m}{2}}),
 \ee
where $f_m(q)$ are (roughly) Laurent power series in $q$ or, more precisely, elements of the field $\k$ defined in Section~\ref{sec:TQFT}. Moreover, the exponents of $x^{1/2}$ that appear in $F_K$ are all odd, so the sum in \eqref{eq:FKf} can be taken over $m=2j+1$ only, with $j \geq 0$.
 
It is sometimes convenient to use a different normalization
\be
\label{eq:fK}
 \boxed{f_K(x, q)  := \frac{F_K(x, q)}{x^{1/2} - x^{-1/2}}}
 \ee
In terms of the coefficients $f_{2j+1}(q)$, we have
$$ f_K(x, q) = \frac{1}{2} \sum_{j \geq 0} f_{2j+1} (q) \cdot (x^{-j} + x^{-j+1} + \dots + x^{j-1} + x^j).$$
The function $f_K$ satisfies the symmetry 
\be
\label{eq:fsymm}
f_K(x, q) = f_K(x^{-1}, q).
\ee
This symmetry is expected from the physical interpretation of $f_K(x,q)$ as a count of BPS states for the theory $T[Y]$ on the knot complement. Indeed, there $x$ and $x^{-1}$ are interpreted as the eigenvalues of the holonomy around the meridian of an $\sl$ flat connection.

\begin{example}
If $K=U \subset S^3$ is the unknot, the condition $\lambda = \lambdasf$ can be ensured by taking the graph with a single vertex labeled $0$, so that $Y$ is the solid torus $\mathbb{S}_0$ in the notation from Section~\ref{sec:DehnS}. Observe that the formula \eqref{eq:Sp} cannot be applied to $p=0$ and arbitrary $a$, since we cannot make sense of the exponent $(pn+a)^2/p$. However, when $a=0$ that exponent can be set to $0$, and $\Zhat_0(\mathbb{S}_0)$ is well-defined, in agreement with Lemma~\ref{lem:KnotCo}. We get
$$ F_U(x, q) = x^{1/2} - x^{-1/2}, \ \ \  f_U(x, q) = 1.$$
\end{example}

\subsection{The Dehn surgery formula}
\label{sec:surgeryformula}
Let us recall from \eqref{eq:Laplace} that the Laplace transform $\CL_{p/r}^{(a)}$, applied to a power series in $x$ and $q$, takes a monomial $x^u \cdot q^v$ to $q^{-\frac{r}{p}u^2} \cdot q^v$, provided $ru - a \in p\Z$, and to zero otherwise. 

We will identify the relative $\Spinc$ structures on the solid torus $\Spr$ with elements in $\Z + \frac{r+1}{2}$
as in \eqref{eq:ahalf}. For $p/r$ surgery on a knot $K$ in a homology sphere $\hat Y$, we identify relative $\Spinc$ structures on $Y = \hat Y \setminus \nu K$ with $\Z$, by mapping the self-conjugate structure to $0$. In view of \eqref{eq:Zp}, the $\Spinc$ structures on the surgery $Y_{p/r}$ will be canonically identified with elements
$$ a \in \Z + \frac{r+1}{2} \pmod{p \Z}.$$

We will consider Laplace transforms $\CL_{p/r}^{(a)}$ for this kind of values of $a$. When $p=\pm 1$, since there is only one possible value of $a$, we will write $\CL_{p/r}$ for $\CL_{p/r}^{(a)}$.

Recall that Theorem~\ref{thm:Dehn}, as advertised in the Introduction, relates the invariants of a knot and its surgeries by the formula:
$$\Zhat_a(Y_{p/r})= \eps  q^{d} \cdot \CL_{p/r}^{(a)} \left[ (x^{\frac{1}{2r}} - x^{-\frac{1}{2r}}) F_K (x,q) \right],$$
which we could also write as
\be
\label{eq:Dehnf}
\Zhat_a(Y_{p/r})= \eps  q^{d} \cdot \CL_{p/r}^{(a)} \left[ (x^{\frac{1}{2r}} - x^{-\frac{1}{2r}}) (x^{\frac{1}{2}} - x^{-\frac{1}{2}})f_K (x,q) \right].
\ee

\begin{proof}[Proof of Theorem~\ref{thm:Dehn}]
As explained in Section~\ref{sec:DehnS}, the surgery $Y_{p/r}$ is obtained by standard gluing from the knot complement $Y$ and the solid torus $\Spr$. We apply the gluing formula \eqref{eq:glue} to these two pieces:
\begin{align}
\label{Zah}
 \Zhat_a(Y_{p/r}) &= (-1)^{\tau} q^{\xi} \sum_n \oint_{|z|=1} \frac{dz}{2\pi i z} \Zhat_{0}(Y; z, n, q)  \Zhat_{a}(\Spr; z, n, q)\\ \notag
 &=(-1)^{\tau} q^{\xi} 
\oint_{|z|=1} \frac{dz}{2\pi i z}F_K(z^2,q) \cdot \sum_n \Zhat_{a}(\Spr; z, n, q). 
 \end{align}
since $\Zhat_0(Y; z, n, q) = F_K(z^2,q)$ is independent of $n$. The values of $\Zhat_{a}(\Spr; z, n, q)$ were computed in \eqref{eq:Zcases}. From there we find that
\be
\label{eq:sumn}
 \sum_{n} \Zhat_{a}(\Spr; z, n, q) = 
 \sign(p)  q^{\alpha(p,r)} \Bigl( \sum_{j \in J_+}q^{-\frac{r}{p}(j + \frac{1}{2} + \frac{1}{2r})^2}z^{2j+1}-\sum_{j \in J_-}  q^{-\frac{r}{p}(j + \frac{1}{2} - \frac{1}{2r})^2}  z^{2j+1} \Bigr)
 \ee
where
$$ J_{\pm} = \bigl\{j \in \Z \ \big | \ jr \equiv a - \tfrac{r\pm 1}{2} \!\!\! \pmod{p} \bigr\}.$$

Let us write
$$ f_K(x, q) = \sum c_{t,v} x^t q^v$$
so that
$$ F_K(x,q) =  \sum c_{t,v} x^{t+\frac{1}{2}} q^v -  \sum c_{t,v} x^{t-\frac{1}{2}} q^v.$$
Plugging \eqref{eq:sumn} into \eqref{Zah}, we see that the integral picks up the monomials with $2t\pm 1 = - (2j+1)$. Therefore, we obtain
\begin{multline}
\label{eq:ZY}
 \Zhat_a(Y_{p/r})=  \sign(p) (-1)^{\tau}q^{\xi + \alpha(p,r)} \sum_v q^v\cdot  \Bigl(\sum_{-t-1 \in J_+} c_{t, v} q^{-\frac{r}{p}(t + \frac{1}{2} - \frac{1}{2r})^2} -  \sum_{-t-1 \in J_-} c_{t, v}q^{-\frac{r}{p}(t + \frac{1}{2} + \frac{1}{2r})^2} \\
 - \sum_{-t \in J_+} c_{t, v} q^{-\frac{r}{p}(t - \frac{1}{2} - \frac{1}{2r})^2} + \sum_{-t \in J_-} c_{t, v} q^{-\frac{r}{p}(t - \frac{1}{2} + \frac{1}{2r})^2}
 \Bigr)
\end{multline}
The conditions on the summation indices all translate into asking for the term $u = t \pm \frac{1}{2} \pm \frac{1}{2r}$ that appears in the corresponding exponent $q^{-\frac{r}{p}u^2}$ to satisfy
$$ ru + a \in p \Z.$$
 Since
 $$ (x^{\frac{1}{2r}} - x^{-\frac{1}{2r}})(x^{\frac{1}{2}} - x^{-\frac{1}{2}}) f_K (x,q) = -\sum c_{t, v} q^v \cdot \bigl( x^{t+ \frac{1}{2} - \frac{1}{2r}} -  x^{t+ \frac{1}{2} + \frac{1}{2r}} -  x^{t- \frac{1}{2} -\frac{1}{2r}} +  x^{t- \frac{1}{2} + \frac{1}{2r}}\bigr) ,$$
 by taking
 $$
  \eps = \sign(p)\cdot (-1)^{\tau + 1}, \ \ \ d = \xi + \alpha(p,r),$$
 we see that \eqref{eq:ZY} implies the desired formula in terms of the Laplace transform. Strictly speaking, we get it for the Laplace transform with $-a$ instead of $a$. However, given the conjugation symmetry of the $\Zhat_a$ invariants, using $a$ or $-a$ gives the same answer.
 \end{proof}

\subsection{Anti-symmetrization}
\label{sec:anti}
Recall from \eqref{eq:Fantisymm} that the series $F_K(x,q)$ is anti-symmetric with respect to the variable $x$. In many cases, it is helpful to write it as the anti-symmetrization of a series $\Fd_K(x,q)$ with only positive powers of $x$:
$$ F_K(x,q) = \frac{1}{2} \cdot(\Fd_K(x,q) - \Fd_K(x^{-1}, q)).$$
If we express $F_K(x,q)$ in terms of its coefficients $f_m(q)$ as in \eqref{eq:FKf}, we have
$$ \Fd_K(x,q) = \sum_{m \geq 1} f_m(q) \cdot x^{\frac{m}{2}}.$$

Observe that, for any Laplace transform $\CL_{p/r}^{(a)}$, we have
$$\CL_{p/r}^{(a)} \left[ (x^{\frac{1}{2r}} - x^{-\frac{1}{2r}}) F_K (x,q) \right]= \CL_{p/r}^{(a)} \left[ (x^{\frac{1}{2r}} - x^{-\frac{1}{2r}}) \Fd_K (x,q) \right].$$

Therefore, when applying the Dehn surgery formula (Theorem~\ref{thm:Dehn}), we could just as well use $\Fd_K(x,q)$ instead of $F_K(x,q)$.

Series of the form $\Fd_K(x,q)$ appear naturally in some circumstances. For example, we will encounter them as stability series for negative torus knots in Section~\ref{sec:stability}.

\newpage \section{Torus knots}
\label{sec:torus}
We now proceed to study the invariants $F_K(x, t)$ for the torus knots $K= T(s,t) \subset S^3$. We will assume that 
$$2 \leq s < t, \ \ \ \gcd(s,t)=1.$$

\subsection{Plumbing presentations}
\label{sec:PlumbedTorus}
For $s$ and $t$ as above, there are unique integers $t' \in (0, t), s' \in (0, s)$ such that $st' \equiv -1 \! \! \pmod{t}$ and $ts' \equiv -1 \!\!\pmod{s}$. These must satisfy the relation
$$ \frac{t'}{t} + \frac{s'}{s} =1 - \frac{1}{st}.$$
We construct a plumbing diagram for $T(s, t)$, with Seifert framing ($\lambda = \lambdasf$) as follows. The diagram consists of a tree with three legs: the central vertex is labeled $-1$, two legs have labels given by the continued fraction representations of $-t/t'$ and $-s/s'$, and the third has just the distinguished vertex, labelled $-st$. We choose the continued fraction representations so that all vertices have negative labels. This ensures that the plumbing satisfies the hypotheses of Lemma~\ref{lem:KnotCo}, and therefore the invariant $\Zhat_0$ is well-defined for the torus knot complement. The framing matrix $M$ for our graph will have one $0$ eigenvalue and the rest all negative eigenvalues.

\begin{example}
For the torus knots $T(2,2l+1)$ we have $t'=l, s'=1$, the continued fraction representations of $-(2l+1)/l$ consists of one $-3$ and $l-1$ copies of $-2$, and the continued fraction representation of $-2$ is just $-2$. For $T(3,4)$, we have $-t/t'= -4$ and $-s/s'=-3/2$, with the latter represented by two copies of $-2$. See Figure~\ref{fig:TorusKnots} for the resulting pictures.
\end{example}

\begin {figure}
\begin {center}
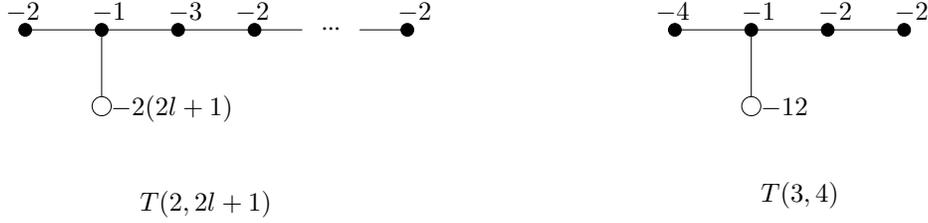
\caption {Plumbing diagrams for the complements of the torus knots $T(2,2l+1)$ and $T(3,4)$.}
\label{fig:TorusKnots}
\end {center}
\end {figure}

Let $p, r \in \Z$ with $p \neq 0$, $r>0$ and $\gcd(p,r)=1$. Moser \cite{Moser} showed that the result of $p/r$ surgery on $T(s,t)$ is a Seifert manifold, fibered over the $S^2$ orbifold with three singular points of orders $s$, $t$, and $|rst-p|$. Its orbifold Euler number is
$$ e = 1 -  \frac{t'}{t} - \frac{s'}{s} - \frac{r}{rst-p} = \frac{p}{st(rst-p)}.$$
The surgeries with values
$$p/r \in \{st-1, st,  st+1\}$$
produce special Seifert manifolds: lens spaces or connected sums of lens spaces. The other surgeries produce generic Seifert manifolds. In the generic case, by Theorem~\ref{thm:euler}, we have that $S^3_{p/r}(T(s,t))$ can be represented by a negative definite plumbing if and only if $e < 0$, that is,
\be
\label{eq:eneg}
\frac{p}{r} < 0 \ \ \text{or} \ \ \frac{p}{r} > st.
\ee

In particular, we will be interested in $-1/r$ surgeries on $T(s,t)$, for $r > 0$. These are the Brieskorn spheres
$$ S^3_{-1/r}(T(s,t)) = \Sigma(s, t, rst+1).$$
They are obtained from the standard gluing 
$$ \bigl(S^3 \setminus \nu T(s,t) \bigr) \ \cup_{T^2} \ \mathbb{S}_{-1/r},$$ 
as in Section~\ref{sec:DehnS}. We will represent the solid torus $\mathbb{S}_{-1/r}$ by the linear plumbing graph
$$ 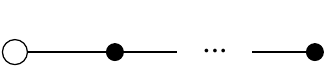$$
having $r-1$ copies of $-2$. We obtain a plumbing representation for $\Sigma(s, t, rst+1)$ from those for the torus knot complement and $\mathbb{S}_{-1/r}$, by adjoining them at their distinguished vertices. The vertex where we did the gluing is now labeled $-st-1$. One can check that the resulting plumbing graph for $\Sigma(s, t, rst+1)$ is negative definite.

\subsection{Calculation using reverse engineering}
Let $Y$ be the complement of the torus knot $T(s,t)$. To compute the invariant $F_K(x,q)= \Zhat_0(Y; x^{1/2}, n, q)$, we could use the plumbing representation above, and do a similar calculation to that for Brieskorn spheres in Proposition~\ref{prop:brieskorn}. However, since we already did the calculation for Brieskorn spheres, it is convenient (and more instructive) to deduce the answer for $T(s,t)$ from that, using the Dehn surgery formula from Theorem~\ref{thm:Dehn}, applied to $-1/r$ surgeries. We refer to the process of calculating a series from knowing its Laplace transforms as {\em reverse engineering}. The following lemma makes this process work.  

\begin{lemma}
\label{lem:uniqueLaplace}
If $F, G \in \Z[x^{1/2}, x^{-1/2}][q^{-1},q]]$ are series such that
$$ F(x, q) = -F(x^{-1}, q), \ \ \ G(x, q) = -G(x^{-1}, q)$$
and
$$\CL_{-1/r}[ (x^{\frac{1}{2r}} - x^{-\frac{1}{2r}})F(x,q)] =  \CL_{-1/r}[ (x^{\frac{1}{2r}} - x^{-\frac{1}{2r}})G(x,q)]$$
for all $r > 0$. Then $F=G$.
\end{lemma}

\begin{proof}
Since the Laplace transforms are linear, we might as well consider the difference $H=F-G$, written as
\be 
\label{eq:Hxq}
H(x, q) = \frac{1}{2} \sum_{m \geq 0} h_m(q) (x^\frac{m}{2} - x^{-\frac{m}{2}}),
\ee
with $h_m(q) \in \Z[q^{-1}, q]].$ We have
\begin{align}
0&= \CL_{-1/r}[ (x^{\frac{1}{2r}} - x^{-\frac{1}{2r}})H(x,q)] =  \sum_m h_m(q) \bigl( q^{r(\frac{m}{2} + \frac{1}{2r})^2} -  q^{r(\frac{m}{2} - \frac{1}{2r})^2} \bigr) \\
&= q^{\frac{1}{4r}} \sum_m h_m(q) \bigl( q^{\frac{rm^2}{4}+\frac{m}{2}} - q^{\frac{rm^2}{4}-\frac{m}{2}} \bigr). 
\label{eq:qmm}
\end{align}
If $H\neq 0$, let $m_0$ be the smallest $m$ such that $h_m(q) \neq 0$. Let $N$ be the smallest exponent of $q$ that appears in $h_{m_0}(q)$:
$$h_{m_0}(q) = c_N q^N + c_{N+1} q^{N+1} + \dots$$
Pick $r \gg 0$ so that
$$ \frac{r(m+1)^2}{4} - \frac{m+1}{2} > \frac{rm^2}{4}-\frac{m}{2} +N.$$
Then, there are no terms in the expansion of 
$$\sum_m h_m(q) \bigl( q^{\frac{rm^2}{4}+\frac{m}{2}} - q^{\frac{rm^2}{4}-\frac{m}{2}} \bigr)$$
which could cancel $c_N q^Nq^{r\frac{m^2}{4}-\frac{m}{2}}$. This gives a contradiction, and we conclude that $H=0$.
\end{proof}

From the proof of Lemma~\ref{lem:uniqueLaplace}, it is worth remembering the formula \eqref{eq:qmm} for the Laplace transform applied to $x^{\frac{1}{2r}} - x^{-\frac{1}{2r}}$ times functions of the form \eqref{eq:Hxq}.

As a simple example of reverse engineering, consider the $-1/r$ surgeries on the right-handed trefoil, which are the manifolds
$$ S^3_{-1/r} (\RHT) = \Sigma(2,3,6r+1) = M\left ( -1; \frac{1}{2}, \frac{1}{3}, \frac{r}{6r+1} \right).$$
The corresponding series $\Zhat_0(q)$, for small values of $r$, are shown in Table~\ref{tab:RHsmallsurgeries}. They can be computed using the formulas \eqref{eq:plumbing1} or \eqref{eq:Zbrieskorn}. 

\begin{table}
\centering
\begin{tabular}{ccc}
\hline\hline
\multicolumn{2}{c}{$Y = S^3_{-1/r} ({\RHT})$} & $\phantom{\oint_{\oint_{\oint}}^{\oint^{\oint}}} \Zhat_0(q)$
\\
\hline\hline
$r=1$ & $\Sigma (2,3,7)$ & $q^{1/2}(1-q-q^5+q^{10}-q^{11}+q^{18}+q^{30}-q^{41}+q^{43}-q^{56}-q^{76}+q^{93}-$ \\
 &  & $-q^{96}+q^{115}+q^{143}-q^{166}+q^{170}-q^{195}-q^{231}+q^{260}-q^{265}+q^{296}+\cdots)$ \\[2ex]
$r=2$ & $\Sigma(2,3,13)$ & $q^{1/2} (1-q-q^{11}+q^{16}-q^{23}+q^{30}+q^{60}-q^{71}+q^{85}-q^{98}-$ \\
 &  & $-q^{148}+q^{165}-q^{186}+q^{205}+q^{275}-q^{298}+ \cdots)$ \\[2ex]
$r=3$ & $\Sigma(2,3,19)$ & $q^{1/2} (1-q-q^{17}+q^{22}-q^{35}+q^{42}+q^{90}-q^{101}+q^{127}-q^{140}-$ \\
 &  & $-q^{220}+q^{237}-q^{276}+q^{295}+ \cdots)$ \\[2ex]
$r=4$ & $\Sigma(2,3,15)$ & $q^{1/2} (1-q-q^{23}+q^{28}-q^{47}+q^{54}+q^{120}-q^{131}+q^{169}-q^{182}-q^{292}+ \cdots)$ \\[2ex]
$r=5$ & $\Sigma(2,3,31)$ & $q^{1/2} (1-q-q^{29}+q^{34}-q^{59}+q^{66}+q^{150}-q^{161}+q^{211}-q^{224} + \cdots)$ \\[2ex]
\hline\hline
\end{tabular}
\caption{The invariants $\Zhat_0 (q)$ for $-1/r$ surgeries on the right-handed trefoil.}
\label{tab:RHsmallsurgeries}
\end{table}

From Table~\ref{tab:RHsmallsurgeries} let us read off the first few terms:
\begin{align*}
 \Zhat_0(S^3_{-1/r} ({\RHT}); q) &= q^{1/2}\bigl( (1-q) - q^{6r-1}(1-q^5) - q^{12r-1}(1-q^7) + \dots\bigr)\\
 &= - q^{1-\frac{r}{4}} \bigl( q^{\frac{r}{4}} (q^{\frac{1}{2}} - q^{-\frac{1}{2}}) - q\cdot q^{\frac{25r}{4}} (q^{\frac{5}{2}} - q^{-\frac{5}{2}}) - q^2 \cdot q^{\frac{49r}{4}} (q^{\frac{7}{2}} - q^{-\frac{7}{2}}) + \dots \bigr)
 \end{align*}
 
 Theorem~\ref{thm:Dehn}, applied to $-1/r$ surgeries on the trefoil, tells us that
 $$ \Zhat_0(S^3_{-1/r} ({\RHT}); q) = q^{-\frac{r}{4} - \frac{1}{4r}}  \CL_{-1/r} \left[ (x^{\frac{1}{2r}} - x^{-\frac{1}{2r}}) F_{\RHT} (x,q) \right].$$
 
Thus, from \eqref{eq:qmm} we obtain the first few terms of the series $F_{\RHT}(x, q)$:
\be
\label{eq:firsttrefoil}
 F_{\RHT}(x, q) = -\frac{1}{2} \Bigl(q (x^{\frac{1}{2}} - x^{-\frac{1}{2}}) - q^2(x^{\frac{5}{2}} - x^{-\frac{5}{2}}) - q^3(x^{\frac{7}{2}} - q^{-\frac{7}{2}})+\dots\Bigr)
 \ee

The same reverse engineering method can be applied to get an exact expression for $F_K(x, q)$ for all torus knots. Recall from Proposition~\ref{prop:brieskorn} that the $\Zhat_0(q)$ invariants for Brieskorn spheres are expressed in terms of the false theta functions
\be
\tPsi^{(a)}_p (q)  :=  \sum_{n=0}^\infty \psi^{(a)}_{2p}(n) q^{\frac{n^2}{4p}} 
\ee
where $\psi^{(a)}_{2p}(n)$ is $\pm 1$ if $ n\equiv \pm a \! \! \pmod{2p}$, and $0$ otherwise.

\begin{lemma}
\label{lem:abcde}
Suppose $a, b, c, d, e \in \Z_+$ with
$$ c > d > 0, \ \ d = \frac{a}{c} + \frac{1}{2}, \ \ e =\frac{bc^2}{2a}.$$
Then
\be
\label{eq:abcde}
\tilde \Psi^{(2ar + b+c)}_{c^2r+e} - \tilde \Psi^{(2ar + b-c)}_{c^2r+e} \; = \;
q^{v} \CL_{-1/r} \left[\,
\big( x^{\frac{1}{2r}}-x^{-\frac{1}{2r}} \big) G(x,q) \right]
\ee
where
$$G(x,q) = \frac{1}{2} \sum_{m \in \Z_+} \psi_{2c}^{(2d-1)}(m) \cdot (x^{\frac{m}{2}} - x^{-\frac{m}{2}}) q^{\frac{e}{4c^2}m^2} $$
and
$$ v = \frac{a}{2(b+2ar)}  - \frac{1}{4r}.$$
\end{lemma}

\begin{proof}
This is a  straightforward calculation using the formula \eqref{eq:qmm}. Compared with \eqref{eq:Hxq}, here we have doubled the summation index $m$, so that it now runs over the positive integers.
\end{proof}

At this point we are ready to prove Theorem~\ref{thm:TorusKnots}, which says that, for $K=T(s,t)$, 
\be
\label{eq:FKtorus}
F_K(x,q) = q^{\frac{(s-1)(t-1)}{2}} \cdot \frac{1}{2} \sum_{m > 0}   \eps_m \cdot (x^{\frac{m}{2}} - x^{-\frac{m}{2}})   q^{\frac{m^2 - (st-s-t)^2}{4st}}
\ee
where 
$$\eps_m = \begin{cases}
-1 & \operatorname{if } \ m \equiv st+s+t \ \operatorname{ or } \ st-s-t \!\! \pmod{2st}\\
+1 & \operatorname{if } \ m \equiv st+s-t \ \operatorname{ or }\  st-s+t \!\! \pmod{2st}\\
0 & \operatorname{otherwise.}
\end{cases}$$

\begin{proof}[Proof of Theorem~\ref{thm:TorusKnots}]
We apply Theorem~\ref{thm:Dehn} to $-1/r$ surgery on $K=T(s,t)$, which yields the Brieskorn sphere
$$S^3_{-1/r}(T(s, t)) = \Sigma(s, t, rst+1).$$
We get
\be
\label{eq:ZF}
 \Zhat_0(\Sigma(s, t, rst+1))= \eps  q^{d} \cdot \CL_{-1/r} \left[ (x^{\frac{1}{2r}} - x^{-\frac{1}{2r}}) F_K (x,q) \right],
 \ee
where in this case we can compute $\xi=0$, $\tau=0$ and hence
$$ \eps  = \sign(p)\cdot (-1)^{\tau + 1} = +1,$$
$$ d = \xi + \alpha(-1, r) = - \frac{r}{4} - \frac{1}{4r}.$$

In view of Lemma~\ref{lem:uniqueLaplace}, it suffices to show that when we plug in the expression \eqref{eq:FKtorus} for $F_K(x,q)$, the relation \eqref{eq:ZF} holds. 

We compute the left hand side of \eqref{eq:ZF} by applying Proposition~\ref{prop:brieskorn} to $\Sigma(s, t, rst+1).$ In the notation of Proposition~\ref{prop:brieskorn}, we have
$$ b_1 = s, \ \ b_2=r, \ \ b_3=str+1,$$
hence $b_1b_2b_3 = (st)^2 r+ st$ and
\begin{align*}
 \alpha_1 &=st(st-s-t)r - (s+t), \\
\alpha_2 &=st(st-s-t)r -(s+t) + 2st, \\
\alpha_3 &=st(st+s-t)r +s-t, \\
\alpha_4 &=st(st+s-t)r + s-t+2st.
\end{align*}
Proposition~\ref{prop:brieskorn} says that
 \be
 \label{eq:Zorst}
  \Zhat_0(\Sigma(s,t,rst+1)) = -q^{\Delta} \cdot \tPsi^{(\alpha_1)-(\alpha_2)-(\alpha_3)+(\alpha_4)}_{b_1b_2b_3}(q).
  \ee
 Here, $\Delta=\Delta(s, t, r)$ is given by the formula~\eqref{eq:Delta} applied to the plumbing graph $\tilde \Gamma$ representing $\Sigma(s, t, srt+1)$. This graph can be obtained from the plumbing tree $\Gamma$ for $T(s, t)$, by gluing it to the graph for the solid tours $\mathbb{S}_{-1/r}$, as described in Section~\ref{sec:PlumbedTorus}.
  
Let $\hat \Gamma$ be the linear plumbing graph obtained from $\Gamma$ by deleting the distinguished vertex. Note that $\Gamma'$ represents $S^3$. Let $k$ be the number of vertices in $\Gamma'$, and let $m_1, \dots, m_k$ be the labels of those vertices. Further, let $\eta_1$ and $\eta_2$ be the cardinalities of the first homology of the plumbed manifolds represented by the result of deleting from $\Gamma'$ one of its two terminal vertices. In other words, if we construct a continued fraction from $\Gamma'$ as in \eqref{contfract}, by going along the graph in either direction, then in both cases, the numerator of the fraction will be $\pm 1$; the denominator in one direction will be $\eta_1$, and in the other direction $\eta_2$.

Looking at the formula~\eqref{eq:Delta} in our case, we see that what we denoted $s$ in that formula corresponds to $k+r$ in our current notation, and what was $\sum m_v$ there is now $\sum_{i=1}^k m_i - st -2r+1$. Furthermore, the values $h_i$ from \eqref{eq:Delta} are
$$ h_1=rs^2 + \eta_1, \ \ h_2 = rt^2 +\eta_2, \ \ h_3=1.$$
One can check that
$$ \eta_1 + \eta_2 = 3k + \sum_{i=1}^k m_i  - st.$$
Therefore, we have
$$ \Delta(s,t,r) = \frac{(s-1)(t-1)}{2} - \frac{r}{4} - \frac{st}{4(1+str)} - \frac{(st-s-t)^2}{4st}$$

We now apply Lemma~\ref{lem:abcde} twice, first with the values
$$ a = \frac{st(st-s-t)}{2}, \ b=st-s-t, \ c=e=st, \ d=\frac{(s-1)(t-1)}{2}$$
and then with
$$ a= \frac{st(st+s-t)}{2}, \ b=st+s-t, \ c=e=st,\ d=\frac{(s-1)(t+1)}{2}+1.$$
In both cases we have $$v = \frac{st}{4str+1} - \frac{1}{4r} = -\frac{1}{4r(str+1)}.$$
By taking the difference of the two resulting relations~\eqref{eq:abcde}, and plugging in $F_K$ from \eqref{eq:FKtorus} and $\Zhat_0$ from \eqref{eq:Zorst}, we obtain that the relation~\eqref{eq:ZF} holds.
\end{proof}

If we prefer, for $K=T(s,t)$, we can write $F_K(x,q)$ as in Section~\ref{sec:anti}, as the anti-symmetrization of the series
\be
\Fd_{K}(x,q) =q^{\frac{(s-1)(t-1)}{2}} \sum_{m\geq 1} \eps_m \cdot x^{\frac{m}{2}}  q^{\frac{m^2 - (st-s-t)^2}{4st}}.
\ee
In expanded form, this is
\begin{align}
 \Fd_K(x,q) &=
 \sum_{n\geq 0} q^{stn^2 + (st-s-t)n} x^{stn+\frac{(s-1)(t-1)}{2}} + \sum_{n\geq 1}  q^{stn^2 - (st-s-t)n} x^{stn-\frac{(s-1)(t-1)}{2}+1}\nonumber \\
& -\sum_{n \geq 0} q^{(sn+s-1)(tn+1)} x^{stn+\frac{(s-1)(t+1)}{2}+1} - \sum_{n\geq 1} q^{(sn-s+1)(tn-1)} x^{stn-\frac{(s-1)(t+1)}{2}},
\end{align}
where we assumed that $2 \leq s < t$.

\subsection{More on the trefoil}
The first few terms of $F_K(x,q)$ for the right-handed trefoil $\RHT = T(2,3)$ were already written down in \eqref{eq:firsttrefoil}. Theorem~\ref{thm:TorusKnots} gives the exact formula
\be
F_{\RHT} (x,q) \; = \; \frac{q}{2} \sum_{m=1}^{\infty} \eps_m (x^{\frac{m}{2}} - x^{- \frac{m}{2}}) \, q^{\frac{m^2-1}{24}}
\label{Fxq31unnorm}
\ee
where
\be
\label{eq:epstrefoil}
\eps_m = \begin{cases}
-1 & \operatorname{if } \ m \equiv 1 \ \operatorname{ or } \ 11 \!\! \pmod{12}\\
+1 & \operatorname{if } \ m \equiv 5 \ \operatorname{ or }\  7 \!\! \pmod{12}\\
0 & \operatorname{otherwise.}
\end{cases}
\ee

The Dehn surgery formula (Theorem~\ref{thm:Dehn}) can be applied to this series for the values $p/r \not \in [0,6]$. For fractional $-1/r$ surgeries, we get back the answers from Table~\ref{tab:RHsmallsurgeries}. For some integer surgeries, the results are tabulated in Table~\ref{tab:RHTsurgeries}; they can also be obtained directly from the plumbing formula \eqref{eq:plumbing1}. 

Note that, by the discussion in Section~\ref{sec:PlumbedTorus}, the manifold $S^3_{p/r}(\RHT)$ bounds a negative definite plumbing iff 
$$ p/r \in (-\infty, 0) \cup \{5\} \cup [6, \infty).$$

\begin{table}
\centering
\begin{tabular}{ccc}
\hline\hline
\multicolumn{2}{c}{$Y = S^3_{p} (\RHT)$} & $\phantom{\oint_{\oint_{\oint}}^{\oint^{\oint}}}\Zhat_a(q)$
\\
\hline\hline
$p=10$ &  $M\! \left(-2;\frac{1}{2},\frac{1}{3},\frac{3}{4} \right)$ &
$\frac{q^{-3/20}}{2} (1 + q^5 + q^7 - q^{11} + q^{18} - q^{24} - q^{28} - q^{47} + q^{73} + \cdots)$ \\
 & & $\frac{q^{-13/20}}{2} (-1 + q + q^2 + q^9 - q^{22} - q^{39} - q^{44} + q^{53} - q^{67} + q^{78} + q^{85} + \cdots)$ \\
 & & $\frac{q^{-17/20}}{2} (1 - q^4 + q^7 - q^{10} + q^{21} - q^{26} + q^{33} + q^{59} - q^{61} - q^{95} + \cdots)$ \\
 & & $\frac{q^{-7/20}}{2} (-1 + q - q^3 - q^{14} + q^{15} + q^{34} - q^{42} + q^{49} - q^{71} + q^{80} - q^{92} + \cdots)$ \\
 & & $q^{7/4} (-1 - q^{10} + q^{15} + q^{35} - q^{85} + \cdots)$ \\
 & & $q^{-3/4} (-1 - q^5 + q^{30} + q^{55} - q^{65} - q^{100} + \cdots)$ \\[2ex]
$p=9$ & $M\! \left(-2;\frac{1}{2},\frac{1}{3},\frac{2}{3} \right)$ & $\frac{q^{1/18}}{2} (1 + q - q^8 - q^{11} + q^{25} + q^{30} - q^{51} - q^{58} + q^{86} + q^{95} + \cdots)$ \\
 & & $\frac{q^{-1/2}}{2} (-1 + 2 q - q^3 + q^6 - 2 q^{10} + q^{15} - q^{21} + 2 q^{28} - q^{36} + \cdots)$ \\
 & & $\frac{q^{-11/18}}{2} (1 - q^4 + q^5 + q^{17} - q^{19} - q^{39} + q^{42} + q^{70} - q^{74} + \cdots)$ \\
 & & $q^{-1/2} (-1 - q^3 + q^6 + q^{15} - q^{21} - q^{36} + q^{45} + q^{66} - q^{78} + \cdots)$ \\
 & & $\frac{q^{-5/18}}{2} (-1 - q^2 + q^7 + q^{13} - q^{23} - q^{33} + q^{48} + q^{62} - q^{82} - q^{100} + \cdots)$ \\[2ex]
$p=8$ & $M\! \left(-2;\frac{1}{2},\frac{1}{3},\frac{1}{2} \right)$  &
$\frac{q^{-3/8}}{2} (1 + q - q^2 + q^5 - q^7 + q^{12} - q^{15} + q^{22} - q^{26} + q^{35} + \cdots)$ \\
 & & $\frac{q^{-3/8}}{2} (-1 + q - q^2 + q^5 - q^7 + q^{12} - q^{15} + q^{22} - q^{26} + q^{35} + \cdots)$ \\
 & & $- q^{1/4} (-1 + q^2 - q^{10} + q^{16} - q^{32} + q^{42} - q^{66} + q^{80} + \cdots)$ \\
 & & $q^{-1/4} (-1 + q^4 - q^8 + q^{20} - q^{28} + q^{48} - q^{60} + q^{88} + \cdots)$ \\[2ex]
$p=7$ & $L(7,1)$   & $-2q$, \ \ $q^{8/7}$ \\[2ex]
$\dots$ & $\dots$ & $\dots$ \\[2ex]
$p=-1$ & $\Sigma (2,3,7)$ & $q^{1/2}( 1-q-q^5+q^{10}-q^{11}+q^{18}+q^{30}-q^{41}+q^{43}-q^{56}-q^{76}+q^{93}+\cdots)$ \\[2ex]
$p=-2$ & $M\! \left(-1;\frac{1}{2},\frac{1}{3},\frac{1}{8} \right)$ & $q^{3/4} (1-q^3+q^{10}-q^{23}+q^{25}-q^{44}+q^{65}-q^{94}+q^{98} + \cdots)$ \\
 &  & $-q^{5/4} (1-q^5+q^6-q^{17}+q^{31}-q^{52}+q^{55}-q^{82} + \cdots)$ \\[2ex]
$p=-3$ & $M\! \left(-1;\frac{1}{2},\frac{1}{3},\frac{1}{9} \right)$ & $q+q^5-q^6-q^{18}+q^{20}+q^{40}-q^{43}-q^{71}+q^{75} +\cdots$ \\
 &  & $-q^{4/3} (1+q^2-q^7-q^{13}+q^{23}+q^{33}-q^{48}-q^{62}+q^{82} + \cdots)$ \\[2ex]
$p=-4$ & $M\! \left(-1;\frac{1}{2},\frac{1}{3},\frac{1}{10} \right)$ & $q^{5/4} (1+q^6-q^{28}-q^{58}+q^{62} + \cdots)$ \\
 &  & $-q^{3/2} (1-q^3+q^4-q^{11}+q^{19}-q^{32}+q^{35}-q^{52}+q^{68}-q^{91}+q^{96} + \cdots)$ \\
 &  & $-q^{13/4} (1+q^{12}-q^{14}-q^{38}+ q^{82} + \cdots)$ \\[2ex]
\hline\hline
\end{tabular}
\caption{The invariants $\Zhat_a (q)$ for some integer surgeries on the right-handed trefoil.}
\label{tab:RHTsurgeries}
\end{table}

\subsection{Negative torus knots}
\label{sec:negative}
Recall from \eqref{eq:reversal} that the WRT invariants of a three-manifold $Y$ are related to those of $-Y$ by the change of variables $q \mapsto q^{-1}$. On the other hand, in general, the $\Zhat_a(q)$ invariants of $Y$ and $-Y$ are related in a more complicated fashion: the change $q \mapsto q^{-1}$ involves going from a series converging in the unit disk $|q| < 1$ to one converging for $|q| > 1$; see \cite{CCFGH}. 

Similarly, for a knot $K$ in a homology sphere $\hat Y$, we expect the series $F_K(x,q)$ to be related to that of the same knot in $-\hat Y$ in a complicated way. On the other hand, when $\hat Y = S^3$ or a lens space, the underlying $\Zhat_a(q)$ invariants of the closed manifold are Laurent polynomials (rather than Laurent power series), and we expect things to simplify. Indeed, for plumbed knots $K$ in such manifolds, one can check that, in the series $F_K(x,q)$ from \eqref{eq:FK} and \eqref{eq:FKf}, each coefficient $f_m(q)$ of a power of $x$ is a Laurent polynomial in $q$. 

In particular, this is true for torus knots in $S^3$ (or, more generally, for algebraic knots). For such knots, it makes sense to define the series for the mirror knot by
$$ F_{m(K)}(x,q) : = -F_{K}(x^{-1}, q^{-1}) = F_K(x, q^{-1}),$$
so that
$$ f_{m(K)}(x, q) = f_K(x^{-1},q^{-1}) = f_K(x, q^{-1}).$$
Observe that the series $F_{m(K)}(x,q)$ is no longer an element of some
$$2^{-c} q^{\Delta} \Z[x^{1/2}, x^{-1/2}][q^{-1},q]],$$
because there is no lower bound for the exponents of $q$ over all $x$. Rather, it is a formal power series in both $x$ and $x^{-1}$, i.e. 
$$F_{m(K)}(x,q) \in 2^{-c} q^{\Delta} \Z[q^{-1}, q][[x^{1/2}, x^{-1/2}]].$$
It is worth noting that formal power series in $x$ and $x^{-1}$ form only a vector space, not a ring.

Going back to torus knots, as we noted in the Introduction, the series for $m(T(s, t)) = T(s, -t)$ is the anti-symmetrization of
$$ \Psi(x, q) := \Fd_{T(s, -t)} (x, q) = q^{-\frac{(s-1)(t-1)}{2}} \sum_{m\geq 1} \eps_m \cdot x^{\frac{m}{2}}  q^{-\frac{m^2 - (st-s-t)^2}{4st}}.$$
 
Let us focus on the negative trefoil $\LHT = T(2, -3)$, for which we have
\begin{align*}
 \Psi(x, q) &= q^{-1} \sum_{m\geq 1} \eps_m \cdot x^{\frac{m}{2}}  q^{-\frac{m^2 - 1}{24}} \\
&= q^{-1}x^{1/2}(-1 + q^{-1}x^2 + q^{-2}x^3 - q^{-5}x^5 - \dots)
 \end{align*}
 
We can apply Laplace transforms to this series, multiplied by a suitable factor, as in Theorem~\ref{thm:Dehn}. This makes sense for values $p/r \in (-6, 0)$. The results for $\Zhat_a (q)$ of $S^3_{p/r}(\LHT)$ are summarized in Tables~\ref{tab:surgeriesLHT} and ~\ref{tab:LHsmallsurgeries}. An indication that we have the right definition of $F_{\LHT}(x, q)$ is that the same answers can be obtained from the plumbing formula \eqref{eq:plumbing1}. Note that, by the discussion in Section~\ref{sec:PlumbedTorus}, the manifold $$S^3_{p/r}(\LHT) = - S^3_{-p/r}(\RHT)$$ bounds a negative definite plumbing iff 
$$ p/r \in \{-7\} \cup [-6, 0).$$

\begin{table}
\centering
\begin{tabular}{ccc}
\hline\hline
\multicolumn{2}{c}{$Y = S^3_{p} ({\bf 3_1^{\ell}})$} & $\phantom{\oint_{\oint_{\oint}}^{\oint^{\oint}}} \Zhat_a(q)$
\\
\hline\hline
\rule{0pt}{3ex}    $p=2$ & $M\! \left(-2;\frac{1}{2},\frac{2}{3},\frac{7}{8} \right)$ & $\ldots$ \\[2ex]
$p=1$ & $-\Sigma (2,3,7)$ & $\ldots$ \\[2ex]
$p=-1$ & $\Sigma (2,3,5)$ & $q^{-3/2}(1-q-q^3-q^7+q^8+q^{14}+q^{20}+q^{29}-q^{31}-q^{42}+\cdots)$ \\[2ex]
$p=-2$ & $M\! \left(-2;\frac{1}{2},\frac{2}{3},\frac{3}{4} \right)$ & $-q^{-3/4} (1+q^2-q^3-q^7+q^{17}+q^{25}-q^{28}-q^{38}+q^{58}+q^{72}+\cdots)$ \\
 &  & $q^{-5/4} (1-q+q^6+q^{11}-q^{13}-q^{20}+q^{35}+q^{46}-q^{50}-q^{63}+\cdots)$ \\[2ex]
$p=-3$ & $M\! \left(-2;\frac{1}{2},\frac{2}{3},\frac{2}{3} \right)$ & $-q^{-2/3} (1-q^3+q^9-q^{18}+q^{30}-q^{45}+q^{63}-q^{84}+\cdots)$ \\
 &  & $q^{-1} (1-q+q^2+q^5-q^7-q^{12}+q^{15}+q^{22}-q^{26}-q^{35}+q^{40}+\cdots)$ \\[2ex]
$p=-4$ & $M\! \left(-2;\frac{1}{2},\frac{1}{2},\frac{2}{3} \right)$ & $-q^{-1/2} (1-q+q^5-q^8+q^{16}-q^{21}+q^{33}-q^{40}+q^{56}-q^{65}+\cdots)$ \\
 &  & $q^{-3/4} (1+q^2-q^4+q^{10}-q^{14}+q^{24}-q^{30}+q^{44}-q^{52}+q^{70}+\cdots)$ \\
 &  & $- q^{-3/4} (1-q^2+q^4-q^{10}+q^{14}-q^{24}+q^{30}-q^{44}+q^{52}-q^{70}+\cdots)$ \\[2ex]
$p=-5$ & $-L(5,1)$ & $-2q^{-1/2}, \ \ q^{-7/10}$ \\[2ex]
$p=-6$ & $L(3,1) \# L(2,1)$ & $\ldots$ \\[2ex]
$p=-7$ & $-L(7,1)$ & $\ldots$ \\[2ex]
$p=-8$ &$M\! \left(-1;\frac{1}{2},\frac{1}{2},\frac{2}{3} \right)$  & $\ldots$ \\[2ex]
\hline\hline
\end{tabular}
\caption{The invariants $\Zhat_a (q)$ for some integer surgeries on the left-handed trefoil. }
\label{tab:surgeriesLHT}
\end{table}

\begin{table}
\centering
\begin{tabular}{ccc}
\hline\hline
\multicolumn{2}{c}{$Y = S^3_{-1/r} ({\bf 3_1^{\ell}})$} & $\phantom{\oint_{\oint_{\oint}}^{\oint^{\oint}}} \Zhat_a(q)$
\\
\hline\hline
\rule{0pt}{3ex} $r=1$ & $\Sigma (2,3,5)$ & $q^{-3/2}(1-q-q^3-q^7+q^8+q^{14}+q^{20}+q^{29}-q^{31}-q^{42}+\cdots)$ \\[2ex]
$r=2$ & $\Sigma(2,3,11)$ &
$q^{-3/2} (1 - q - q^9 + q^{14} - q^{19} + q^{26} + q^{50} - q^{61} + q^{71} - q^{84}$ \\
 &  & $- q^{124} + q^{141} - q^{156} + q^{175} + q^{231} - q^{254} + q^{274} - q^{299}+ \cdots)$ \\[2ex]
$r=3$ & $\Sigma(2,3,17)$ &
$q^{-3/2} (1 - q - q^{15} + q^{20} - q^{31} + q^{38} + q^{80} - q^{91} + q^{113}$ \\
 &  & $- q^{126} - q^{196} + q^{213} - q^{246} + q^{265} + \cdots)$ \\[2ex]
$r=4$ & $\Sigma(2,3,23)$ &
$q^{-3/2} (1 - q - q^{21} + q^{26} - q^{43} + q^{50} + q^{110} - q^{121}$ \\
 &  & $+ q^{155} - q^{168} - q^{268} + q^{285} + \cdots)$ \\[2ex]
$r=5$ & $\Sigma(2,3,29)$ &
$q^{-3/2} (1 - q - q^{27} + q^{32} - q^{55} + q^{62} + q^{140} - q^{151} + q^{197} - q^{210} + \cdots)$ \\[2ex]
\hline\hline
\end{tabular}
\caption{The invariants $\Zhat_0 (q)$ for $-1/r$ surgeries on the left-handed trefoil.}
\label{tab:LHsmallsurgeries}
\end{table}

It is helpful to study the normalized version of $\Psi(x, q)$:
\begin{align*}
 \psi(x, q) &= \frac{\Psi(x,q)}{x^{1/2} - x^{-1/2}} \\
&= -x^{-1/2} \Psi(x,q) \cdot (1+ x + x^2 + \dots)\\
&= q^{-1}x(1 - q^{-1}x^2 - q^{-2}x^3 + q^{-5}x^5 + \dots)(1+ x + x^2 + \dots)\\
&= \frac{x}{q} \left( 1 + x + x^2(1-q^{-1}) + x^3 ( 1-q^{-1}) + \dots \right)\\
&=\frac{x}{q} \sum_{m=0}^{\infty} x^m \Bigl( 1- \frac{x}{q} \Bigr) \dots \Bigl( 1- \frac{x}{q^m} \Bigr).
\end{align*}
The last expression is (up to a normalization factor) the Garoufalidis-Le stability series for the trefoil; cf. \cite[p.11]{GaroufalidisLe}. It is obtained by setting $x=q^n$ in the following formula for the colored Jones polynomial of the trefoil, from \cite[Section 1.1.4]{HuynhLe}:
\begin{align}
 \label{JnLHT}
 J_{\LHT, n}(q) &= q^{n-1} \sum_{m=0}^{\infty} q^{mn} (1-q^{n-1})(1-q^{n-2}) \dots (1-q^{n-m})\\
 &=q^{n-1} \sum_{m=0}^{\infty} q^{mn} (q^{n-m})_m.
\notag
 \end{align}
Observe that the sum in \eqref{JnLHT} is finite, because the terms are $0$ for $m \geq n$.

\begin{remark}
The papers \cite{GaroufalidisLe} and \cite{HuynhLe} give these formulas for the right-handed, rather than the left-handed, trefoil. This is because their conventions for the colored Jones polynomial differ from ours by $q \leftrightarrow q^{-1}$. See Section~\ref{sec:conventions}. 
\end{remark}

\begin{remark}
In the literature there is another well-known formula for the colored Jones polynomial of the trefoil---its cyclotomic expansion \eqref{eq:LHTcyclo}, whose terms involve two Pochhammer symbols instead of one:
$$J_{\LHT,n}(q)= \sum_{m =0}^{\infty} q^m (q^{n+1})_m (q^{1-n})_m.$$
We could set $q^n = x$ and get a series as in \eqref{eq:Fhabiro}:
$$ C_K(x,q)=\sum_{m= 0}^{\infty} q^m(qx)_m (qx^{-1})_m.$$ 
We could try to apply the Dehn surgery formula~\eqref{eq:Dehnf} to $C_K(x,q)$ instead of $f_K(x,q)$, and see if we recover the invariants $\Zhat_a(q)$ of the surgeries. Interestingly, we get the right answer for the $-1$ surgery (the case of the Poincar\'e sphere), but not for other surgeries.
\end{remark}

\subsection{Stability series}
\label{sec:stability}
In this Section we prove Theorem~\ref{thm:EvenStability}, relating $\Psi(x,q)$ to the stability series for the negative torus knot. Recall the definition of stability series from \eqref{eq:stability1},  \eqref{eq:stability2}. 
In \cite{GaroufalidisLe}, Garoufalidis and Le studied the stability series 
$$\Phi(x,q)=\sum_j \Phi_j(q)x^j$$
for a version $\Jhat_{K, n}$ of the colored Jones polynomial. Specifically, $\Jhat_{K, n}$ is obtained from the unnormalized version $\tJ_{K, n}$ by dividing by its lowest monomial, so that $\Jhat_{K, n}$ starts in degree $0$. Of particular interest in the literature has been $\Phi_0(q)$, which consists of the lowest degree terms of $\Jhat_{K, n}$, and is called the {\em tail} of the colored Jones polynomials. The highest degree terms give the {\em head}, which can be obtained from the tail of the mirror knot by taking $q \mapsto q^{-1}$. See \cite{DasbachLin}, \cite{ArmondDasbach}.

Garoufalidis and Le showed that stability series exist for alternating knots; cf. \cite[Theorem 1.4]{GaroufalidisLe}. However, stability series do not exist for all knots. For example, it was observed by Armond and Dasbach  in \cite[Proposition 6.1]{ArmondDasbach} that the positive torus knots $T(s, t)$ with $s, t > 2$ do not have a tail; rather, in that case, the even colored Jones polynomials $\Jhat_{2n}$ have one tail, and the odd ones $\Jhat_{2n-1}$ have a different tail. On the other hand, by \cite[Theorem 1.17]{GaroufalidisLe}, the colored Jones polynomials of negative knots admits stability series.\footnote{Because of the difference in conventions for the colored Jones polynomial, the results stated in \cite{ArmondDasbach} and \cite{GaroufalidisLe} for positive knots apply to negative knots in our setting, and vice versa.}

\begin{proof}[Proof of Theorem~\ref{thm:EvenStability}]
Let $K = T(s, t)$. We use Morton's formula \cite{Morton} for $J_{K, n}(q)$, as rehashed by Hikami in \cite[Theorem 1]{HikamiDifference}:
\be
\label{eq:hikami}
J_{K, n}(q)= -\frac{q^{-\frac{stn}{4}} q^{\frac{(s-1)(t-1)}{2}}}{q^{\frac{n}{2}} - q^{-\frac{n}{2}}} \sum_{k=0}^{stn} \eps_{stn-k} q^{\frac{k^2 - (st-s-t)^2}{4st}},
\ee
where the values $\eps_m$ are as in \eqref{eq:epsem}. By replacing $q$ with $q^{-1}$ and multiplying by the quantum integer $[n]$ we obtain the unnormalized Jones polynomial for the mirror $m(K)=T(s, -t)$:
\be
\label{eq:tJm}
 \tJ_{m(K), n} = \frac{q^{\frac{stn}{4}} q^{-\frac{(s-1)(t-1)}{2}}}{q^{\frac{1}{2}} - q^{-\frac{1}{2}}} \sum_{k=0}^{stn} \eps_{stn-k} q^{-\frac{k^2 - (st-s-t)^2}{4st}}.
 \ee
 After changing variables to $m = stn-k$, we get
 \be
\label{eq:tJm2}
 (q^{\frac{1}{2}} - q^{-\frac{1}{2}}) \cdot \tJ_{m(K), n} =  q^{-\frac{(s-1)(t-1)}{2}} \sum_{m=0}^{stn} \eps_{m} q^{\frac{mn}{2}} q^{-\frac{m^2 - (st-s-t)^2}{4st}}.
 \ee
Clearly, the stability series of this is given by $\Psi(x,q)$ from \eqref{eq:Psixq}.
\end{proof}

\begin{remark}
While for negative torus knots we have a direct relation between $\Fd_K(x,q)$ and the stability series, we cannot expect this to hold for arbitrary knots. Indeed, as noted above, stability series do not even exist for all knots. Even when they do, e.g. for the positive torus knots $T(2,t)$, the same relation does not hold. 
\end{remark}

\begin{remark}
From the formula \eqref{eq:hikami} we can extract the stability series for the even and odd colored Jones polynomials of the positive torus knot $K=T(s,t)$. For simplicity, we will work with the normalized colored Jones polynomials $J_{K,n}(q)$, shifted to start in degree $0$. When $n$ is even, these admit a stability series $\Phi(x,q)=\sum_j \Phi_j(q)x^j$ with
$$ \Phi_0(q) = \sum_{m=0}^{\infty} \eps_{m} q^{\frac{m^2 - (st-s-t)^2}{4st}} = q^{-\frac{(s-1)(t-1)}{2}} \cdot \Fd_K(1, q)$$
and
$$\Phi_j(q) = 0 \ \text{for } j > 0.$$
If we want the series for the unnormalized versions of colored Jones, we divide $\Phi(x,q)$ by $1-x$, that is, we multiply it by $1+x +x^2 + \dots$ We get the same tail $\Phi_0(q)$ as before, but this time it is replicated in all degrees $j \geq 0$.

When $n$ is odd, for the normalized versions $J_{K,n}(q)$ we get the stability series $\Upsilon(x,q)=\sum_j \Upsilon_j(q)x^j$ with
$$ \Upsilon_0(q) = \sum_{m=0}^{\infty} \eps_{st-m} q^{\frac{m^2 - (s-t)^2}{4st}}, \ \ \Upsilon_j(q) = 0 \ \text{for } j > 0.$$
\end{remark}

\begin{remark}
When $s=2$, the torus knot $T(2,t)$ is alternating. In that case it is easy to see that $\Phi_0(q) = \Upsilon_0(q)$, so the odd and even stability series coincide. 

In particular, for the right-handed trefoil $\RHT=T(2,3)$, the tail of $J_{\RHT, n}(q)$ is
$$\Phi_0(q) = \sum_{m=0}^{\infty} \eps_{m} q^{\frac{m^2 - 1}{24}}=\sum_{k \in \Z} (-1)^k q^{k(3k-1)/2}=(q)_{\infty},$$
which is Euler's pentagonal series from \eqref{eq:Euler}.
\end{remark}

\newpage \section{Resurgence}
\label{sec:resurgence}
Resurgence is the process of recovering non-perturbative features of a function from its asymptotic (perturbative) expansion. This is very useful in quantum mechanics and quantum field theory. For introductions to resurgence, see  \cite{Costin}, \cite{Dorigoni}, or \cite{Marino}. 

We are interested in applying resurgence analysis to the Chern-Simons functional. This was done for closed $3$-manifolds in \cite{GMP}, and we will show how the same techniques can be used for knot complements. 

\subsection{Closed three-manifolds}
Let us briefly review how resurgence was applied in \cite{GMP} to the Chern-Simons functional on some closed three-manifolds $Y$. This gave a construction of the invariants $\Zhat_a(Y; q)$ for those manifolds. The examples in \cite{GMP} were the Brieskorn spheres $\Sigma(2,3,5)$ and $\Sigma(2,3,7)$. More Seifert fibered examples were analyzed in  \cite{CCFGH} and \cite{Chung}. In principle, resurgence can be done for any $Y$, but it is not a completely algorithmic procedure, and it is difficult to carry out in practice. 

As usual, we will assume that $Y$ is a rational homology sphere. Recall that the set of labels $a$ for the invariants $\Zhat_a(Y; q)$ is $\Spinc(Y)$; if we take into account the conjugation symmetry, we could say it is $\Spinc(Y)/\Z_2$. Noncanonically, this can be identified with 
$$T:=H_1(Y; \Z) /\Z_2=\cM_{\operatorname{flat}}^{\operatorname{ab}}(Y; \su),$$
the moduli space of Abelian flat $\su$ connections on $Y$. (Equivalently, we could consider Abelian flat $\sl$ connections.)

Let us also introduce
$$\cM_{\operatorname{flat}}(Y; \su),$$
the moduli space of all flat $\su$ connections on $Y$.

The analysis in \cite{GMP} starts by considering the asymptotic expansion of the Chern-Simons partition function as $k \to +\infty$:
$$ Z_{\CS}(Y; k) \approx \sum_{\alpha \in \pi_0(\cM_{\operatorname{flat}}(Y; \su))} e^{2\pi i k \CS(\alpha)} \Zpert_{\alpha}(k).$$
It is convenient to change variables to 
$$\hbar=\frac{2\pi i}{k} \to 0$$
and write $\Zpert_{\alpha}$ in terms of $\hbar$. We obtain a trans-series of the form
$$ Z_{\CS}(Y; k) \approx \sum_{\alpha \in \pi_0(\cM_{\operatorname{flat}}(Y; \su))} e^{-4\pi^2\CS(\alpha)/\hbar } \Zpert_{\alpha}(\hbar),$$
where
\be
\label{eq:Zh}
 \Zpert_{\alpha}(\hbar)= \hbar^{\delta_{\alpha}}  \, \sum_{n=0}^{\infty} c_n^{(\alpha)} \hbar^n
 \ee
for some $\delta_{\alpha} \in \Q$. (For example, when $\theta$ is the trivial flat connection, then $\Zpert_{\theta}(\hbar)$ is the Ohtsuki series from \cite{OhtsukiIntegral} and \cite{OhtsukiRational}.)

We will use the notation $a$ for Abelian flat connections (elements of $T$). For Abelian flat connections on rational homology spheres, the value $\delta_{a}$ is a half-integer, so we could take it to be $1/2$ for simplicity. From $ \Zpert_{a}(\hbar)$ we construct its Borel transform
$$\BZ_a(\xi) = \sum_{n=0}^{\infty}  \frac{c_n^{(a)}}{\Gamma (n + \frac{1}{2})}\xi^{n-\frac{1}{2}},$$
where $\Gamma$ is the gamma function. We analytically continue $\BZ_a(\xi)$ for $\xi$ in some open subset of $\C$, and denote the result by $\tBZ_a$. The poles $\xi_a$ of the Borel transform in the complex plane will be exactly $2\pi$ times the values $\xi_a$ of the $\CS$ functional at all (Abelian and non-Abelian) flat $\sl$ connections on $Y$.

Then, we do an inverse Borel transform 
$$Z_a(\hbar) = \int_{\R_+}  e^{-(\xi-\xi_a)/\hbar} \tBZ_a(\xi) \frac{d\xi}{\sqrt \pi \xi }.$$

We could also replace the integration contour $\R_+$ by another contour. This choice of contour (which can be a linear combination of curves) corresponds to a prescription for Borel resummation. Going between contours  is related to the residues $n_{ab}$ around each pole $b$ of $\tBZ_a$. In the examples studies in \cite{GMP}, the new contour was taken to be $1/2$ the union of two half-lines $ie^{i\epsilon}\R_+ \cup ie^{-i\epsilon}\R_+$, for $\epsilon > 0$ small.

By analyzing the poles and residues of $\tBZ_a$, and regularizing the infinite sums that come out of this process, in good cases one obtains a closed form expression for $Z_a(\hbar)$. We then set $q = e^{ \hbar}$ in $Z_a$, expand in $q$ and get a $q$-series $Z_a(q)$. Finally, from this, we use an S-transform 
\be
\label{eq:stransf}
Z_a(q) = \sum_b S_{ab} \Zhat_b(q).
\ee
to turn $Z_a(q)$ into the desired invariants $\Zhat_a(q)$. The values $S_{ab}$ in \eqref{eq:stransf} are as in Equation~\eqref{eq:Sab}.

\subsection{Resurgence for knot complements}
Let $Y = S^3 \setminus \nu K$ be a knot complement. We denote by $\Char(Y)$ the $\sl$ character variety of $Y$, and by $\Char(\del Y)$ that of the boundary $\del Y \cong T^2$. The image of the restriction map
$$ r: \Char(Y) \; \to \; \Char(\del Y) \cong (\C^* \times \C^*)/\Z_2$$
has some one-dimensional and (possibly) some zero-dimensional components. The closure of the one-dimensional components, lifted to the double cover $\C^* \times \C^*$, is the zero locus of a polynomial $A(x, y)$, called the {\em A-polynomial} of the knot \cite{Apolynomial}. Here, $x, y \in \C^*$ are the eigenvalues of the holonomies around the meridian and longitude, respectively. 

For each fixed value of $x$, let $\alpha$ label the different $\sl$ flat connections with meridian holonomy $x$.  
(In fact, for many simple knots, such connections are uniquely determined by the value of $x$ and the choice of a branch of the A-polynomial curve $A(x,y) = 0$.) Around each connection $\alpha$, the trans-series expansion of the Chern-Simons functional on the knot complement produces a perturbative series just as in \eqref{eq:Zh}:
\be
\Zpert_{\alpha}(\hbar)= \hbar^{\delta_{\alpha}}  \, \sum_{n=0}^{\infty} c_n^{(\alpha)} \hbar^n
\label{Zgenasympt}
\ee
For example, for the figure-eight, the A-polynomial curve has three branches (illustrated in Figure~\ref{fig:Acurve}), corresponding to the hyperbolic connection, its conjugate, and the Abelian connections. At $x=1$ we have
\be
\delta_{\alpha} \; = \;
\begin{cases}
0, & \alpha = \text{geom.} \\
\frac{3}{2}, & \alpha = \text{abel.} \\
0, & \alpha = \text{conj.}
\end{cases}
\ee
For $\alpha = \text{geom}$ or $\alpha = \text{conj}$, the perturbative coefficients $c_n^{(\alpha)}$
were computed in \cite{Dimofte:2009yn} and extensively used in subsequent developments ({\it e.g.} in 3d-3d correspondence). However, the case $\alpha = \text{abel}$ has received less attention.

\begin{figure}[ht]
\centering
\includegraphics[trim={0 0in 0 0in},clip,width=4.5in]{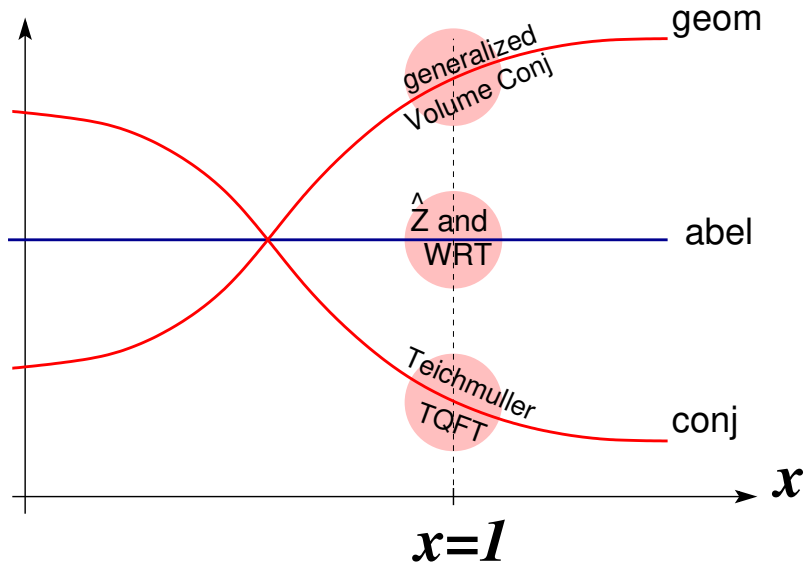}
\caption{An illustration of the different branches of the A-polynomial curve.
For the figure-8 knot there are three branches and, correspondingly, three asymptotic expansions \eqref{Zgenasympt} for every fixed value of $x$.
In particular, near $x=1$ these three asymptotic expansions are related to the generalized volume conjecture,
the asymptotics of $\Zhat$ and WRT invariants, and the Teichm\"uller TQFT, respectively.}
\label{fig:Acurve}
\end{figure}

For general hyperbolic knots, the series \eqref{Zgenasympt} for the hyperbolic connection has been studied in relation to the  generalized volume conjecture \cite{Gukov:2003na}, and that for its conjugate in relation to the Teichm\"uller TQFT \cite{AndersenKashaev}.

Our interest here is to do resurgence as for closed $3$-manifolds, but with an additional parameter $x$. Thus, we will start with the asymptotics \eqref{Zgenasympt} around the Abelian flat connection, and the outcome will be the series $F_K(x,q)$, which plays the role of $\Zhat_a(q)$. Since there is only one Abelian branch (i.e., for each fixed $x$, we have only one Abelian flat connection), there will be a unique series $F_K(x,q)$.

The asymptotics around the Abelian flat connection are the basis of the Melvin-Morton conjecture \cite{MelvinMorton}, proved in \cite{BNG} and \cite{RozanskyContribution}, and extended by Rozansky in \cite{Rozansky96}, \cite{RozanskyMMR}.

Specifically, the perturbative $\hbar$-expansion of the Chern-Simons functional on the knot complement $Y = S^3 \setminus \nu K$ is proportional to the colored Jones polynomial of $K$:
\be
Z_{\CS}(Y; k) \; \sim \; J_n (q = e^{\hbar}) \; = \; \sum_{m=0}^{\infty} R_m (x) \hbar^m
\label{pertknotcompl}
\ee
where, as usual, $x = q^n = e^{n \hbar}$.
Rozansky proved \cite{RozanskyMMR} that $R_m (x)$ are rational functions, such that
\be
R_m (x) \; = \; \frac{P_m (x)}{\Delta_K (x)^{2m+1}}
\label{Rmrational}
\ee
where $P_m (x) \in \mathbb{Q} [x^{\pm1}]$ are Laurent polynomials, $P_0 (x)=1$,
and $\Delta_K (x)$ is the Alexander polynomial of $K$.

The Alexander polynomial and the first few polynomials $P_m (x)$ for some simple knots are listed in the following table:

\begin{table}[htb]
\centering
\renewcommand{\arraystretch}{1.2}
\begin{tabular}{|@{\quad}l@{\quad}|@{\quad}l@{\quad}|@{\quad}l@{\quad}|@{\quad}p{5.7cm}@{\quad}|}
\hline knot $K$ & $\Delta_K (x)$ & $P_1 (x)$ & $P_2 (x)$
\\
\hline
\hline unknot & $1$ & $0$ & $0$ \\
\hline $3_1$ & $- 1 + x^{-1} + x$ & $2 - 2x^{-1} -2x + x^{-2} + x^2$ &\begin{footnotesize}$9 - 6x^{-1}-6 x +\frac{7x^{-2}}{2}+\frac{7 x^2}{2} -{2}x^{-3}-2 x^3+ \frac{x^4}{2}+\frac{x^{-4}}{2}$ \end{footnotesize} \\
\hline $4_1$ & $3 - x^{-1} - x$ & $0$ & $5 -4x^{-1}-4 x+ x^{-2}+x^{2}$ \\
\hline
\end{tabular}
\label{tab:AlexP}
\end{table}

From \eqref{pertknotcompl}--\eqref{Rmrational} we get an asymptotic expansion
\begin{align}
J_n (e^{\hbar})
& = \frac{1}{\Delta_K (x)} + \frac{P_1 (x)}{\Delta_K (x)^{3}} \hbar + \frac{P_2 (x)}{\Delta_K (x)^{5}} \hbar^2 + \ldots \label{Jnpertexp} \\
& = \sum_{m,j=0}^{\infty} c_{m+j,j} \, (\hbar n)^j \hbar^{m} \; = \; \sum_{m=0}^{\infty} \sum_{j=0}^m c_{m,j} \, n^j \hbar^m
\notag
\end{align}
 
The coefficients $c_{m,j}$ that appear in the above expansion are  Vassiliev invariants of the knot; cf.  \cite{BarNatan}, \cite{BirmanLin}. As proposed in Conjecture~\ref{conj:Borel}, the series $f_K (x,q)$ should  be a repackaging of these coefficients, obtained through resurgence via Borel resummation:
\be
\xymatrixcolsep{9pc}\xymatrix{
\boxed{~{\text{Vassiliev} \atop \text{invariants $c_{m,j}$}}~} \quad \ar@/^/[r]^{\text{resurgence}} &
\quad \boxed{~\phantom{\oint} f_K (x,q) \phantom{\oint}~} \ar@/^/[l]^{\text{$x = q^n$ and $q = e^{\hbar} \to 1$}}}
\ee

The resummation of a double series with variables $\hbar$ and $n$ into a series with variables $q$ and $x$ is a problem in {\it parametric resurgence}. Parametric resurgence has been used in the Mathieu equation, in matrix models, and in other problems of mathematical physics; see \cite{Dunne:2016qix,Ahmed:2017lhl,Ahmed:2018gbt}. In our case, the resurgence will be in the variable $\hbar$ (which upon resummation turns into $q$) and the role of parameter can be played either by $n$ or $x = e^{\hbar n}$; these should give the same answer.

\begin{figure}[ht]
\centering
\includegraphics[trim={0 0in 0 0in},clip,width=3.0in]{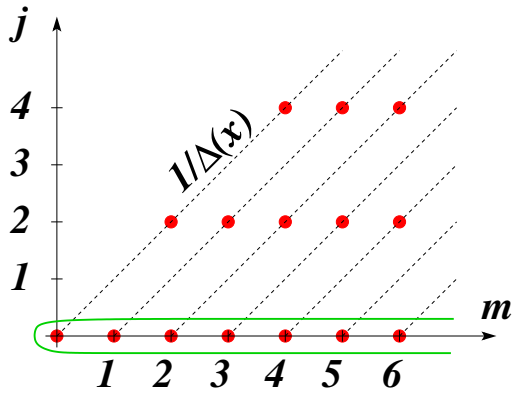}
\caption{The Vassiliev invariants $c_{m,j}$ are non-zero only for $j \le m$.
Encircled in green are the terms $\sum_m c_{m,0} \hbar^m$,
whose resummation gives the $q$-series $f_K (x=1,q)$.}
\label{fig:cmj}
\end{figure}

Observe that in \eqref{Jnpertexp} we used the {\it normalized} (or, {\it reduced}) version $J_n(q)$ of the colored Jones polynomial.
Its resummation should give a function $f_K (x,q)$ symmetric under $x \leftrightarrow x^{-1}$:
\be
f_K (x^{-1},q) \; = \; f_K (x,q)
\ee
If, instead, in \eqref{Jnpertexp} we used the unnormalized (unreduced) version $\tJ_n(q)$ of the colored Jones polynomial,
then the result of the resummation would be
\be
\FKU (x,q) \; := \; \frac{F_K (x,q)}{q^{1/2} - q^{- 1/2}}  \; = \; \frac{x^{1/2} - x^{-1/2}}{q^{1/2} - q^{- 1/2}} \, f_K (x,q)
\ee
antisymmetric under $x \leftrightarrow x^{-1}$. In practice, it is helpful to multiply this with the overall factor of $q^{1/2} - q^{- 1/2}$ and work with $F_K(x,q)$.

In Sections~\ref{sec:ResurgentTrefoil} and \ref{sec:recursion} we will present evidence for Conjecture~\ref{conj:Borel} by explicitly resumming \eqref{Jnpertexp} for a few simple knots and uncovering the knot invariant $F_K (x,q)$, or some of its terms.

\subsection{Relation to the Alexander polynomial}
\begin{table}
\centering
\begin{tabular}{ccc}
\hline\hline
Knot &  & $\phantom{\oint_{\oint_{\oint}}^{\oint^{\oint}}} 2 f_K (x, q \to 1)$
\\
\hline\hline
${\bf 3_1}$ & & $x+\frac{1}{x}+x^2+\frac{1}{x^2}-x^4-\frac{1}{x^4}-x^5-\frac{1}{x^5}+x^7+\frac{1}{x^7}+x^8+\frac{1}{x^8}
-x^{10}-\frac{1}{x^{10}}
+\ldots$ \\[2ex]
${\bf 4_1}$ & & $-x-\frac{1}{x}-3 x^2-\frac{3}{x^2}-8 x^3-\frac{8}{x^3}-21 x^4-\frac{21}{x^4}-55 x^5-\frac{55}{x^5}-144 x^6-\frac{144}{x^6}$ \\
 & & $-377 x^7-\frac{377}{x^7}-987 x^8-\frac{987}{x^8}-2584 x^9-\frac{2584}{x^9}-6765 x^{10}-\frac{6765}{x^{10}}
+\ldots$ \\[2ex]
${\bf 5_1}$ & & $x^2+\frac{1}{x^2}+x^3+\frac{1}{x^3}-x^7-\frac{1}{x^7}-x^8-\frac{1}{x^8}+x^{12}+\frac{1}{x^{12}}+x^{13}+\frac{1}{x^{13}}
-x^{17}-\frac{1}{x^{17}}
+\ldots$ \\[2ex]
${\bf 5_2}$ & & $\frac{x}{2}+\frac{1}{2 x}+\frac{3 x^2}{4}+\frac{3}{4 x^2}+\frac{5 x^3}{8}+\frac{5}{8 x^3}+\frac{3 x^4}{16}+\frac{3}{16 x^4}
-\frac{11 x^5}{32}-\frac{11}{32 x^5}-\frac{45 x^6}{64}-\frac{45}{64 x^6}$ \\
 & & $-\frac{91 x^7}{128}-\frac{91}{128 x^7}-\frac{93 x^8}{256}-\frac{93}{256 x^8}+\frac{85 x^9}{512}+\frac{85}{512 x^9}
 +\frac{627 x^{10}}{1024}+\frac{627}{1024 x^{10}}
+\ldots$
\\[1ex]
\hline\hline
\end{tabular}
\caption{The specialization of $f_K (x,q)$ in the limit $q \to 1$ for knots with up to 5 crossings.}
\label{tab:FAlex}
\end{table}

\begin{table}
\centering
\begin{tabular}{ccc}
\hline\hline
Knot &  & $\phantom{\oint_{\oint_{\oint}}^{\oint^{\oint}}}
2 F_K (x,q \to 1)$
\\
\hline\hline
${\bf 3_1}$ & & $-\sqrt{x}+\frac{1}{\sqrt{x}}+x^{5/2}-\frac{1}{x^{5/2}}+x^{7/2}-\frac{1}{x^{7/2}}
-x^{11/2}+\frac{1}{x^{11/2}}-x^{13/2}+\frac{1}{x^{13/2}}
+\ldots$ \\[2ex]
${\bf 4_1}$ & & $\sqrt{x}-\frac{1}{\sqrt{x}}+2 x^{3/2}-\frac{2}{x^{3/2}}+5 x^{5/2}-\frac{5}{x^{5/2}}
+13 x^{7/2}-\frac{13}{x^{7/2}}+34 x^{9/2}-\frac{34}{x^{9/2}}$ \\
 & & $+89 x^{11/2}-\frac{89}{x^{11/2}}
+233 x^{13/2}-\frac{233}{x^{13/2}}+610 x^{15/2}-\frac{610}{x^{15/2}}
+\ldots$ \\[2ex]
${\bf 5_1}$ & & $-x^{3/2}+\frac{1}{x^{3/2}}+x^{7/2}-\frac{1}{x^{7/2}}+x^{13/2}-\frac{1}{x^{13/2}}
-x^{17/2}+\frac{1}{x^{17/2}}-x^{23/2}+\frac{1}{x^{23/2}}
+\ldots$ \\[2ex]
${\bf 5_2}$ & & $-\frac{\sqrt{x}}{2}+\frac{1}{2 \sqrt{x}}-\frac{x^{3/2}}{4}+\frac{1}{4 x^{3/2}}
+\frac{x^{5/2}}{8}-\frac{1}{8 x^{5/2}}+\frac{7 x^{7/2}}{16}-\frac{7}{16 x^{7/2}}+\frac{17 x^{9/2}}{32}-\frac{17}{32 x^{9/2}}$ \\
 & & $+\frac{23 x^{11/2}}{64}-\frac{23}{64 x^{11/2}}
+\frac{x^{13/2}}{128}-\frac{1}{128 x^{13/2}}-\frac{89 x^{15/2}}{256}+\frac{89}{256 x^{15/2}}
+\ldots$ \\[1ex]
\hline\hline
\end{tabular}
\caption{The specialization of $F_K(x,q)$ in the limit $q\to 1$ for knots with up to 5 crossings.}
\label{tab:FAlexunnormalized}
\end{table}

 The Alexander polynomial  $\Delta_K (x)$ of a knot $K$ is symmetric under $x \leftrightarrow x^{-1}$ and takes values in $\Z [x,x^{-1}]$.

We define $$\se(1/\Delta_K (x)) \in \mathbb{Z} [[x,x^{-1}]]$$ to be the half-sum of the power series expansions
in $x$ and $1/x$ at $x=0$ and $x = \infty$, respectively; compare Section~\ref{sec:conventions}. Then, from \eqref{Jnpertexp} and the symmetry properties of $f_K$ and $F_K$, we see that we should expect the relation
\be
\lim_{q \to 1} f_K (x,q) \; = \; \se\Bigl(\frac{1}{\Delta_K (x)}\Bigr)
\ee
and its unnormalized version:
\be
\lim_{q \to 1} \; F_K(x,q) \; = \; \se\left(\frac{x^{1/2} - x^{- 1/2}}{\Delta_K (x)}\right)
\ee
Both versions of this relation are illustrated in Tables~\ref{tab:FAlex} and \ref{tab:FAlexunnormalized}.

\subsection{An example: the trefoil}
\label{sec:ResurgentTrefoil}
We now explain how resurgence works for a specific example, the right-hand trefoil $K =\RHT=T(2,3)$. 
In our conventions (cf. Section~\ref{sec:conventions}), from the formula~\eqref{eq:hikami} we get that $K$ has the following normalized colored Jones polynomials
\begin{align}
J_1 (q) & = 1
\nonumber \\
J_2 (q) & = q^{-1} + q^{-3} - q^{-4}
\label{Jntref} \\
J_3 (q) & = q^{-2} + q^{-5} - q^{-7} + q^{-8} - q^{-9} - q^{-10} + q^{-11}
\nonumber \\
& \vdots
\nonumber
\end{align}
By setting $q=e^{\hbar}$, expanding in $\hbar$, and then looking at the coefficient of each $\hbar^k$ as we vary $n$, we find polynomial expressions in $n$. From here we can compute the first few terms in the expansion \eqref{Jnpertexp},
\begin{align}
J_n (q = e^{\hbar})
& = 1
\nonumber \\
& + (1-n^2) \hbar^2
\nonumber \\
& + (-2 + 2n^2) \hbar^3
\label{Jnmtrefoilexp} \\
& + \left( \frac{73}{12} - 7 n^2 + \frac{11}{12} n^4 \right) \hbar^4
\nonumber \\
& + \left( - \frac{43}{2} + \frac{79}{3} n^2 - \frac{29}{6} n^4 \right) \hbar^5
\nonumber \\
& + \ldots
\nonumber
\end{align}
We could also write this ``diagonally,'' in terms of the variable $n \hbar$ instead of $\hbar$. The ``first diagonal'' of this double series indeed agrees with the power series expansion of $\frac{1}{\Delta (e^{n \hbar})}$,
\be
\frac{1}{e^{n\hbar} + e^{- n\hbar} - 1} \; = \; 1 - (n\hbar)^2 + \frac{11}{12} (n\hbar)^4 - \frac{301}{360} (n\hbar)^6 + \ldots
\ee
This works even better: for other diagonals, we find a polynomial $P_m(x)$ divided by a corresponding power of $\Delta(e^{n\hbar})$, so we can find the polynomials $P_m(x)$ experimentally.

In fact, the problem of finding $P_m(x)$ for torus knots was tackled by Rozansky in \cite[Section 2]{Rozansky96}. Following his methods,
we obtain the following result, close to his Equation (2.2):
\be
J_{\RHT, n}(q) \; = \; q^{\frac{23}{24}}
\sum_{m=0}^{\infty} \frac{1}{m!} \left( \frac{\hbar}{24} \right)^m
\frac{\partial^{2m}}{ \partial y^{2m}} \frac{e^y x^{\frac{1}{2}} - e^{-y} x^{-\frac{1}{2}}}{1 + (e^y x^{\frac{1}{2}} - e^{-y} x^{-\frac{1}{2}})^2}\Big|_{y=0}
\label{Jnmtrefalaroz}
\ee
When all factors of $q = e^{\hbar}$ and $x = e^{n \hbar} = q^n$ are expanded in $\hbar$ and $n$,
this agrees with \eqref{Jnmtrefoilexp}.

We will do resurgence in $\hbar$ with parameter $x$. The simplest case is $x=1$, which corresponds to $n=0$.
In the notations of \eqref{Jnpertexp}, the perturbative series at $n=0$ is $\sum_{m=0}^{\infty} c_{m,0} \hbar^m$.
Following \cite{GMP}, let us introduce
\be
\Zpert (\hbar) \; = \; \sum_{m=0}^{\infty} a_m \hbar^{m + \frac{1}{2}}
\; := \;
\sqrt{\hbar} q^{- \frac{23}{24}} \sum_{m=0}^{\infty} c_{m,0} \hbar^m
\ee
and its Borel transform
\be
BZ (\xi) \; := \; \sum_{m=0}^{\infty} \frac{a_m}{\Gamma (m + \frac{1}{2})} \xi^{m - \frac{1}{2}}
\ee
Using $\Gamma (m+\frac{1}{2}) = \frac{\sqrt{\pi}}{4^m} \cdot \frac{(2m)!}{m!}$,
we learn that
\be
BZ (\xi) \; = \; \frac{1}{\sqrt{\xi}} \sum_{m=0}^{\infty}
a_m \frac{4^m}{\sqrt{\pi}} \cdot \frac{m!}{(2m)!} \xi^m
\; = \; \frac{1}{\sqrt{\pi \xi}} \sum_{m=0}^{\infty}
b_m \frac{m!}{(2m)!} \left( - \frac{\xi}{6} \right)^m
\label{BZtrefcompl}
\ee
where 
$$ b_m = (-24)^m a_m.$$
 In this form, the right-hand side of \eqref{BZtrefcompl}
can be experimentally related to the logarithmic derivative of the simple rational function
\be
\frac{z-z^{-1}}{1+(z-z^{-1})^2} \; = \;
\frac{z^2 - z^{-2}}{z^3 + z^{-3}} \; = \;
\frac{1}{2} \sum_{m=0}^{\infty} \eps_m \left( z^m - z^{-m} \right)
\label{zzzzzz}
\ee
where $\eps_m$ are precisely the coefficients introduced in \eqref{eq:epstrefoil}. Observe that the Alexander polynomial of the trefoil appears in the denominator of \eqref{zzzzzz}. Alternatively, we could get the same rational fraction from Rozansky's formula \eqref{Jnmtrefalaroz}.

As we shall see shortly, the rational function in \eqref{zzzzzz} is basically the sought-after Borel transform of the original series in $\hbar$. We can relate it to \eqref{BZtrefcompl} by writing $z = e^y$
and differentiating with respect to $y$. We get
$$
\frac{1}{2} \sum_{m=0}^{\infty} m \eps_m \left( z^m + z^{-m} \right)
\; = \; - \frac{z(z^2+1)(z^4-3z^2+1)}{(z^4-z^2+1)^2}
\; = \;
2 \sum_{m=0}^{\infty} b_m \frac{m!}{ (2m)! } \left( -y^{2} \right)^m
$$
Comparing this with the right-hand side of \eqref{BZtrefcompl}, we see that the two expressions match if we identify
\be
y^2 = \frac{\xi}{6}
\label{yviaxi}
\ee
In other words, the exact Borel transform is
\be
\widetilde{BZ}(\xi) \; = \; - \frac{1}{2 \sqrt{\pi \xi}}
\frac{z(z^2+1)(z^4-3z^2+1)}{(z^4-z^2+1)^2}
\ee
with the above identification of variables $z = e^y = e^{\sqrt{\xi/6}}$.

Now, what remains is to perform the inverse Borel transform, {\it i.e.} an integral over $\xi$.
In general, if the Borel transform has the form
\be
\widetilde{BZ} (\xi) \; = \; \frac{1}{\sqrt{\pi \xi}} \sum_{m} c_m e^{- 2 \sqrt{\frac{\pi i m \xi}{p}}}
\ee
then the inverse Borel transform gives the $q$-series ({\it cf.} eq.(3.38) in \cite{GMP}):
\be
f(q) \; = \;
\sqrt{\frac{2\pi i}{\hbar}} \int_{i\mathbb{R}_+}  \, \widetilde{BZ} (\xi) \, e^{- \frac{2\pi i \xi}{\hbar}} d \xi
\; = \; \sum_{m} c_m q^{\frac{m}{p}}
\label{FfromBZnice}
\ee
In other words, the inverse Borel transform acts as a familiar to us ``Laplace transform'' with respect to powers of $z$.
Luckily, our expression for $\widetilde{BZ}(\xi)$ is most conveniently presented in terms of the variable $z$ anyway.
Therefore, applying \eqref{FfromBZnice} we see that in our case $p=24$ and
\be
f(q) \; = \; \frac{q}{2} \sum_{m} m \eps_m \, q^{\frac{m^2-1}{24}}
\ee
This is precisely the answer we got for $f_K(1,q)$ of the trefoil, obtained from \eqref{Fxq31unnorm} by dividing by $x^{1/2} - x^{-1/2}$, and then taking the limit as $x\to 1$. 

The resurgent analysis above, in fact, extends to general values of $x$. Indeed, using the simple change of variables \eqref{yviaxi}, we can write the $\xi$-integral \eqref{FfromBZnice}
as a $y$-integral which, in turn, following Rozansky's formula~\eqref{Jnmtrefalaroz}, is a limit ($x \to 1$) of a more general family of integrals parametrized by $x$ ({\it cf.} \cite{Rozansky96}):
\be
f_K(x,q) \; = \; \frac{q^{\frac{23}{24}} }{x^{\frac{1}{2}} - x^{-\frac{1}{2}} }
\sqrt{\frac{6}{\pi \hbar}} \int_{\mathbb{R}} dy \, e^{- 6 y^2/\hbar} \,
\frac{e^y x^{\frac{1}{2}} - e^{-y} x^{-\frac{1}{2}}}{1 + (e^y x^{\frac{1}{2}} - e^{-y} x^{-\frac{1}{2}})^2}
\label{FtrefAlexint}
\ee
On the one hand, for all values of $x$, such integrals have the form of the inverse Borel transform. With the help of the familiar formulae \eqref{zzzzzz} and \eqref{FfromBZnice}, we get the answer in   \eqref{Fxq31unnorm}:
\be
F_K (x,q) \; = \; ( x^{\frac{1}{2}} - x^{-\frac{1}{2}}) f_K(x,q) \; = \;
\frac{q}{2} \sum_{m=1}^{\infty} \eps_m (x^{\frac{m}{2}} - x^{-\frac{m}{2}}) \, q^{\frac{m^2-1}{24}}.
\ee
On the other hand, using the general formula,
\be
\int_{\mathbb{R}} e^{- y^2/\hbar} f(y) dy\; = \;
\sqrt{\pi \hbar} \sum_{m=0}^{\infty} \frac{f^{(2m)} (0)}{m!} \left( \frac{\hbar}{4} \right)^m
\ee
we immediately reproduce the perturbative $\hbar$-expansion \eqref{Jnmtrefalaroz} from \eqref{FtrefAlexint}.
This means that the integrand
$$\widetilde{BZ}(x, y) \sim \frac{e^y x^{1/2} - e^{-y} x^{-1/2}}{1 + (e^y x^{1/2} - e^{-y} x^{-1/2})^2}$$
is indeed the Borel transform of the original perturbative series~\eqref{Jnmtrefoilexp}.

\newpage \section{Recursion}
\label{sec:recursion}

\subsection{The general principle}
Consider the character variety of a surface, $\Char(\Sigma) = \mathcal{M}_{\text{flat}} (\Sigma, G_{\C})$. 
Quantization replaces the algebra of functions on $\Char(\Sigma)$ by an algebra of operators. In particular, for $\Sigma = T^2$ and $G=SU(2)$ the classical $\C^*$-valued holonomy eigenvalues $x$ and $y$ become operators (which, for simplicity, we denote by the same letters)
that no longer commute, but rather ``$q$-commute'':
\be
y x = q x y.
\ee
Correspondingly, as in the standard deformation quantization, the A-polynomial of a knot, $A(x,y)$, turns into a $q$-difference operator $$\hA = \hA_q (x,y).$$ This is called the {\em quantum (or noncommutative) A-polynomial}.

It was proposed independently in \cite{Garoufalidis} and \cite{Gukov:2003na} that the colored Jones function
$$ J_K : \mathbb{N} \to \Z[q^{-1}, q], \ \ n \mapsto J_{K, n}(q)$$
satisfies a recursion given by $\hA$. Here, the variable $x$ acts by multiplication by $q^n$, and the variable $y$ acts by shifting the index $n$ by $1$, i.e., taking $J_n(q)$ to $J_{n+1}(q)$.

Mathematically, it was proved that the colored Jones function is $q$-holonomic, that is, it satisfies a finite linear recursion with polynomial coefficients \cite{GaroufalidisLeHolonomic}. The $\hA$-polynomial is then defined 
as the generator of its recurrence ideal \cite{Garoufalidis}, and the AJ conjecture says that its specialization to $q=1$ is the usual A-polynomial. The AJ conjecture remains open in general, although it has been proved for several families of knots; see \cite{LeTran} and \cite{LeZhang}. 

In practice, the operator $\hA$ can be computed either numerically \cite{GaroufalidisLeHolonomic} by looking for a recursion satisfied by $J_n (q)$,
or via deformation quantization \cite{Dimofte:2010ep} of the algebra of functions on $(\mathbb{C}^* \times \mathbb{C}^*) / \Z_2$, or via the B-model (``topological recursion'') applied to the classical curve $A (x,y) = 0$ \cite{Gukov:2011qp}. See {\it e.g.} \cite{Gukov:2012jx} for a review.

Recall that Conjecture~\ref{conj:Borel} says that the series $f_K(x,q)$ is obtained from the colored Jones polynomials by resurgence through Borel resummation. Since the colored Jones function is annihilated by $\hA_q (x,y)$, one expects that so is $f_K(x,q)$, where now $x$ acts by multiplication by $x$ and $y$ takes the variable $x$ to $xq$.

In fact, it was argued in \cite{Gukov:2003na} that any partition function of $\sl$ Chern-Simons theory on the knot complement should be annihilated by $\hA$. Later, this was justified in the framework of 3d-3d correspondence. In particular, this applies to $f_K(x,q)$ of our interest here, leading to the first claim in Conjecture~\ref{conj:Aq}:
\be
\label{eq:haha}
\hA \; f_K(x,q) \; = \; 0.
\ee
Usually, this equation is solved for the normalized version of the colored Jones polynomial or some
other $SL(2,\C)$ partition function. Here, it is convenient to work with it in its unnormalized form: by conjugating $\hA$ with $x^{\frac{1}{2}} - x^{-\frac{1}{2}}$, we obtain an operator $\tA$ which annihilates $\tJ_n(q)$. We can rephrase \eqref{eq:haha} as
\be
\label{eq:tata}
\tA \; F_K(x,q) \; = \; 0
\ee
Let us write
\be
\label{Ffgeneralform}
F_K (x,q) =\frac{1}{2} \sum_{m \geq 1}  f_m(q) \cdot (x^{\frac{m}{2}} - x^{-\frac{m}{2}})\ee 
as in \eqref{eq:FKf}, where $f_m(q)$ can be non-zero only for $m$ odd. Then, Equation~\eqref{eq:tata}, after clearing denominators, becomes
\be
a_0 f_{m} + a_1 f_{m+1} + a_2 f_{m+2} + \ldots + a_n f_{m+n} = 0, \label{newrecursion}
\ee
where $f_i=f_i(q)$ and $a_i=a_i(q)$ are Laurent power series in $q$. This form of the $q$-difference equation appears to be new and has not been discussed in the literature so far.

To start the recursion given by \eqref{newrecursion}, we need to know the initial values $f_1(q), \dots, f_n(q)$.
In view of \eqref{Jnpertexp}, we should look for a solution of \eqref{Ffgeneralform} of the form 
\be
\label{FP}
F_K (x,q=e^{\hbar}) \; = \;  (x^{1/2} - x^{-1/2})\cdot \Bigl( \frac{1}{\Delta_K (x)} + \frac{P_1 (x)}{\Delta_K (x)^{3}} \hbar
+ \frac{P_2 (x)}{\Delta_K (x)^{5}} \hbar^2 + \frac{P_3 (x)}{\Delta_K (x)^{7}} \hbar^3 + \ldots \Bigr).
\ee
We will first apply the recursion \eqref{eq:tata} to the right hand side of \eqref{FP}. This will determine the Laurent polynomials $P_k (x)$ up to an arbitrary order $k$. Indeed, solving the recursion order-by-order in $\hbar$, we obtain a linear PDE for each $P_k (x)$. A convenient way to fix the corresponding ``integration constant'' is to require that the value of $P_k (x)$ at $x=1$ ({\it i.e.} $n=0$) matches the coefficient of $\hbar^k$ in \eqref{Jnpertexp}. This coefficient $c_{k, 0}$ can be found by expanding the colored Jones polynomials in $\hbar$, as we did for the trefoil in \eqref{Jnmtrefoilexp}.

Once we know the polynomials $P_k(x)$, we read off the coefficients of $x^{m/2}$ (as power series in $\hbar$) from the right hand side of \eqref{FP}, for $m=1, \dots, n$. We then convert these power series in $\hbar$ into Laurent power series in $q=e^{\hbar}$. In general, this conversion means doing resurgence, but in the two examples below (the trefoil and the figure-eight) the power series in $\hbar$ are seen experimentally to be just finite Laurent polynomials in $q = e^{\hbar}$. These give the desired initial conditions $f_1(q), \dots, f_n(q)$, allowing the recursion \eqref{Ffgeneralform} to start. The result of the recursion is the series $F_K(x,q)$.

\begin{remark}
Both versions $F_K (x,q)$ and $f_K(x,q)$ could be used for the recursion analysis. However, since Tables~\ref{tab:FAlex} and \ref{tab:FAlexunnormalized} suggest that the unnormalized version $F_K(x,q)$ may be easier to deal with, we chose to use that one. 
\end{remark}

\begin{remark}
The coefficients of each power of $x$ in $F_K(x,q)$ are finite Laurent polynomials in $q$ for all algebraic knots (compare Section~\ref{sec:negative}), and for the figure-eight knot. However, this cannot be true in general. For example, the knot $K={\bf 5_2}$ has non-monic Alexander polynomial; hence, the specialization $q \to 1$ of $F_K(x,q)$, which is $\se(1/\Delta(x))$, does not have integer coefficients. This can be seen from Table~\ref{tab:FAlexunnormalized}. Therefore, in this case, we expect the coefficients of the powers of $x$ in $F_K(x,q)$ to be infinite Laurent power series in $q$, rather than Laurent polynomials.
\end{remark}

\subsection{The trefoil} 
The $\hA$-polynomial for the right-handed trefoil $K = \RHT$ can be read off, for example, from \cite[Section 3.2]{Garoufalidis} or \cite[Example 6]{Gukov:2012jx}. Conjugating it with $x^{\frac{1}{2}} - x^{- \frac{1}{2}}$, we find the recursion relation \eqref{eq:tata} for $F_K(x,q)$ in this case:
\be
\alpha (x; q) F_K (x , q)
+ \beta (x; q) F_K (x q , q)
+ \gamma (x; q) F_K (x q^2 , q)
\; = \; 0
\label{twosteprecF}
\ee
where
\begin{align}
\alpha (x;q) & = \frac{q^3 x^2 - 1}{q^4 x^3 (qx^2 - 1)}
\; = \; \tfrac{1}{x^3} - \tfrac{2 (x^2-2)}{x^3 (x^2-1)} \hbar
+ \tfrac{2 (x^4-6 x^2+4)}{x^3 (x^2-1)^2} \hbar^2 + O \left( \hbar^3 \right) \nonumber \\
\beta (x;q) & = \frac{q^5 x^5 - q^2 x^3 - q x^2  + 1}{q^{9/2} x^3 (qx^2 - 1) }
\; = \; \tfrac{x^3 - 1}{x^3} - \tfrac{x^5 - 5 x^3 - 9 x^2 + 9}{2 x^3 (x^2 - 1)} \hbar + \\
& + \tfrac{x^7-42 x^5-81 x^4+25 x^3+162 x^2-81}{8 (x-1)^2 x^3 (x+1)^2} \hbar^2
+ O \left( \hbar^3 \right) \nonumber \\
\gamma (x;q) & = -1 \nonumber
\end{align}

We apply this recurrence to find the first few Laurent polynomials $P_k (x)$. We use the initial conditions $$P_1 (1) =0, \ \ P_2 (1) =1,\  \ P_3 (1) =-2,\ \ P_4 (1) =\frac{73}{12}, \ \ \dots$$
which follow from \eqref{Jnmtrefoilexp}. The results of the recursion are tabulated in Table~\ref{tab:Pntrefoil}.

Plugging the polynomials $P_k(x)$ into ~\eqref{FP}, we obtain
\begin{align*}
-2F_K(x, e^{\hbar}) &= 
 (x^{1/2} - x^{-1/2} - x^{5/2} + x^{-5/2} - x^{7/2} + x^{-7/2} + \dots)\\
&+ \hbar (x^{1/2} -x^{-1/2} - 2 x^{5/2} + 2 x^{-5/2} - 3 x^{7/2} + 3 x^{-7/2} +\dots ) \\
& + \frac{\hbar^2}{2} (x^{1/2} - x^{-1/2} - 4 x^{5/2} + 4 x^{-5/2} - 9 x^{7/2} + 9 x^{-7/2}+\dots) \\
& + \frac{\hbar^3}{6} (x^{1/2} - x^{-1/2} - 8 x^{5/2} + 8 x^{-5/2} - 27 x^{7/2} +27 x^{-7/2}+\dots) \\
 &+ \frac{\hbar^4}{24} (x^{1/2} - x^{-1/2} - 16 x^{5/2} + 16 x^{-5/2} - 81 x^{7/2} + 81 x^{-7/2}+\dots)  \\
 & + \frac{\hbar^5}{120} (x^{1/2} -x^{-1/2} - 32 x^{5/2} + 32 x^{-5/2} - 243 x^{7/2} + 243 x^{-7/2}+\dots) \\
&  +\dots
\end{align*}

From here, we find the initial conditions 
$$ f_1 = -q, \ \ f_3=0, \ \ f_5 = q^2, \ \ f_7 = q^3,\ \  f_9=0 $$
for the recursion \eqref{newrecursion}, which is
\be
\label{eq:rectrefoil}
f_{m+10} = \frac{q^3}{1 - q^{\frac{m}{2} + \frac{9}{2}} } \left[ f_m (q^{\frac{m}{2} + \frac{3}{2}} - q^{m+2}) + f_{m+4} (q^{m+5} - q^{\frac{m}{2} + \frac{1}{2}}) + f_{m+6} ( 1- q^{\frac{m}{2} + \frac{1}{2}})
\right]. \ee
Note that the steps in $m$ are multiples of 2 in our notation (only odd values of $m$ give nonzero terms), so this is a $5$-step recursion.

Solving the recursion \eqref{eq:rectrefoil} experimentally up to any desired order $m$, we find that
\be
\label{eq:fme}
f_m (q) \; = \eps_m q^{\frac{m^2 + 23}{24}}.
\ee
This is what we expected from~\eqref{Fxq31unnorm}. Of course, a posteriori, one can also check directly that the functions in \eqref{eq:fme} satisfy the recursion \eqref{eq:rectrefoil}.

\begin{table}
\centering
\begin{tabular}{rl}
\hline\hline
$P_1 \; =$ & $x^2+\frac{1}{x^2}-2 x-\frac{2}{x}+2 \phantom{\oint_{\oint_{\oint}}^{\oint^{\oint}}}$ \\[2ex]
$P_2 \; =$ & $\frac{x^4}{2}+\frac{1}{2 x^4}-2 x^3-\frac{2}{x^3}+\frac{7 x^2}{2}+\frac{7}{2 x^2}-6 x-\frac{6}{x}+9$ \\[2ex]
$P_3 \; =$ & $\frac{x^6}{6}+\frac{1}{6 x^6}-x^5-\frac{1}{x^5}+\frac{7 x^4}{3}+\frac{7}{3 x^4}-\frac{17 x^3}{3}
-\frac{17}{3 x^3}+\frac{46 x^2}{3}+\frac{46}{3 x^2}-\frac{49 x}{3}-\frac{49}{3 x}+\frac{25}{3}$ \\[2ex]
$P_4 \; =$ & $\frac{x^8}{24}+\frac{1}{24 x^8}-\frac{x^7}{3}-\frac{1}{3 x^7}+\frac{7 x^6}{8}+\frac{7}{8 x^6}-3 x^5-\frac{3}{x^5}+\frac{117 x^4}{8}+\frac{117}{8 x^4}-16 x^3-\frac{16}{x^3}$ \\
& $-\frac{193 x^2}{12}-\frac{193}{12 x^2}-\frac{82 x}{3}-\frac{82}{3 x}+\frac{201}{2}$ 
\\[1ex]
\hline\hline
\end{tabular}
\caption{Laurent polynomials $P_k (x)$ for the trefoil knot $K = {\bf 3_1}$.}
\label{tab:Pntrefoil}
\end{table}

\subsection{Figure-eight}
\label{sec:f8}
The normalized version of the colored Jones polynomial for the figure-eight knot $K={\bf 4_1}$ is given by the following formula; cf. \cite[Section 6.2]{GaroufalidisLeHolonomic}:
\be
\label{eq:cJ8}
 J_n(q) = 1 + \sum_{m=1}^{n-1} \prod_{j=1}^m (q^n + q^{-n} - q^j- q^{-j}).
 \ee
The first few polynomials are
\begin{align}
J_1(q) & = 1
\nonumber \\
J_2(q) & = q^{-2} - q^{-1} + 1 - q + q^2
\label{fig8Jones} \\
J_3(q) & = q^{-6} - q^{-5} - q^{-4} + 2 q^{-3} - q^{-2} - q^{-1} + 3 - q - q^2 + 2 q^3 - q^4 - q^5 + q^6
\nonumber \\
& \vdots
\nonumber
\end{align}
Expanding them in powers of $\hbar$ we get
\begin{align}
J_n (q = e^{\hbar})
& = 1
\nonumber \\
& + (- 1+n^2) \hbar^2
\label{Jnfig8exp} \\
& + \left( \frac{47}{12} - 5 n^2 + \frac{13}{12} n^4 \right) \hbar^4
\nonumber \\
& + \left( -\frac{12361}{360} + \frac{571}{12} n^2 -\frac{173}{12} n^4 + \frac{421}{360} n^6 \right) \hbar^6
\nonumber \\
& + \ldots
\nonumber
\end{align}
Thus, the coefficients $c_{k, 0}$ in \eqref{Jnpertexp} are:
\be
\label{eq:ck08}
1 \,, \quad -1 \,, \quad \frac{47}{12} \,, \quad -\frac{12361}{360} \,, \quad \ldots
\ee
In fact, in this case, a neat way to find all the coefficients $c_{k, 0}$ is to consider the following function, obtained from \eqref{eq:cJ8} by replacing $q^n$ and $q^{-n}$ with $1$, and taking the summation over $m$ to infinity:
\be
J_0 (q) \; := \; 1+ \sum_{m=1}^{\infty} \prod_{j=1}^m (1 - q^j) (1 - q^{-j})
\ee
Setting $q = e^{\hbar}$ and expanding this function in $\hbar$ produces the coefficients $c_{k, 0}$.

The $\hA$-polynomial of the figure-eight knot appears, for example, in \cite[Section 3.2]{Garoufalidis} or \cite[Section 3.2]{Gukov:2011qp}. We find that the series $F_K(x,q)$ should obey a 3-step recursion relation:
\be
\alpha (x; q) F_K (x , q)
+ \beta (x; q) F_K (x q , q)
+ \gamma (x; q) F_K (x q^2 , q)
+ F_K (x q^3 , q)
\; = \; 0
\label{threesteprecF}
\ee
where
\begin{align}
\alpha (x;q) & = -\frac{(q^2 x+1) (q^5 x^2-1)}{q^{5/2} (q x+1) (q x^2-1)}  \nonumber \\
& = -1 - \tfrac{5 x^2 - 2 x + 5}{ 2 (x^2-1)} \hbar
- \tfrac{25 x^2 - 58 x + 25}{8 (x-1)^2} \hbar^2 + O \left( \hbar^3 \right) \nonumber \\
\beta (x;q) & = \frac{(q^5 x^2-1 ) (q x (q x (q (x (q x-2)-1)+x+1)+q-x-2)+1)}{q^4 x^2 (q x^2-1)} \\
& = \tfrac{x^4-x^3-x^2-x+1}{x^2} + \tfrac{2 (2 x^6-2 x^5-x^4-x^2-2 x+2)}{x^2 (x^2-1)} \hbar
+ \tfrac{8 x^8-7 x^7-15 x^6+11 x^5+10 x^4+11 x^3-15 x^2-7 x+8}{x^2 (x^2-1)^2} \hbar^2
+ O \left( \hbar^3 \right) \nonumber \\
\gamma (x;q) & = -\frac{(q^2 x+1) (q x (q (q x (q (q^2 x-1) (q^2 x+q-1)-1)-2)+1)+1)}{q^{9/2} x^2 (q x+1)} \nonumber \\
& = (-x^2-\tfrac{1}{x^2}+x+\tfrac{1}{x}+1)
-\tfrac{9 x^5 + 4 x^4 - 2 x^3 + 2 x^2 - 4 x - 9}{2 x^2 (x+1)} \hbar
+ O \left( \hbar^2 \right) \nonumber
\end{align}

This recursion, together with the initial conditions given by \eqref{eq:ck08}, produces the Laurent polynomials $P_k(x)$ listed in Table~\ref{tab:Pnfigure8}. Using this method, we can produce explicit expressions for $P_k(x)$ up to any desired order $k$.

\begin{table}
	\centering
	\begin{tabular}{rl}
		\hline\hline
		$P_1 \; =$ & $0 \phantom{\oint_{\oint_{\oint}}^{\oint^{\oint}}}$ \\[2ex]
		$P_2 \; =$ & $x^2+\frac{1}{x^2}-4 x-\frac{4}{x}+5$ \\[2ex]
		$P_3 \; =$ & $0$ \\[2ex]
		$P_4 \; =$ & $\frac{x^6}{12}+\frac{1}{12 x^6}+\frac{2 x^5}{3}+\frac{2}{3 x^5}-\frac{3 x^4}{4}-\frac{3}{4 x^4}-\frac{98 x^3}{3}-\frac{98}{3 x^3}+\frac{293 x^2}{2}+\frac{293}{2 x^2}-\frac{862 x}{3}-\frac{862}{3 x}+\frac{4211}{12}$ \\
		& $~$ \\[2ex]
		$P_5 \; =$ & $0$ \\[2ex]
		$P_6 \; =$ & $\frac{(x^2-3 x+1)^3}{360 x^{10}} (x^{14}+101 x^{13}+3160 x^{12}+12171 x^{11}+8061 x^{10}-102498 x^9+214337 x^8$ \\
		& $~~~~~~~~  -258305 x^7+214337 x^6-102498 x^5+8061 x^4+12171 x^3+3160 x^2+101 x+1)$ \\[2ex]
		$P_7 \; =$ & $0$ \\[2ex]
		$P_8 \; =$ & $\frac{(x^2-3 x+1)^4}{20160 x^{14}} (x^{20}+476 x^{19}+67393 x^{18}+1645236 x^{17}+14061303 x^{16}+8176392 x^{15}$\\ 
		& $~~~~~~~~ -41755650 x^{14}-127433568 x^{13}+583375485 x^{12}-1066253508 x^{11}+1267004367 x^{10}$\\
		& $~~~~~~~~-1066253508 x^9+583375485 x^8-127433568 x^7 -41755650 x^6+8176392 x^5+14061303 x^4$\\
		& $~~~~~~~~+1645236 x^3+67393 x^2+476 x+1)$ \\[2ex]
		$P_9 \; =$ & $0$ \\[2ex]
		$P_{10} \; =$ & $\frac{(x^2-3 x+1)^5}{1814400 x^{18}} (x^{26}+2003 x^{25}+1134523 x^{24}+91512582 x^{23}+2727924123 x^{22} +26367610587 x^{21}$ \\
		& $~~~~~~~~ +80642770303 x^{20} -185974355518 x^{19}-170592137312 x^{18}+55832596182 x^{17}$ \\
		& $~~~~~~~~ +2753722904868 x^{16}-8501480211618 x^{15}+14284755783843 x^{14}-16668636494613 x^{13}$ \\
		& $~~~~~~~~ +14284755783843 x^{12}-8501480211618 x^{11}+2753722904868 x^{10}+55832596182 x^9$ \\
		& $~~~~~~~~-170592137312 x^8-185974355518 x^7+80642770303 x^6+26367610587 x^5+2727924123 x^4$\\
		& $~~~~~~~~+91512582 x^3+1134523 x^2+2003 x+1)$ \\[2ex]		
		& $\phantom{~~~~~} \vdots$
		\\[1ex]
		\hline\hline
	\end{tabular}
	\caption{Laurent polynomials $P_k (x)$ for the figure-8 knot $K = {\bf 4_1}$. }
	\label{tab:Pnfigure8}
\end{table}

This leads to the first terms of the series $F_{K} (x,q=e^{\hbar})$ in terms of $\hbar$ and $x$.

\begin{footnotesize}
\begin{multline}
2F_{K} (x, e^{\hbar}) \; = \; \\
= \left( x^{1/2}-\frac{1}{x^{1/2}} + 2 x^{3/2} -\frac{2}{x^{3/2}} +5 x^{5/2} -\frac{5}{x^{5/2}} +13 x^{7/2}-\frac{13}{x^{7/2}}
+34 x^{9/2} -\frac{34}{x^{9/2}} +89 x^{11/2} -\frac{89}{x^{11/2}} +233 x^{13/2} -\frac{233}{x^{13/2}} + \ldots \right) \\
+ \hbar^2 \left( x^{5/2}-\frac{1}{x^{5/2}}+10 x^{7/2}-\frac{10}{x^{7/2}} +64 x^{9/2} -\frac{64}{x^{9/2}}
+331 x^{11/2} -\frac{331}{x^{11/2}} +1505 x^{13/2}-\frac{1505}{x^{13/2}}+ \ldots \right) \\
+ \hbar^4 \left( \frac{x^{5/2}}{12}-\frac{1}{12 x^{5/2}}+\frac{17 x^{7/2}}{6}-\frac{17}{6 x^{7/2}}
+\frac{142 x^{9/2}}{3}-\frac{142}{3 x^{9/2}}+\frac{6115 x^{11/2}}{12}-\frac{6115}{12 x^{11/2}}
+\frac{50057 x^{13/2}}{12}-\frac{50057}{12 x^{13/2}}+ \ldots \right) \\
+ \hbar^6 \left( \frac{x^{5/2}}{360}-\frac{1}{360 x^{5/2}}+\frac{13 x^{7/2}}{36}-\frac{13}{36 x^{7/2}}
+\frac{818 x^{9/2}}{45}-\frac{818}{45 x^{9/2}}+\frac{154891 x^{11/2}}{360}-\frac{154891}{360 x^{11/2}}
+\frac{472573 x^{13/2}}{72}-\frac{472573}{72 x^{13/2}}+ \ldots \right) \\
+ \hbar^8 \Bigl( \frac{x^{5/2}}{20160}-\frac{1}{20160 x^{5/2}}+\frac{257 x^{7/2}}{10080}-\frac{257}{10080 x^{7/2}}
+\frac{10781 x^{9/2}}{2520}-\frac{10781}{2520 x^{9/2}}
+\frac{916439 x^{11/2}}{4032}-\frac{916439}{4032 x^{11/2}}+\frac{19085471 x^{13/2}}{2880}\\
-\frac{19085471}{2880 x^{13/2}}+ \ldots \Bigr) \\
+ \hbar^{10} \Bigl( \frac{x^{5/2}}{1814400}-\frac{1}{1814400 x^{5/2}}+\frac{41 x^{7/2}}{36288}-\frac{41}{36288 x^{7/2}}
+\frac{9608 x^{9/2}}{14175}-\frac{9608}{14175 x^{9/2}}+\frac{147178651 x^{11/2}}{1814400}-\frac{147178651}{1814400 x^{11/2}}\\
+\frac{47916623 x^{13/2}}{10368}-\frac{47916623}{10368 x^{13/2}}+ \ldots \Bigr)  \\
+ \ldots
\notag
\end{multline}
\end{footnotesize}

By analyzing the coefficients of $x^{m/2}$ for $|m| \leq 13$ in this expression, we find that they are polynomials in $q = e^{\hbar}$. Precisely, we get
\begin{align*}
f_1 & = 1,  \nonumber \\
f_3 & = 2,  \nonumber \\
f_5 & =  1/q + 3+ q,  \nonumber \\
f_7 & = 2/q^2 + 2/q +5+ 2 q + 2 q^2,  \\
f_9 & = 1/q^4 + 3/q^3 + 4/q^2 + 5/q + 8+ 5 q + 4 q^2 + 3 q^3 + q^4,  \nonumber \\
f_{11} & = 2/q^6 + 2/q^5 + 6/q^4 + 7/q^3 + 10/q^2 + 10/q + 15+ 10 q + 10 q^2 + 7 q^3 + 6 q^4 + 2 q^5 + 2 q^6, \nonumber \\
f_{13} & = 1/q^9 + 3/q^8 + 4/q^7 + 7/q^6 + 11/q^5 + 15/q^4 + 18/q^3 + 21/ q^2 + 23/q + 27+23 q \nonumber \\
& ~~+ 21 q^2 + 18 q^3 + 15 q^4 + 11 q^5 + 7 q^6 + 4 q^7 + 3 q^8 + q^9, \nonumber
\end{align*}

These will act as initial conditions for the following 7-step recursion in terms of $f_m(q)$:
\begin{multline}
f_{m+14} \; = \; -\frac{q^{-\frac{m}{2}-\frac{11}{2}}}{q^{\frac{m}{2}+\frac{13}{2}}-1} \Big[
f_m (q^{\frac{m}{2}+\frac{17}{2}}- q^{m+9}) + f_{m+2} (q^{\frac{m}{2}+\frac{15}{2}} - q^{\frac{m}{2}+\frac{17}{2}}
+  q^{m+9} -  q^{m+10})
\\
+ f_{m+4}(- q^{\frac{m}{2}+\frac{11}{2}} -  q^{\frac{m}{2}+\frac{17}{2}} - q^{\frac{m}{2}+\frac{19}{2}} + q^{\frac{3 m}{2}+\frac{21}{2}} + q^{m+8} + q^{m+9} + q^{m+12})
\\
+ f_{m+6}(- q^{\frac{m}{2}+\frac{9}{2}} +  q^{\frac{m}{2}+\frac{11}{2}} -  q^{\frac{m}{2}+\frac{15}{2}}
-  q^{\frac{m}{2}+\frac{17}{2}} +  q^{\frac{3 m}{2}+\frac{25}{2}} +  q^{m+9} + q^{m+10}
-  q^{m+12} +  q^{m+13})
\\
+ f_{m+8} (q^{\frac{m}{2}+\frac{11}{2}} +  q^{\frac{m}{2}+\frac{13}{2}} -  q^{\frac{m}{2}+\frac{17}{2}} + q^{\frac{m}{2}+\frac{19}{2}} -  q^{\frac{3 m}{2}+\frac{31}{2}} - q^{m+8} +  q^{m+9} -   q^{m+11} -q^{m+12})
\\
+ f_{m+10} (q^{\frac{m}{2}+\frac{9}{2}} +  q^{\frac{m}{2}+\frac{11}{2}} +   q^{\frac{m}{2}+\frac{17}{2}} -  q^{\frac{3 m}{2}+\frac{35}{2}} -  q^{m+9} -  q^{m+12} - q^{m+13})
\\
+ f_{m+12} (q^{\frac{m}{2}+\frac{11}{2}} -  q^{\frac{m}{2}+\frac{13}{2}} +  q^{m+11} -   q^{m+12}) - f_{m+4} q^7 - f_{m+6} q^5 + f_{m+8} q^2 + f_{m+10} \Big]
\end{multline}

In this way we can determine $F_{K}(x,q)$ for the figure-eight knot up to any desired order. Written as in \eqref{eq:41}, it will be the anti-symmetrization of a series $\Xi(x,q)$, whose first terms are
\begin{align}
\label{eq:Xixq}
\Xi(x,q) &= x^{1/2} + 2 x^{3/2} + (q^{-1} + 3+ q) x^{5/2} + (2q^{-2} + 2q^{-1}+ 5 +  2 q + 
    2 q^2) x^{7/2} \\
& \phantom{=} + (q^{-4} + 3q^{-3} + 4q^{-2} + 5q^{-1}+ 8 +  5 q + 4 q^2 + 
    3 q^3 + q^4) x^{9/2} \notag \\
 & \phantom{=}    + (2q^{-6} + 2q^{-5} + 6q^{-4} + 7q^{-3} + 10q^{-2} + 10 q^{-1} + 15 + 10 q + 10 q^2 + 7 q^3 + 6 q^4 + 2 q^5 + 2 q^6) x^{11/2}  \notag \\
  & \phantom{=}    +     \dots \notag
    \end{align}

\subsection{Surgeries on the figure-eight}
\label{sec:surgeries8}
Thurston \cite{Thurston} proved that all but nine values of $p/r \in \Q$ produce hyperbolic surgeries on the figure-eight knot. The nine exceptional surgeries are for the coefficients
$$ \frac{p}{r} \in \bigl \{ -4, -3, -2, -1, 0, 1, 2, 3, 4 \bigr \}.$$
The $\pm 4$ and $0$ surgeries are toroidal, and the $\pm 1$, $\pm 2$ and $\pm 3$ surgeries are Seifert fibered:
$$ S^3_1({\bf 4_1}) = - S^3_{-1}({\bf 4_1}) = M\! \left(-1; \frac{1}{2}, \frac{1}{3}, \frac{1}{7} \right),$$
$$ S^3_2({\bf 4_1}) = - S^3_{-3}({\bf 4_1}) = M\! \left(-1; \frac{1}{2}, \frac{1}{4}, \frac{1}{5} \right),$$
 $$  S^3_3({\bf 4_1}) = - S^3_{-3}({\bf 4_1}) = M\! \left(-1; \frac{1}{3}, \frac{1}{3}, \frac{1}{4} \right).$$
See  Figure 6 in \cite{Boyle}.

The $+1$, $+2$ and $+3$ surgeries bound negative definite plumbings. Thus, for those we can apply the plumbing formula \eqref{eq:plumbing1} to compute the invariants $\Zhat_a(q)$. The results are shown  
in Table~\ref{tab:figure8}.

\begin{table}
\centering
\begin{tabular}{ccc}
\hline\hline
\multicolumn{2}{c}{$Y = S^3_p ({\bf 4_1})$} & $\phantom{\oint_{\oint_{\oint}}^{\oint^{\oint}}} \Zhat_a(q)$
\\
\hline\hline
$p=1$ & $\Sigma (2,3,7)$ &
$q^{1/2} (1 - q - q^5 + q^{10} - q^{11} + q^{18} + q^{30} - q^{41} + q^{43} - q^{56} - q^{76}
+\cdots)$ \\[2ex]
$p=2$ & $M\!\left(-1;\frac{1}{2},\frac{1}{4},\frac{1}{5} \right)$ &
$q^{1/4} (1 - q + q^{12} - q^{19} + q^{21} - q^{30} + q^{63} - q^{78} + q^{82} - q^{99} + q^{154}
+\cdots)$ \\
 &  & $q^{7/4} (-1 + q^3 - q^4 + q^9 - q^{31} + q^{42} - q^{45} + q^{58} - q^{102} + q^{121}
+\cdots)$ \\[2ex]
$p=3$ & $M\! \left(-1;\frac{1}{3},\frac{1}{3},\frac{1}{4} \right)$ &
$1 - q + q^6 - q^{11} + q^{13} - q^{20} + q^{35} - q^{46} + q^{50} - q^{63} + q^{88} - q^{105}
+\cdots)$ \\
 &  & $q^{5/3} (-1 + q^3 - q^{21} + q^{30} - q^{66} + q^{81} - q^{135} + q^{156} - q^{228} + q^{255} + \cdots)$ \\[2ex]
\hline\hline
\end{tabular}
\caption{The invariants $\Zhat_a (q)$ for negative definite Seifert fibered surgeries on the figure-8 knot.}
\label{tab:figure8}
\end{table}

Using modularity analysis, from these answers one can also obtain $\Zhat_a(q)$ for the reverse manifolds, which are the $-3$, $-2$, and $-1$ surgeries. In particular, for $$S^3_{-1}({\bf 4_1})=-\Sigma(2,3,7)$$ the  invariant was computed in \cite[Equation (7.21)]{CCFGH}, yielding (up to a power of $q$) Ramanujan's mock theta function $F_0(q)$ of order $7$: 
\begin{align}
\label{eq:-237}
\Zhat_0(-\Sigma(2,3,7)) &=-q^{-1/2} \sum_{n \geq 0} \frac{q^{n^2}}{(q^{n+1})_n} \\
&= -q^{-1/2}(1 + q + q^3 + q^4 + q^5 + 2 q^7 + q^8 + 2 q^9 + q^{10} + 2 q^{11} + q^{12} + 3 q^{13} 
+ \dots) \notag
\end{align}

\begin{remark}
Equation (7.21) in \cite{CCFGH} had a factor of $q^{-1/168}$ instead of $q^{-1/2}$. This is related to the modularity properties of the function, but $q^{-1/2}$ is the correct factor to ensure that $\Zhat_0(q)$ converges to the WRT invariants. 
\end{remark}

Conjecture~\ref{conj:Dehn} says that we can compute $\Zhat_a(S^3_{p/r}({\bf 4_1}))$ for a whole range of surgeries, using the calculation of $F_{\bf 4_1}(x,q)$ in Section~\ref{sec:f8} and the formula
\be
\label{eq:zyp}
\Zhat_a(Y_{p/r})= \eps q^d \cdot \CL_{p/r}^{(a)} \left[ (x^{\frac{1}{2r}} - x^{-\frac{1}{2r}}) F_K (x,q) \right].
\ee
In general, suppose that for a series $F_K(x,q)$, the lowest powers of $q$ in the coefficients $f_m(q)$ of $x^{m/2}$ have exponents of the order of $cm^2$, for some $c \in \R$. Then, after we apply the Laplace transform, the lowest powers of $q$ have exponents of the order
$$ \left( \frac{m}{2} \pm \frac{1}{2r} \right)^2 \cdot \frac{r}{p} + c{m^2}.$$
In order to get Laurent power series in $q$, we need to have a lower bound on these values. This can be guaranteed by asking that
\be
\label{eq:crp}
 4c + \frac{r}{p} > 0.
 \ee
The  values of $p/r$ that satisfy \eqref{eq:crp} are the range of applicability for the surgery formula \eqref{eq:zyp}. For example, for the trefoils $\RHT$ and $\LHT$ we had $c = 1/24$ and $c=-1/24$, respectively. For the figure-eight knot, by calculating more terms in \eqref{eq:Xixq}, we find experimentally that $c=-1/16$, which means that we should be able to apply \eqref{eq:zyp} for
$$ \frac{p}{r} \in (-4, 0).$$
 
In particular, we recover the answer for the $-1$ surgery in \eqref{eq:-237}. We can also compute the invariants $\Zhat_a(q)$ for some hyperbolic manifolds, for example for the $-1/r$ surgeries on ${\bf 4_1}$, for $r > 1$. Note that Conjecture~\ref{conj:Dehn} does not specify the values of $\eps$ and $d$ in \eqref{eq:zyp}. However, by analogy with what happens for torus knots, in the case of $-1/r$ surgeries we take them to be
$$\eps = 1,  \ \ \ d = \alpha(1,r) = -\frac{r}{4} - \frac{1}{4r}.$$ 
 The results are given in Table~\ref{tab:figeightsmallsurgeries}.

\begin{table}
	\centering
	\begin{tabular}{cc}
		\hline\hline
		$Y = S^3_{-1/r} ({\bf 4_1})$ & $\phantom{\oint_{\oint_{\oint}}^{\oint^{\oint}}} \Zhat_a(q)$
		\\
		\hline\hline
		\rule{0pt}{3ex}
		$r=2$ & $-q^{-1/2}(1-q+2 q^3-2 q^6+q^9+3 q^{10}+q^{11}-q^{14}-3 q^{15}-q^{16}+2 q^{19}$ \\
		&  $+2 q^{20}+5 q^{21}+2 q^{22}+2 q^{23}-2 q^{26}-2 q^{27}-5 q^{28}-2 q^{29}-2 q^{30}+\cdots)$ \\[2ex]
		$r=3$ & $-q^{-1/2}( 1-q+2 q^5-2 q^8+q^{15}+3 q^{16}+q^{17}-q^{20}-3 q^{21}-q^{22}+2 q^{31}$ \\
		&  $+2 q^{32}+5 q^{33}+2 q^{34}+2 q^{35}-2 q^{38}-2 q^{39}-5 q^{40}-2 q^{41}-2 q^{42}+ \cdots)$ \\[2ex]
		$r=4$ & $ -q^{-1/2}(1-q+2 q^7-2 q^{10}+q^{21}+3 q^{22}+q^{23}-q^{26}-3 q^{27}-q^{28}+2 q^{43}$ \\
		&  $+2 q^{44}+5 q^{45}+2 q^{46}+2 q^{47}-2 q^{50}-2 q^{51}-5 q^{52}-2 q^{53}-2 q^{54}+ \cdots)$ \\[2ex]
		$r=5$ & $-q^{-1/2}(1-q+2 q^9-2 q^{12}+q^{27}+3 q^{28}+q^{29}-q^{32}-3 q^{33}-q^{34}+2 q^{55}$ \\
		&  $+2 q^{56}+5 q^{57}+2 q^{58}+2 q^{59}-2 q^{62}-2 q^{63}-5 q^{64}-2 q^{65}-2 q^{66}+ \cdots)$ \\[2ex]
		$r=6$ & $-q^{-1/2}(1-q+2 q^{11}-2 q^{14}+q^{33}+3 q^{34}+q^{35}-q^{38}-3 q^{39}-q^{40}+2 q^{67}$ \\
		&  $+2 q^{68}+5 q^{69}+2 q^{70}+2 q^{71}-2 q^{74}-2 q^{75}-5 q^{76}-2 q^{77}-2 q^{78}+q^{112}+ \cdots)$ \\[2ex]
		$r=7$ & $-q^{-1/2}( 1-q+2 q^{13}-2 q^{16}+q^{39}+3 q^{40}+q^{41}-q^{44}-3 q^{45}-q^{46}+2 q^{79}$ \\
		&  $+2 q^{80}+5 q^{81}+2 q^{82}+2 q^{83}-2 q^{86}-2 q^{87}-5 q^{88}-2 q^{89}-2 q^{90}+\cdots )$ \\[2ex]
		$r=8$ & $ -q^{-1/2}(1-q+2 q^{15}-2 q^{18}+q^{45}+3 q^{46}+q^{47}-q^{50}-3 q^{51}-q^{52}+2 q^{91}$ \\
		&  $+2 q^{92}+5 q^{93}+2 q^{94}+2 q^{95}-2 q^{98}-2 q^{99}-5 q^{100}-2 q^{101}-2 q^{102}+ \cdots)$ \\[2ex]
		$r=9$ & $ -q^{-1/2}(1-q+2 q^{17}-2 q^{20}+q^{51}+3 q^{52}+q^{53}-q^{56}-3 q^{57}-q^{58}+2 q^{103}$ \\
		&  $+2 q^{104}+5 q^{105}+2 q^{106}+2 q^{107}-2 q^{110}-2 q^{111}-5 q^{112}-2 q^{113}-2 q^{114}+ \cdots)$ \\[2ex]
		$r=10$ & $ -q^{-1/2}(1-q+2 q^{19}-2 q^{22}+q^{57}+3 q^{58}+q^{59}-q^{62}-3 q^{63}$ \\
		&  $-q^{64}+2 q^{115}+2 q^{116}+5 q^{117}+2 q^{118}+2 q^{119}-2 q^{122}-2 q^{123}-5 q^{124}+ \cdots)$ \\[2ex]
		\hline\hline
	\end{tabular}
	\caption{The invariants $\Zhat_0 (q)$ for some hyperbolic $-1/r$ surgeries on the figure-eight knot.}
	\label{tab:figeightsmallsurgeries}
\end{table}

\newpage \section{Comments on physics and categorification}
\label{sec:categorify}
As mentioned in the Introduction and in Remark~\ref{rem:hblock}, we expect the series $\Zhat_a(q)$ to admit a categorification, i.e., a homology theory $\mathcal{H}^{*,*}_{\text{BPS}}(Y)$ whose (graded) Euler characteristic gives $\Zhat_a(q)$. Relatively little is known about this categorification. We will discuss here  some clues in this direction, their relation to physics and to the work in the current paper.

\subsection{Plumbed $3$-manifolds} 
Since we have the formula \eqref{eq:plumbing1} for $\Zhat_a(q)$ for (weakly) negative definite plumbed manifolds, it is natural to ask for a similar formula for $\mathcal{H}^{*,*}_{\text{BPS}}(Y)$ in that case. In fact, in this paper we gave an equivalent version of \eqref{eq:plumbing1}, namely Equation~\eqref{eq:newformula}, in terms of contributions from vertices and edges, according to the rules \eqref{vertexrulenew} and \eqref{edgerulenew}. This formula may be easier to categorify. Indeed, the key property is that the contributions are ``local'': for every building block (vertex or edge), its contribution depends only on that block and its nearest neighbors, not on the rest of the construction.

Physically, the origin of this simple but important property has to do with the fact that each basic building block of $Y$ corresponds to a particular building block of the corresponding 3d $\cN=2$ theory $T[Y]$.
This has an implication to the categorification of the $q$-series invariants $\Zhat_a (q)$.
In the context of 3d-3d correspondence or, equivalently, in the fivebrane system
\begin{multline}
\Zhat_a(Y; q)
 =
\boxed{~{\text{6d $\cN = (0,2)$ theory}~ \atop \text{on}~D^2 \times_q S^1 \times  Y }~}
= 
\boxed{~{\text{3d $\cN=2$ theory}~ T[Y] \atop \text{on}~D^2 \times_q S^1}~}
 = 
\raisebox{-0.6in}{\includegraphics[width=1.8in]{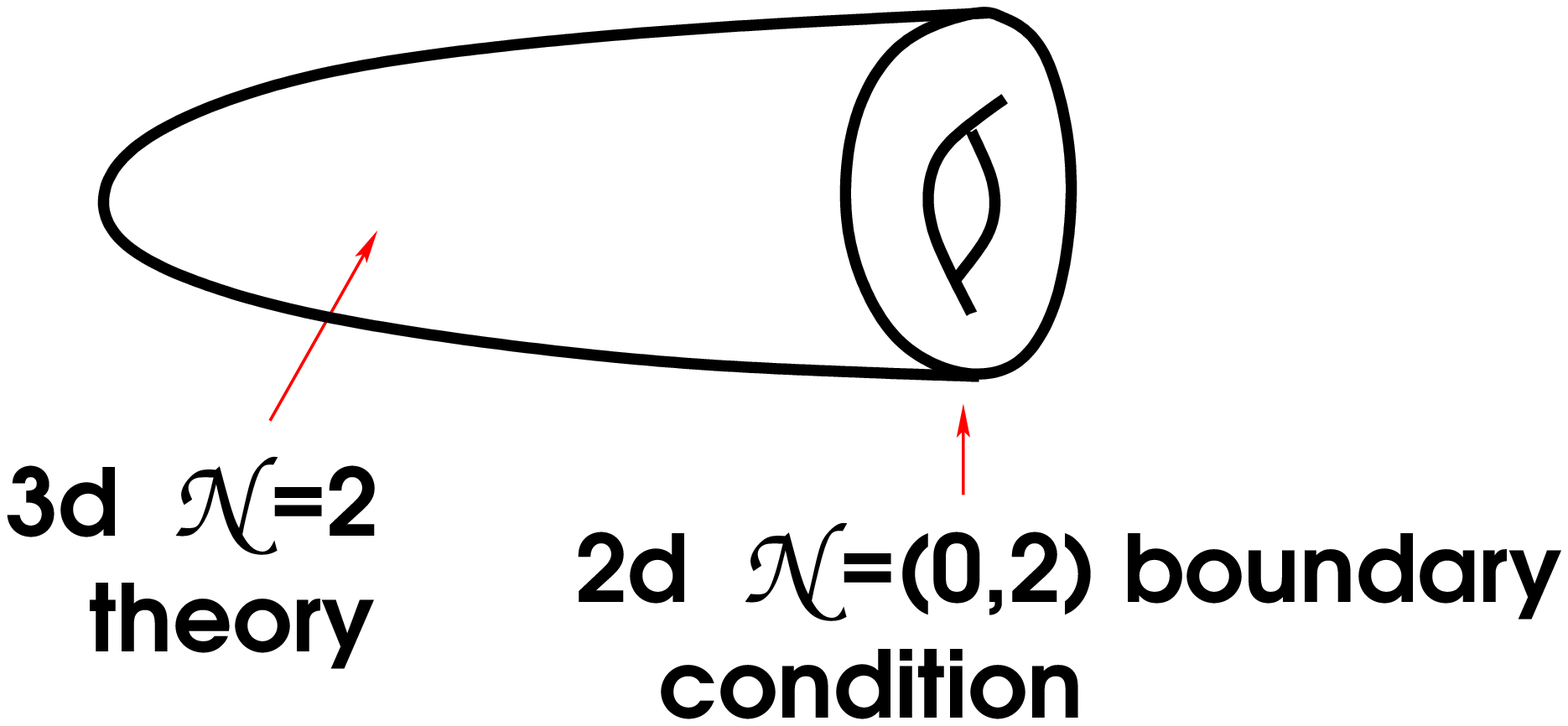}}
\label{6dsetup}
\end{multline}
categorification is achieved by passing from the BPS $q$-series to the space of BPS states $\cH_{\text{BPS}}^{*, *}$. 
At the level of the Poincar\'e polynomial, this corresponds to turning on a fugacity $t$ (sometimes denoted $y$)
that keeps track of a $U(1)_{\beta}$ symmetry that ``locks'' with the R-symmetry $U(1)_R$
(see \cite{Gukov:2004hz, AwataVolume} for more details).

Important for us here is that turning on the fugacity $t$ does not spoil the local nature of the gluing rules.
Therefore, we only need to know how $t$ enters  \eqref{vertexrulenew}--\eqref{edgerulenew} .

The categorification / $t$-deformation of the vertex factor is relatively simple and has already appeared in \cite{GPPV}.
It follows from the fact that $T[Y, G]$ is a quiver-like gauge theory, such that each vertex contributes
a gauge symmetry $G$ with $\cN=2$ supersymmetric Chern-Simons coupling at level $a$ and a chiral multiplet in
the adjoint representation:
\be
\text{vertex}
\qquad \overset{\displaystyle{a}}{\bullet}
\quad = \quad
\boxed{\begin{array}{c}
~\text{3d } \cN=2 \text{ super-Chern-Simons} \\
~\text{with } G_a \text{ and an adjoint chiral}
\end{array}}
\label{Tvertex}
\ee
In this paper, we are only interested in the case $G = SU(2)$.
Note that \eqref{vertexrulenew} indeed equals a 2d-3d half-index of this simple theory
with Neumann boundary conditions \cite{GPPV}.
Its ``refinement'' with a $t$-variable turned on looks like
\be
\text{vertex}
\qquad \overset{\displaystyle{a}}{\bullet}
\quad = \quad- q^{-m_vn_v^2-\frac{m_v}{4}-a_vn_v} z_v^{2m_v n_v+a_v}
\frac{\overbrace{(q;q)_{\infty} (z_v^2;q)_{\infty} (z_v^{-2};q)_{\infty}}^{SU(2) \text{ gauge}}}{\underbrace{(-qt;q)_{\infty} (-qtz_v^2;q)_{\infty} (-qtz_v^{-2};q)_{\infty}}_{\text{adjoint chiral}}}
\label{vertextdeformed}
\ee
Specializing to $t=-1$ returns \eqref{vertexrulenew}, as it should.

An edge connecting two vertices decorated with Euler numbers $a$ and $b$ contributes to $T[Y,G]$
interactions between 3d $\cN=2$ super-Chern-Simons theories $G_a$ and $G_b$. We do not have a formula for this yet.

\subsection{Relation to log-VOAs}
In Remark~\ref{rem:unred} we noted that there should also be an ``unreduced'' version of $\Zhat_a(Y; q)$, denoted $\Zhat^{\unred}_a(Y; q)$. Among other things, the physical setup \eqref{6dsetup} suggests that, for closed 3-manifolds, the $q$-series invariants $\Zhat^{\unred}_a(Y;q)$ should be related to characters of 2d chiral algebras (non-strongly-finite for hyperbolic $Y$); compare \cite{GPV} and \cite[Section 5]{CCFGH}. 
Recall that a character of a VOA module $M$ is defined as
$$
\chi [M] \; = \; \text{Tr}_M \, q^{L_0 - \frac{c}{24}}
$$
where $L_0$ is the conformal vector and $c$ is the central charge.
Therefore, it is natural to expect that a categorification of $\Zhat^{\unred}_a(Y;q) = \chi [M_a]$
is given by $M_a$ (or, rather, by its Felder resolution \cite{Felder}).

For example, in the case of Brieskorn spheres $Y=\Sigma(b_1, b_2, b_3)$
the corresponding algebra was found in \cite{CCFGH}
to be the so-called logarithmic $(1,p)$ singlet VOA \cite{Flohr:1995ea,Adamovic:2007qs}
with $p = b_1 b_2 b_3$ and central charge $c = 13 - 6 (p + p^{-1})$.
Namely, $\Zhat^{\unred}_0 (q)$ is a character of the atypical module
$M_{1,\alpha_1} \oplus M_{1,\alpha_2} \oplus M_{1,\alpha_3} \oplus M_{1,\alpha_4}$
where $\alpha_i$, $i=1, \ldots, 4$, are as in Proposition~\ref{prop:brieskorn},
and each $M_{1,\alpha}$ admits a Felder resolution
in terms of the standard Fock modules $\mathcal{F}_{\lambda}$ (a.k.a. Feigin-Fuchs modules),
see {\it e.g.} \cite{Ridout:2013pwa}.
In particular,
$$
\chi [M_{1,\alpha}] \; = \; \sum_{n \ge 0}
\left( \chi [\mathcal{F}_{\lambda_{-2n,p-\alpha}}] - \chi [\mathcal{F}_{\lambda_{-2n-1,\alpha}}] \right)
$$
where $\lambda_{r,s} = - \frac{r-1}{2} \sqrt{2p} + \frac{s-1}{\sqrt{2p}}$.

It would be interesting to identify log-VOAs that correspond to other types of 3-manifolds, such as  the hyperbolic surgeries on the figure-eight knot for which we computed $\Zhat_a(q)$ in Section~\ref{sec:surgeries8}.

\subsection{Knot complements}
The series $F_K(x,q)$ is the analogue of $\Zhat_a(Y, q)$ for knot complements. It has the same physical interpretation \eqref{6dsetup}, with $Y = S^3 \setminus \nu K$. We also expect it to admit a categorification:  this could be a triply-graded homology theory (whose Poincar\'e polynomial is in the variables $x$, $q$, and $t$), or perhaps a more complicated algebraic object. Furthermore:
\begin{itemize}
\item If we could find the categorification of the edge contribution for plumbing graphs, then the same formula as in the closed case would give a categorification of $F_K(x,q)$ for plumbed knot complements, in the spirit of \eqref{eq:newformula2};
\item The series $F_K(x,q)$ should give characters of some VOA modules, and the categorification of $F_K(x,q)$ should be related to Felder resolutions;
\item At least for simple knots, the categorification of $F_K(x,q)$ should produce a Poincar\'e polynomial that satisfies a recurrence given by a categorified $\hA$ operator, as in \cite{AwataVolume}.
\end{itemize}

The physical formulation \eqref{6dsetup} of the two-variable series $F_{K} (x,q)$ for knot complements is useful for producing concrete computations and, perhaps more importantly, for understanding some of its general properties. However, the information also goes in the opposite direction and many results of the present paper shed light on
various aspects of the ``knot complement theory'' $T[S^3 \setminus \nu K]$.
Until now, a complete formulation of such theory remained elusive for most knots,
even though it was clear since the early days of 3d-3d correspondence
that its vacua should correspond to all $\sl$ flat connections on $S^3 \setminus \nu K$, abelian and non-abelian.
Another longstanding open problem in 3d-3d correspondence was the ``gluing'' of knot complement
theories that corresponds to gluing along torus boundaries.
This is precisely the main subject of the present paper and, in particular, Theorem~\ref{thm:glue}
can be viewed as an answer to this question at the level of the half-index $\Zhat_a(Y;q)$.

We also learned here that 2d $\cN=(0,2)$ boundary conditions in \eqref{6dsetup} should be labeled by $\Spinc$ structures
and that in the limit $q \to 1$ the partition function \eqref{6dsetup} should equal the symmetric expansion of  $1/\Delta_K(x)$. We expect that the $q \to 1$ limit admits an interpretation as a topologically twisted index
of the knot complement theory $T[S^3 \setminus \nu K]$ on $D^2 \times S^1$, similar to \cite{GPV}.

\subsection{Mathematical ingredients}
In Sections~\ref{sec:resurgence} and \ref{sec:recursion}, we noted that $F_K(x,q)$ is obtained from the colored Jones polynomials of $K$ by resurgence and/or recurrence relations. Hence, we expect that one of the ingredients in the categorification of $F_K(x,q)$ should be the categorification of the colored Jones polynomials. In fact, for the un-normalized colored Jones polynomial $\tJ_n(q)$, there are at least two different homology theories in the literature:
\begin{itemize}
\item One due to Khovanov \cite{KhovanovColored}, and extended by Beliakova and Wehrli using Bar-Natan style techniques \cite{BeliakovaWehrli}. Khovanov's colored homology is finite dimensional, for every knot;

\item Another one which is infinite dimensional, even for the unknot. This theory was constructed independently by Cooper-Krushkal \cite{CooperKrushkal}, Frenkel-Stroppel-Sussan \cite{FSS}, Rozansky \cite{RozanskyJW}, and Webster \cite{Webster, Webster2}. 
\end{itemize}

The two theories are expected to have the same reduced version (except for the grading). They also agree for $n=2$, giving Khovanov homology. 

To categorify $F_K(x,q)$, one should also develop ways to extract information from the categorifications of $\tJ_n(q)$ for all $n$. As far as we know, the only mathematical work in this direction is that of Rozansky \cite{Rozansky}, who categorified the tail $\Phi_0$ of the colored Jones polynomials for adequate knots. Recall that $\Phi_0$ is the first term in the stability series $\Phi(x,q)$ discussed in Section~\ref{sec:stability}. By Theorems~\ref{thm:TorusKnots} and \ref{thm:EvenStability}, the stability series for negative torus knots is closely related to its series $F_K(x,q)$. Thus, a natural (and perhaps tractable) problem would be to categorify $F_K(x,q)$ for negative torus knots.

\newpage
\bibliographystyle{custom}
\bibliography{biblio}

\begin{thebibliography}{10}
\providecommand{\url}[1]{\texttt{#1}}
\providecommand{\urlprefix}{URL }
\expandafter\ifx\csname urlstyle\endcsname\relax
  \providecommand{\doi}[1]{doi:\discretionary{}{}{}#1}\else
  \providecommand{\doi}{doi:\discretionary{}{}{}\begingroup
  \urlstyle{rm}\Url}\fi
\providecommand{\eprint}[2][]{\url{#2}}

\bibitem{Adamovic:2007qs}
D.~Adamovi\'{c} and A.~Milas, \emph{Logarithmic intertwining operators and {$
  W(2,2p-1)$} algebras}, J. Math. Phys., \textbf{48}(2007), no.~7, 073503, 20.

\bibitem{Ahmed:2017lhl}
A.~Ahmed and G.~V. Dunne, \emph{Transmutation of a trans-series: the
  {G}ross-{W}itten-{W}adia phase transition}, J. High Energy Phys., (2017),
  no.~11, 054, front matter+51.

\bibitem{Ahmed:2018gbt}
A.~Ahmed and G.~V. Dunne, \emph{Non-perturbative large {$N$} trans-series for
  the {G}ross-{W}itten-€{W}adia beta function}, Phys. Lett.,
  \textbf{B785}(2018), 342--346.

\bibitem{AndersenKashaev}
J.~E. Andersen and R.~Kashaev, \emph{A {TQFT} from {Q}uantum {T}eichm\"{u}ller
  theory}, Comm. Math. Phys., \textbf{330}(2014), no.~3, 887--934.

\bibitem{ArmondDasbach}
C.~Armond and O.~Dasbach, \emph{{R}ogers-{R}amanujan type identities and the
  head and tail of the colored {J}ones polynomial}, preprint,
  \eprint{arXiv:1106.3948}.

\bibitem{BarNatan}
D.~Bar-Natan, \emph{On the {V}assiliev knot invariants}, Topology,
  \textbf{34}(1995), no.~2, 423--472.

\bibitem{BNG}
D.~Bar-Natan and S.~Garoufalidis, \emph{On the {M}elvin-{M}orton-{R}ozansky
  conjecture}, Invent. Math., \textbf{125}(1996), no.~1, 103--133.

\bibitem{Barkan}
P.~Barkan, \emph{Sur les sommes de {D}edekind et les fractions continues
  finies}, C. R. Acad. Sci. Paris S\'{e}r. A-B, \textbf{284}(1977), no.~16,
  A923--A926.

\bibitem{BeliakovaBlanchetLe}
A.~Beliakova, C.~Blanchet, and T.~L\^{e}, \emph{Laplace transform and universal
  sl(2) invariants}, preprint, \eprint{arXiv:math/0509394}.

\bibitem{BeliakovaBuhlerLe}
A.~Beliakova, I.~B\"{u}hler, and T.~L\^{e}, \emph{A unified quantum {${\rm
  SO}(3)$} invariant for rational homology 3-spheres}, Invent. Math.,
  \textbf{185}(2011), no.~1, 121--174.

\bibitem{BeliakovaChenLe}
A.~Beliakova, Q.~Chen, and T.~T.~Q. L\^{e}, \emph{On the integrality of the
  {W}itten-{R}eshetikhin-{T}uraev 3-manifold invariants}, Quantum Topol.,
  \textbf{5}(2014), no.~1, 99--141.

\bibitem{BeliakovaLe}
A.~Beliakova and T.~T.~Q. L\^{e}, \emph{Integrality of quantum 3-manifold
  invariants and a rational surgery formula}, Compos. Math.,
  \textbf{143}(2007), no.~6, 1593--1612.

\bibitem{BeliakovaWehrli}
A.~Beliakova and S.~Wehrli, \emph{Categorification of the colored {J}ones
  polynomial and {R}asmussen invariant of links}, Canad. J. Math.,
  \textbf{60}(2008), no.~6, 1240--1266.

\bibitem{BirmanLin}
J.~S. Birman and X.-S. Lin, \emph{Knot polynomials and {V}assiliev's
  invariants}, Invent. Math., \textbf{111}(1993), no.~2, 225--270.

\bibitem{Boyle}
K.~Boyle, \emph{On the virtual cosmetic surgery conjecture}, New York J. Math.,
  \textbf{24}(2018), 870--896.

\bibitem{CCFGH}
M.~C.~N. Cheng, S.~Chun, F.~Ferrari, S.~Gukov, and S.~M. Harrison, \emph{3d
  {M}odularity}, preprint, \eprint{arXiv:1809.10148}.

\bibitem{CGPS}
S.~Chun, S.~Gukov, S.~Park, and N.~Sopenko, \emph{3d-3d correspondence for
  mapping tori}, preprint, \eprint{arXiv:1911.08456}.

\bibitem{Chung}
H.-J. Chung, \emph{{BPS} invariants for {S}eifert manifolds}, preprint,
  \eprint{arXiv:1811.08863}.

\bibitem{CooperKrushkal}
B.~Cooper and V.~Krushkal, \emph{Categorification of the {J}ones-{W}enzl
  projectors}, Quantum Topol., \textbf{3}(2012), no.~2, 139--180.

\bibitem{Apolynomial}
D.~Cooper, M.~Culler, H.~Gillet, D.~D. Long, and P.~B. Shalen, \emph{Plane
  curves associated to character varieties of {$3$}-manifolds}, Invent. Math.,
  \textbf{118}(1994), no.~1, 47--84.

\bibitem{Costin}
O.~Costin, \emph{Asymptotics and {B}orel summability}, volume 141 of
  \emph{Chapman \& Hall/CRC Monographs and Surveys in Pure and Applied
  Mathematics}, CRC Press, Boca Raton, FL, 2009.

\bibitem{DasbachLin}
O.~T. Dasbach and X.-S. Lin, \emph{On the head and the tail of the colored
  {J}ones polynomial}, Compos. Math., \textbf{142}(2006), no.~5, 1332--1342.

\bibitem{Dimofte:2010ep}
T.~Dimofte and S.~Gukov, \emph{Quantum field theory and the volume conjecture},
  in \emph{Interactions between hyperbolic geometry, quantum topology and
  number theory}, Amer. Math. Soc., Providence, RI, volume 541 of
  \emph{Contemp. Math.}, pp. 41--67, 2011.

\bibitem{Dimofte:2009yn}
T.~Dimofte, S.~Gukov, J.~Lenells, and D.~Zagier, \emph{Exact results for
  perturbative {C}hern-{S}imons theory with complex gauge group}, Commun.
  Number Theory Phys., \textbf{3}(2009), no.~2, 363--443.

\bibitem{Dorigoni}
D.~Dorigoni, \emph{An introduction to resurgence, trans-series and alien
  calculus}, Ann. Physics, \textbf{409}(2019), 167914, 38.

\bibitem{Dunne:2016qix}
G.~V. Dunne and M.~\"{U}nsal, \emph{W{KB} and resurgence in the {M}athieu
  equation}, in \emph{Resurgence, physics and numbers}, Ed. Norm., Pisa,
  volume~20 of \emph{CRM Series}, pp. 249--298, 2017.

\bibitem{EliasQi}
B.~Elias and Y.~Qi, \emph{A categorification of quantum {$\mathfrak{sl}(2)$} at
  prime roots of unity}, Adv. Math., \textbf{299}(2016), 863--930.

\bibitem{Felder}
G.~Felder, \emph{{BRST Approach to Minimal Models}}, Nucl. Phys.,
  \textbf{B317}(1989), 215.

\bibitem{Flohr:1995ea}
M.~A.~I. Flohr, \emph{{On modular invariant partition functions of conformal
  field theories with logarithmic operators}}, Int. J. Mod. Phys.,
  \textbf{A11}(1996), 4147--4172.

\bibitem{FSS}
I.~Frenkel, C.~Stroppel, and J.~Sussan, \emph{Categorifying fractional {E}uler
  characteristics, {J}ones-{W}enzl projectors and {$3j$}-symbols}, Quantum
  Topol., \textbf{3}(2012), no.~2, 181--253.

\bibitem{AwataVolume}
H.~Fuji, S.~Gukov, P.~Su{\l}kowski, and H.~Awata, \emph{Volume conjecture:
  refined and categorified}, Adv. Theor. Math. Phys., \textbf{16}(2012), no.~6,
  1669--1777.

\bibitem{Garoufalidis}
S.~Garoufalidis, \emph{On the characteristic and deformation varieties of a
  knot}, in \emph{Proceedings of the {C}asson {F}est}, Geom. Topol. Publ.,
  Coventry, volume~7 of \emph{Geom. Topol. Monogr.}, pp. 291--309, 2004.

\bibitem{GaroufalidisLeHolonomic}
S.~Garoufalidis and T.~T.~Q. L\^{e}, \emph{The colored {J}ones function is
  {$q$}-holonomic}, Geom. Topol., \textbf{9}(2005), 1253--1293.

\bibitem{GaroufalidisLe}
S.~Garoufalidis and T.~T.~Q. L\^{e}, \emph{Nahm sums, stability and the colored
  {J}ones polynomial}, Res. Math. Sci., \textbf{2}(2015), Art. 1, 55.

\bibitem{Gukov:2003na}
S.~Gukov, \emph{Three-dimensional quantum gravity, {C}hern-{S}imons theory, and
  the {A}-polynomial}, Comm. Math. Phys., \textbf{255}(2005), no.~3, 577--627.

\bibitem{GMP}
S.~Gukov, M.~Mari{\~n}o, and P.~Putrov, \emph{Resurgence in complex
  {C}hern-{S}imons theory}, preprint, \eprint{arXiv:1605.07615}.

\bibitem{GPPV}
S.~Gukov, D.~Pei, P.~Putrov, and C.~Vafa, \emph{B{PS} spectra and 3-manifold
  invariants}, J. Knot Theory Ramifications, \textbf{29}(2020), no.~2, 2040003,
  85.

\bibitem{GPV}
S.~Gukov, P.~Putrov, and C.~Vafa, \emph{Fivebranes and 3-manifold homology}, J.
  High Energy Phys., (2017), no.~7, 071, front matter+80.

\bibitem{Gukov:2012jx}
S.~Gukov and I.~Saberi, \emph{Lectures on knot homology and quantum curves}, in
  \emph{New ideas in low dimensional topology}, World Sci. Publ., Hackensack,
  NJ, volume~56 of \emph{Ser. Knots Everything}, pp. 105--160, 2015.

\bibitem{Gukov:2004hz}
S.~Gukov, A.~Schwarz, and C.~Vafa, \emph{Khovanov-{R}ozansky homology and
  topological strings}, Lett. Math. Phys., \textbf{74}(2005), no.~1, 53--74.

\bibitem{Gukov:2011qp}
S.~Gukov and P.~Su{\l}kowski, \emph{A-polynomial, {B}-model, and quantization},
  J. High Energy Phys., (2012), no.~2, 070, front matter+56.

\bibitem{HabiroKnots}
K.~Habiro, \emph{On the quantum {$sl_2$} invariants of knots and integral
  homology spheres}, in \emph{Invariants of knots and 3-manifolds ({K}yoto,
  2001)}, Geom. Topol. Publ., Coventry, volume~4 of \emph{Geom. Topol.
  Monogr.}, pp. 55--68, 2002.

\bibitem{Habiro}
K.~Habiro, \emph{A unified {W}itten-{R}eshetikhin-{T}uraev invariant for
  integral homology spheres}, Invent. Math., \textbf{171}(2008), no.~1, 1--81.

\bibitem{HikamiDifference}
K.~Hikami, \emph{Difference equation of the colored {J}ones polynomial for
  torus knot}, Internat. J. Math., \textbf{15}(2004), no.~9, 959--965.

\bibitem{Hikami}
K.~Hikami, \emph{On the quantum invariants for the spherical {S}eifert
  manifolds}, Comm. Math. Phys., \textbf{268}(2006), no.~2, 285--319.

\bibitem{HikamiDecomposition}
K.~Hikami, \emph{Decomposition of {W}itten-{R}eshetikhin-{T}uraev invariant:
  linking pairing and modular forms}, in \emph{Chern-{S}imons gauge theory: 20
  years after}, Amer. Math. Soc., Providence, RI, volume~50 of \emph{AMS/IP
  Stud. Adv. Math.}, pp. 131--151, 2011.

\bibitem{HuynhLe}
V.~Huynh and T.~T.~K. Le, \emph{The colored {J}ones polynomial and the
  {K}ashaev invariant}, Fundam. Prikl. Mat., \textbf{11}(2005), no.~5, 57--78.

\bibitem{Khovanov}
M.~Khovanov, \emph{A categorification of the {J}ones polynomial}, Duke Math.
  J., \textbf{101}(2000), no.~3, 359--426.

\bibitem{KhovanovColored}
M.~Khovanov, \emph{Categorifications of the colored {J}ones polynomial}, J.
  Knot Theory Ramifications, \textbf{14}(2005), no.~1, 111--130.

\bibitem{Hopfological}
M.~Khovanov, \emph{Hopfological algebra and categorification at a root of
  unity: the first steps}, J. Knot Theory Ramifications, \textbf{25}(2016),
  no.~3, 1640006, 26.

\bibitem{KMOS}
P.~Kronheimer, T.~Mrowka, P.~Ozsv\'{a}th, and Z.~Szab\'{o}, \emph{Monopoles and
  lens space surgeries}, Ann. of Math. (2), \textbf{165}(2007), no.~2,
  457--546.

\bibitem{KMUnknot}
P.~B. Kronheimer and T.~S. Mrowka, \emph{Khovanov homology is an
  unknot-detector}, Publ. Math. Inst. Hautes \'Etudes Sci., (2011), no. 113,
  97--208.

\bibitem{LawrenceZagier}
R.~Lawrence and D.~Zagier, \emph{Modular forms and quantum invariants of
  {$3$}-manifolds}, Asian J. Math., \textbf{3}(1999), no.~1, 93--107, {S}ir
  Michael Atiyah: a great mathematician of the twentieth century.

\bibitem{LeTran}
T.~T.~Q. Le and A.~T. Tran, \emph{On the {AJ} conjecture for knots}, Indiana
  Univ. Math. J., \textbf{64}(2015), no.~4, 1103--1151, with an appendix
  written jointly with Vu Q. Huynh.

\bibitem{LeZhang}
T.~T.~Q. L\^{e} and X.~Zhang, \emph{Character varieties, {$A$}-polynomials and
  the {AJ} conjecture}, Algebr. Geom. Topol., \textbf{17}(2017), no.~1,
  157--188.

\bibitem{Lickorish}
W.~B.~R. Lickorish, \emph{An introduction to knot theory}, volume 175 of
  \emph{Graduate Texts in Mathematics}, Springer-Verlag, New York, 1997.

\bibitem{LinkSurg}
C.~Manolescu and P.~S. Ozsv{\'a}th, \emph{Heegaard {F}loer homology and integer
  surgeries on links}, preprint (2010), \url{arXiv:1011.1317}.

\bibitem{Marino}
M.~Mari\~{n}o, \emph{Lectures on non-perturbative effects in large {$N$} gauge
  theories, matrix models and strings}, Fortschr. Phys., \textbf{62}(2014), no.
  5-6, 455--540.

\bibitem{MNOP}
D.~Maulik, N.~Nekrasov, A.~Okounkov, and R.~Pandharipande,
  \emph{Gromov-{W}itten theory and {D}onaldson-{T}homas theory. {I}}, Compos.
  Math., \textbf{142}(2006), no.~5, 1263--1285.

\bibitem{MelvinMorton}
P.~M. Melvin and H.~R. Morton, \emph{The coloured {J}ones function}, Comm.
  Math. Phys., \textbf{169}(1995), no.~3, 501--520.

\bibitem{Morton}
H.~R. Morton, \emph{The coloured {J}ones function and {A}lexander polynomial
  for torus knots}, Math. Proc. Cambridge Philos. Soc., \textbf{117}(1995),
  no.~1, 129--135.

\bibitem{Moser}
L.~Moser, \emph{Elementary surgery along a torus knot}, Pacific J. Math.,
  \textbf{38}(1971), 737--745.

\bibitem{NeumannCalculus}
W.~D. Neumann, \emph{A calculus for plumbing applied to the topology of complex
  surface singularities and degenerating complex curves}, Trans. Amer. Math.
  Soc., \textbf{268}(1981), no.~2, 299--344.

\bibitem{NeumannRaymond}
W.~D. Neumann and F.~Raymond, \emph{Seifert manifolds, plumbing, {$\mu
  $}-invariant and orientation reversing maps}, in \emph{Algebraic and
  geometric topology ({P}roc. {S}ympos., {U}niv. {C}alifornia, {S}anta
  {B}arbara, {C}alif., 1977)}, Springer, Berlin, volume 664 of \emph{Lecture
  Notes in Math.}, pp. 163--196, 1978.

\bibitem{Ng}
L.~Ng, \emph{A {L}egendrian {T}hurston-{B}ennequin bound from {K}hovanov
  homology}, Algebr. Geom. Topol., \textbf{5}(2005), 1637--1653.

\bibitem{OhtsukiIntegral}
T.~Ohtsuki, \emph{A polynomial invariant of integral homology {$3$}-spheres},
  Math. Proc. Cambridge Philos. Soc., \textbf{117}(1995), no.~1, 83--112.

\bibitem{OhtsukiRational}
T.~Ohtsuki, \emph{A polynomial invariant of rational homology {$3$}-spheres},
  Invent. Math., \textbf{123}(1996), no.~2, 241--257.

\bibitem{OSS}
P.~Ozsv\'{a}th, A.~I. Stipsicz, and Z.~Szab\'{o}, \emph{Knots in lattice
  homology}, Comment. Math. Helv., \textbf{89}(2014), no.~4, 783--818.

\bibitem{OSlens}
P.~Ozsv\'{a}th and Z.~Szab\'{o}, \emph{On knot {F}loer homology and lens space
  surgeries}, Topology, \textbf{44}(2005), no.~6, 1281--1300.

\bibitem{AbsGraded}
P.~S. Ozsv{\'a}th and Z.~Szab{\'o}, \emph{Absolutely graded {F}loer homologies
  and intersection forms for four-manifolds with boundary}, Adv. Math.,
  \textbf{173}(2003), no.~2, 179--261.

\bibitem{Plumbed}
P.~S. Ozsv{\'a}th and Z.~Szab{\'o}, \emph{On the {F}loer homology of plumbed
  three-manifolds}, Geometry and Topology, \textbf{7}(2003), 185--224.

\bibitem{IntSurg}
P.~S. Ozsv{\'a}th and Z.~Szab{\'o}, \emph{Knot {F}loer homology and integer
  surgeries}, Algebr. Geom. Topol., \textbf{8}(2008), no.~1, 101--153.

\bibitem{RatSurg}
P.~S. Ozsv{\'a}th and Z.~Szab{\'o}, \emph{Knot {F}loer homology and rational
  surgeries}, Algebr. Geom. Topol., \textbf{11}(2011), no.~1, 1--68.

\bibitem{Pic}
L.~Piccirillo, \emph{The {C}onway knot is not slice}, Ann. of Math. (2),
  \textbf{191}(2020), no.~2, 581--591.

\bibitem{Qi}
Y.~Qi, \emph{Hopfological algebra}, Compos. Math., \textbf{150}(2014), no.~1,
  1--45.

\bibitem{Rasmussen}
J.~Rasmussen, \emph{Khovanov homology and the slice genus}, Invent. Math.,
  \textbf{182}(2010), no.~2, 419--447.

\bibitem{ReshetikhinTuraev}
N.~Reshetikhin and V.~G. Turaev, \emph{Invariants of {$3$}-manifolds via link
  polynomials and quantum groups}, Invent. Math., \textbf{103}(1991), no.~3,
  547--597.

\bibitem{Ridout:2013pwa}
D.~Ridout and S.~Wood, \emph{{Modular Transformations and Verlinde Formulae for
  Logarithmic $(p_+,p_-)$-Models}}, Nucl. Phys., \textbf{B880}(2014), 175--202.

\bibitem{RozanskyContribution}
L.~Rozansky, \emph{A contribution of the trivial connection to the {J}ones
  polynomial and {W}itten's invariant of {$3$}d manifolds. {I}}, Comm. Math.
  Phys., \textbf{175}(1996), no.~2, 275--296.

\bibitem{Rozansky96}
L.~Rozansky, \emph{Higher order terms in the {M}elvin-{M}orton expansion of the
  colored {J}ones polynomial}, Comm. Math. Phys., \textbf{183}(1997), no.~2,
  291--306.

\bibitem{RozanskyMMR}
L.~Rozansky, \emph{The universal {$R$}-matrix, {B}urau representation, and the
  {M}elvin-{M}orton expansion of the colored {J}ones polynomial}, Adv. Math.,
  \textbf{134}(1998), no.~1, 1--31.

\bibitem{RozanskyJW}
L.~Rozansky, \emph{An infinite torus braid yields a categorified
  {J}ones-{W}enzl projector}, Fund. Math., \textbf{225}(2014), no.~1, 305--326.

\bibitem{Rozansky}
L.~Rozansky, \emph{Khovanov homology of a unicolored {B}-adequate link has a
  tail}, Quantum Topol., \textbf{5}(2014), no.~4, 541--579.

\bibitem{SavelievBook}
N.~Saveliev, \emph{Invariants for homology {$3$}-spheres}, volume 140 of
  \emph{Encyclopaedia of Mathematical Sciences}, Springer-Verlag, Berlin,
  {L}ow-Dimensional Topology, I, 2002.

\bibitem{Taubes1}
C.~H. Taubes, \emph{Compactness theorems for {$\operatorname{SL}(2;\mathbb C)$}
  generalizations of the $4$-dimensional anti-self dual equations}, preprint,
  \eprint{arXiv:1307.6447}.

\bibitem{Taubes2}
C.~H. Taubes, \emph{Sequences of {N}ahm pole solutions to the {SU}(2)
  {K}apustin-{W}itten equations}, preprint, \eprint{arXiv:1805.02773}.

\bibitem{Thurston}
W.~Thurston, \emph{The geometry and topology of three-manifolds}, {P}rinceton
  {U}niversity lecture notes, \eprint{http://library.msri.org/books/gt3m}.

\bibitem{Webster2}
B.~Webster, \emph{Tensor product algebras, {G}rassmannians and {K}hovanov
  homology}, in \emph{Physics and mathematics of link homology}, Amer. Math.
  Soc., Providence, RI, volume 680 of \emph{Contemp. Math.}, pp. 23--58, 2016.

\bibitem{Webster}
B.~Webster, \emph{Knot invariants and higher representation theory}, Mem. Amer.
  Math. Soc., \textbf{250}(2017), no. 1191, v+141.

\bibitem{WittenCS}
E.~Witten, \emph{Quantum field theory and the {J}ones polynomial}, Comm. Math.
  Phys., \textbf{121}(1989), no.~3, 351--399.

\bibitem{Witten5B}
E.~Witten, \emph{Fivebranes and knots}, Quantum Topol., \textbf{3}(2012),
  no.~1, 1--137.

\end{thebibliography}

\end{document}